\documentclass[preprint]{imsart}
\RequirePackage[OT1]{fontenc}
\usepackage{mypackage}
\usepackage{mymacros2}
\usepackage{mymacros_3}
\usepackage{mymacros}
\usepackage{myenv, mymacros_brackets}
\doi{10.1214/154957804100000000}
\pubyear{0000}
\volume{0}
\firstpage{1}
\lastpage{1}

\begin{document}

\begin{frontmatter}
% \title{Location estimation for symmetric and log-concave densities}
\title{Adaptive estimation in symmetric location model under log-concavity constraint}
\runtitle{Location estimation  }

\begin{aug}
\author{\fnms{Nilanjana} \snm{Laha}}

\runauthor{N. Laha}

\affiliation{Department of Biostatistics, Harvard University}
\address{Department of Biostatistics, Harvard University}
%  677 Huntington Ave,\\ Boston, MA 02115, U.S.A\\
% \thanksmark{m1}\printead{e1}}
\end{aug}	

\begin{abstract}
 We revisit the problem of  estimating  the center of symmetry $\theta$ of an unknown symmetric density $f$.  Although \cite{stone}, \cite{eden}, and \cite{saks}  constructed adaptive estimators of $\th$ in this model, their estimators depend on  external tuning parameters. In an effort to reduce the burden of  tuning parameters,  we impose  an additional restriction of log-concavity on $f$.  We construct  truncated one-step estimators which are adaptive under the log-concavity assumption. Our simulations suggest that the untruncated version of the one step estimator, which is tuning parameter free, is also asymptotically efficient.
 We also study the maximum likelihood estimator (MLE) of $\theta$  in the shape-restricted model. 
\end{abstract}



\date{\today}
\end{frontmatter}

%\tableofcontents

\section{Introduction}
\label{sec:intro}
\setcounter{page}{3}
\listoftodos
In this paper, we  revisit the symmetric location model with an additional shape-restriction of log-concavity. 
We  let $\mathcal{P}$  denote the class of  all densities  on the real line $\RR$. 
For any $\th\in\RR$, denote by $\mathcal{S}_{\th}$ the class of all  densities  symmetric about $\theta$. Then the symmetric location model $\mathcal P_s$ is given by
\begin{equation}\label{definition: the symmetric model Ps}
\mathcal{P}_s =\bigg\{f\in\mathcal{P}\ \bl\ f(x;\th)=g(x-\th),\ \th\in\RR,\ g\in\mathcal{S}_0,\ \If<\infty\bigg\},
\end{equation}
where $\If$ is the Fisher information for location. 
It is well-established that \citep[][Theorem~3]{huber} $\If$ is finite if and only if $f$ is an absolutely continuous density satisfying
\[\edint\lb\dfrac{f'(x)}{f(x)}\rb^2 f(x)dx<\infty,\]
where $f'$ is an $L_1$-derivative of $f$. Also, in this case, $\If$ takes the form 
\begin{equation*}
\If=\edint\lb\dfrac{f'(x)}{f(x)}\rb^2 f(x)dx.
\end{equation*} 

    Estimation of $\th$ in $\mathcal P_s$
is an old semi-parametric problem, dating back to \cite{stein}. From then on, the problem of estimating $\theta_0$ in $\mathcal P_s$  has been considered by many early authors including, but not limited to, \cite{stone}, \cite{beran}, \cite{saks}, and \cite{eden}. 
   There are two main reasons behind the assumption of symmetry in this model. First, as \citeauthor{stone} has pointed out, if $f$ is totally unrestricted, $\th$ is not identifiable.  Second,
% the  assumption of symmetry is quite natural because 
 the definition of location becomes unclear in the absence of symmetry  \citep{unimodal1}. 
 The appeal of the above  model lies in the fact that
 adaptive estimation of $\th$ is possible in this model \citep{stone}.
In other words,  there exist consistent estimators of $\th$ in $\mP_s$, whose asymptotic variance attains the parametric lower bound, which is $\If^{-1}$ in this case.  See Sections $3.2$, $3.3$, and $6.3$ of \cite{jonsemi} for more discussion on adaptive estimation in $\mathcal{P}_s$. 
% The early authors  construct adaptive estimators of $\theta_0$ in $\mathcal P_s$, sometimes under additional smoothness or structural assumption.

 There are general classes of nonparametric estimators, which, following some clever reconstruction,  lead to adaptive estimators of $\theta$ in $\mathcal P_s$. Examples include the  one step estimator used by \citeauthor{stone}, and the Hodges-Lehmann rank estimator used by \citeauthor{eden}. \citeauthor{beran} uses a linearized rank estimator introduced by \cite{kraft1970}, where \citeauthor{saks} uses a linear functional of order statistics. All these estimators involve various tuning parameters. The success of these type of nonparametric estimators generally depend crucially on the choice of the tuning parameters.  (cf. \cite{saks}; see also   \cite{park} for a thorough empirical study of similar estimators in  a closely related nonparametric problem, i.e. the two-sample location problem.)  However, no data-dependent method has been prescribed to choose these tuning parameters. Therefore, despite attractive theoretical properties, the implementation of the asymptotically efficient estimators of $\theta_0$ is not straightforward. 
%  Furthermore,  \cite{hogg1974} points out that in practice, the nonparametric methods may be too slow and may require high sample size.
 
 Although the  tuning parameters stemming from different nonparametric approaches appear to be different,  they generally fall in one of the following categories: (a) scaling parameter for approximating  derivatives by quotient (e.g. \citeauthor{beran}, \citeauthor{saks} and  \citeauthor{eden}), (b) bandwidth selection parameter if kernels are used (e.g. \citeauthor{stone}), (c)  the number of basis functions  (e.g. \citeauthor{beran}), (d) parameters arising due to truncation (e.g. the  estimators of  \citeauthor{stone} and \citeauthor{saks}) or data-partitioning (e.g., \citeauthor{eden}). 
 We will elaborate a little bit on the first three type of tuning parameters.   They arise solely because the adaptive estimators of $\theta_0$  require estimating
   $g$, $g'$ (e.g. \citeauthor{stone}, \citeauthor{eden}, and \citeauthor{beran}), and in some cases, higher derivatives (e.g. $g''$, \citeauthor{saks}). In fact, such tuning parameters are unavoidable in nonparametric estimation of the above quantities.  Moreover, \cite{hogg1974} points out that in practice,  nonparametric estimation of such functions  may be too slow. 
  This is precisely  where  semi-parametric models can help because the additional structure  can  be exploited to construct  computationally efficient  estimators of $g$ and $g'$ without using tuning parameters.

    If we impose an additional shape restriction of log-concavity on $g$, for instance,  the task of estimating $g'$   becomes much simpler. The reason is, the class of  log-concave  densities is structurally rich enough  to admit a  maximum likelihood estimator (MLE) \citep{exist, 2009rufi}. Similar results hold for its subclasses, e.g. the class of all symmetric (about the origin) log-concave densities  as well \citep{dosssymmetric}. The log-concave MLE type density estimators allow for computationally efficient estimation of the scores without any tuning parameters. These score estimates can readily be used to construct a  one step estimator.

    We show that under the log-concavity assumption, truncated versions of the above-mentioned one step estimator are adaptive provided the truncation parameter $\eta_n\to 0$ slowly enough.  This truncation parameter is our only tuning parameter,  which also is introduced purely due to technical reasons in the proof. Moreover, we empirically show that 
    the efficiency of our estimators monotonously increases as $\eta_n\to 0$. In fact, the  untruncated one step estimator attains the highest efficiency, and also performs reliably under varied settings.
     % We empirically show that  our  one step estimator performs reasonably well even in absence of any truncation. In fact, simulations show that  thus the best efficiency is achieved at the truncation level zero. 
    Thus, for practical implementation, the proposed estimator of this paper is the untruncated one-step estimator, which is fully tuning parameter free.
     We also touch upon another important tuning parameter free estimator of $\theta_0$, namely the MLE. In particular, we establish its existence under the  shape-constrained model. Our methods can be implemented using the R package \texttt{log.location} which can be accessed at \url{https://github.com/nilanjanalaha/log.location}.

 The imposition of log-concavity on $\mathcal P_s$ may seem forced, but is not at all unnatural.  The class of log-concave densities, $\mathcal{LC}$, is an important subclass of the class of unimodal densities.   Many common symmetric unimodal  densities, e.g. Gaussian, logistic, and Laplace,   are log-concave.  Unimodality is a reasonable assumption  in  context of location estimation of symmetric densities.  As \cite{unimodal1} points out, in practice, multimodal densities generally result from unimodal mixtures. Separate procedures are available for the latter class.   The difficulty with  the  unimodality shape restriction, however, stems from the fact that the corresponding  density-class  is still large, especially it is  not structurally rich enough to admit an MLE \citep{birge1997}. Therefore unlike the log-concavity assumption, the unimodality assumption does not provide  computational advantages. Hence, we impose the assumption of  log-concavity  on $\mP_s$ instead of just unimodality.

Finally, this paper is an attempt towards  bridging the gap between  the symmetric location model and log-concavity.
 Although shape-constrained estimation has a rich history,  so far there has been  little to no use of shape-constraints in  one-sample symmetric location problem. 
%For a short list of exceptions, see below.  
In fact, to  the best of our knowledge, \cite{eden} is the only one to incorporate shape-constraints in treating  the problem considered here. Actually \citeauthor{eden} requires $f$ to be log-concave although her paper  does not mention log-concavity.  She requires the function $f'(F^{-1}(u))/f(F^{-1}(u))$ to be non-increasing in $u\in(0,1)$, which is equivalent to  $f$ being log-concave \citep[][Proposition A.1]{bobkov1996}. As made clear by our earlier discussion, \citeauthor{eden} does not use shape-restricted tools  tailored for log-concave densities because they were not available at that time. 
 We also want to mention  \cite{sharmada}, who consider both location  and scale estimation in an  elliptical symmetry model, which albeit bearing some resemblance,  is  different from the model considered in this paper. Also, \cite{sharmada}'s estimation procedure is completely different from ours.  

\subsection{Notation and terminology}
\label{sec: Preliminaries}
For a concave function $\ps:\RR\mapsto\RR$, the domain  $\dom(\ps)$ will be defined as in \citep[][p. 40]{rockafellar}, that is, $\dom(\ps)=\{x\in\RR\ :\ \ps(x)>-\infty\}$. 
  For any concave function $\ps$, we say $x\in\RR$ is a knot of $\ps$, if either $\ps'(x+)\neq\ps'(x-)$, or  $x$ is at the boundary of $\dom(\ps)$. We denote by $\mathcal{K}(\ps)$ the set of the knots of $\ps$. Unless otherwise mentioned, for a real valued function $h$, provided they exist,  $h'$ and $h'(\mathord{\cdot}-)$ will refer to the right and left derivatives of $h$, respectively.   We denote the support of any density $f$ by $\supp(f)=\{x\in\RR\ :\ f(x)>0\}$. 
  
  For a distribution function $F$, we let $J(F)$ denote the set $\{x\ :\ 0<F(x)<1\}$. For a sequence of distribution functions $\{F_n\}_{n\geq 1}$, we say  $F_n$ converges weakly to  $F$, and write $F_n\to_d F$, if  for all bounded continuous functions $h:\RR\mapsto\RR$, we have $\lim\limits_{n\to\infty}\int h dF_n=\int hdF$. 
   For  any real valued function $h:\RR\mapsto\RR$, we let $||h||_k$ denote its $L_k$ norm, i.e.
\[||h||_{k}=\lb\edint |h(x)|^k dx\rb^{1/k},\quad k\geq 1.\]
  For densities $f_1$ and $f_2$, the Hellinger distance $H(f_1,f_2)$ is defined by
\[H^2(f_1,f_2)=\frac{1}{2}\edint (\sqrt f_1(x)-\sqrt f_2(x))^2 dx.\]
We denote the order statistics of a random sample  $(Y_1,\ldots,Y_n)$  by
 $(Y_{(1)},\ldots,Y_{(n)})$.

% We denote the empirical distribution function of the pseudo-observations $X_i-\th_0$'s by $\mathbb{G}_n$, and let $\mathbb{H}_n=\sqn(\mathbb{G}_n-G_0)$ denote the corresponding empirical process. 
% Observe that $U_i=F_0(X_i)\sim U(0,1).$ The empirical distribution of the $U_i$'s will be denoted by $\Gamma_n$ and we will use the notation $\mathbb{V}_n(t)$ for the empirical process of the $U_i$'s. Thus
%  \[\mathbb{V}_n(t)=\sqn(\Gamma_n(t)-t),\quad t\in[0,1].\]
%  %   

 As usual, we denote the set of natural numbers by $\mathbb{N}$. We  denote by $C$  an arbitrary constant which may vary from line to line. The expression $x\lesssim y$ will imply that there exists $C>0$ so that $x\leq Cy$.

%   For the concave function $\ps\in\mathcal{SC}_0$, both $\ps'$ and $\ps'(\mathord{\cdot}-)$ exist. In this paper, we choose to work with $\ps'$ or the version of the derivative which is right continuous. Choosing a different version does not alter  the estimators discussed in this paper except for the untruncated one-step estimator defined in \eqref{def: one-step estimator full}. However, for large samples this difference is negligible. The asymptotic results derived in this paper remain valid with different versions of $\ps'$ as well.

\subsection{Problem set up}
\label{sec: estimators}
To formalize the set up, first, let us define
\begin{equation}\label{def: class of concave functions}
\mathcal{C}:=\bigg\{\ph:\RR\mapsto[-\infty,\infty)\ \bl\ \ph \text{ is concave, closed, and proper} \bigg\}.
\end{equation}
We let $\mathcal{SC}_{\th}=\mathcal{S}_{\th}\cap\mathcal{C}$ denote the class of all closed and proper concave functions symmetric about $\theta\in\RR$. Here  a proper and closed concave function is as defined in \cite{rockafellar}, page 24 and 50.
Letting $\mathcal{LC}$ denote the class of log-concave densities
\[\mathcal{LC}:= \bigg\{f\in\mathcal{P}\ \bl\ \phi=\log f\in\mathcal{C}\bigg\},\]
we set $\mathcal{SLC}_{\th}= \mathcal{LC}\cap\mathcal{S}_{\th}$.
Suppose we observe $n$  independent and identically distributed (i.i.d.) random variables  $X\equiv X_1,\ldots, X_n$ with density $f_0\equiv g_0(\cdot-\theta_0)\in\mP_0$, where
\begin{equation}
\label{definition: model}
\mP_0=\bigg\{f\in\mP \ \bl\ f(x;\th)=g(x-\th),\ \th\in\RR,\ g\in\mathcal{SLC}_0,\ \mathcal I_g<\infty \bigg\}
\end{equation}
is the symmetric log-concave location model. Our aim is to estimate the location parameter $\theta_0$.

 Let us denote $\ph_0=\log f_0$, and $\ps_0=\log g_0$. We let $F_0$ and $G_0$ be the respective distribution functions of $f_0$ and $g_0$, and denote by $P_0$ the measure corresponding to $F_0$. We  denote the empirical distribution function of the $X_i$'s by $\Fm$, and write $\Pm$ for the corresponding empirical measure.

We use the following convention  throughout the paper while setting notations for the one step estimators and the MLE. We use a \texttt{hat} on the quantities related to the MLE, e.g. the MLE of $\th_0$ and $g_0$ will be denoted by $\hthm$ and $\hgm$, respectively. The similar quantities in the one-step estimator context will use a \texttt{tilde}, e.g. $\hth$, $\hn$ etc. Some quantities like $\hf$, the MLE in $\mathcal{LC}$,   or $\widehat{f}_\th$, the MLE in $\mathcal{SLC}_\theta$, will be introduced in context of the one step estimators, but their notations use the \texttt{hat} instead of the \texttt{tilde} because they are  MLEs.

The article is organized as follows. In Section \ref{sec: one-step}, we introduce the one step estimator, and  discuss its asymptotic properties. In Section~\ref{sec: MLE},  we explore the MLE of $\th_0$ in $\mP_0$. We provide an empirical study  in Section \ref{sec: simulation}.  The proofs are deferred to the appendix.

\section{One step estimator}
\label{sec: one-step}
Let \bth\ be a preliminary estimator of $\th_0$. Had $g_0$ been known, a valid estimator of $\theta_0$ would be readily given by the one step estimator \citep[see p. 71-72 and 392-399 of][]{vdv} 
\begin{equation}\label{def: one step: parametric}
    \bth-\edint\dfrac{{\ps}'_0(x-\bth)}{\I}d\Fm(x).
\end{equation}
In fact, the above estimator is $\sqn$ consistent with asymptotic variance $\I^{-1}$ when $g_0$ is known \citep[cf. Theorem 5.45 of][]{vdv}.
  Suppose $\hn\in\mathcal{S}_0$ is an estimator of $g_0$. Further suppose $\hln=\log \hn$ is left and right differentiable on the support of $\hn$. The latter always holds if $\hn\in\mathcal{LC}$ \citep[Theorem 0.6.3, pp.15, ][]{hiriart2004}.
 Suppose $\hln'$ is the right derivative of $\hln$.  Defining $\hln'$ to be zero outside $\supp(\hn)$, we can define an estimator of $\theta_0$ along the lines of \eqref{def: one step: parametric} as follows:
\begin{equation}\label{def: one-step estimator full}
\hthf=\bth-\edint\dfrac{\tilde{\ps}'_n(x-\bth)}{\hin}d\Fm(x),
\end{equation}
where
\begin{equation}\label{def: untruncated fisher info}
\hin=\edint \tilde{\ps}'_n(x-\bth)^2d\Fm(x)
\end{equation}
is an estimator of the Fisher information $\I$. We will refer to $\hthf$ as the untruncated one step estimator.

The asymptotic behavior of $\hln'$ can be hard to control in the tails, which creates technical difficulties in the asymptotic analysis of $\hthf$. As we already mentioned in the introduction, a common approach to tackle this problem is trimming the extreme observations, which leads to a truncated one step estimator similar to \citeauthor{stone}. 
% \cite{piet} also
% % Hendrickx and Groeneboom (2017)
%  use a similar idea in a current status linear regression model where  a  truncated log-likelihood was considered to avoid problems in the tails of the unknown error distribution. 
 
 We let $\eta_n$  denote the truncation parameter, which is usually a small positive fraction.
Denote by $\hH$  the distribution function corresponding to \hn.  Letting $\xin$ be the $(1-\eta_n)$-th quantile of $\hH$, we define the truncated one step estimator as follows: 
 \begin{equation}\label{def: one-step estimator: truncated}
\hth=\bth-\dint_{\bth-\xin}^{\bth+\xin}\dfrac{\hln'(x-\bth)}{\hin(\eta_n)}d\Fm(x).
\end{equation} 
Here $\hin(\eta_n)$ is a truncated version of $\hin$, given by
 \begin{align}\label{27indef}
&\hin(\eta_n)=\dint _{\bth-\xin}^{\bth+\xin}\tilde{\ps}'_n(x-\bth)^2d\Fm(x).
\end{align}
 Note that the symmetry of $\hn$ about $0$ implies that $-\xin=\hH^{-1}(\eta_n)$. Ideally, we should denote the one step estimator in \eqref{def: one-step estimator: truncated} by $\hth(\eta_n)$   but here we suppress the dependence  on $\eta_n$ to avoid cumbersome notation.
 
%  It is important to remember that the truncation was introduced to deal with technical difficulties in proving the asymptotic efficiency.
%  Simulation shows that the untruncated estimator typically performs better than any truncated version, although the difference is insignificant if the truncation level is very small. Because  the untruncated estimator is also tuning parameter-free, it is the preferred estimator for  implementation.

 $\I$ could also be estimated by a smoother version of $\hin(\eta_n)$, namely,
 \begin{align*}
&\hin^{*}(\eta_n)=\dint _{\bth-\xin}^{\bth+\xin}\tilde{\ps}'_n(x-\bth)^2\hn(x-\bth)dx.
\end{align*}
However, our simulations indicate that the estimator $\hin(\eta)$ yields a more efficient one-step estimator. Therefore,  $\hin(\eta)$ is our preferred estimator for the Fisher information.

 \subsection{Main result}
 \label{sec: one step: main result}
 
 The first main result of this paper implies that if $\eta_n\to 0$ at a sufficiently slow rate, then the truncated one step estimator defined in \eqref{def: one-step estimator: truncated} is adaptive for certain choices of $\hn$. However, we require a technical assumption on $\psi_0$ to prove this theorem. 
\begin{assumption}
 \label{assump: L}
 There exists $\kappa>0$ so that  
 \[|\psi_0'(x)-\psi'_0(y)|\leq \kappa|x-y|\quad\text{ for all }x,y\in\iint(\dom(\psi_0)),\]
 where  $\psi_0'(x)$ (or $\psi_0'(y)$) is either the left or right derivative of $\psi_0$ at $x$ (or $y$).
 \end{assumption}
 Since $\psi_0$ is concave, it is left and right  differentiable at every $x\in\iint(\dom(\psi_0))$ \citep[pp. 15][]{hiriart2004}. If $\psi_0$ is twice differentiable on $\iint(\dom(\psi_0))$, Assumption~\ref{assump: L} interprets as $|\psi_0''|\leq \kappa$.   Assumption~\ref{assump: L} is essentially a smoothness condition, which is not uncommon in the context of log-concave density estimation.
  A similar assumption appears in \cite{2009rufi} (see Theorem 4.1 therein),
 who consider $\psi_0$ to be in a H\"{o}lder class with exponent $\beta\in [1,2]$, which coincides with  Assumption~\ref{assump: L}  if $\beta=2$.
The  H\"{o}lder-smoothness assumption is also used in  \cite{doss2019} (see Theorem 2.1 therein), who generalize \cite{2009rufi}'s Theorem 4.1 to the case of  unimodal log-concave densities.
Such  smoothness assumptions can also be found in the literature related to monotonocity constraints \cite{bodhida2017, ramu2018}.
 Simple algebra shows that common symmetric log-concave densities like Gaussian, Laplace, and Logistic  satisfy Assumption~\ref{assump: L}.
 Later in Section~\ref{sec: simulation}, we consider an example where Assumption~\ref{assump: L} is violated. Whether Assumption~\ref{assump: L} is necessary is unknown to us, although Section~\ref{sec: simulation} hints that the truncated one step estimators may still be adaptive even under the  violation of Assumption~\ref{assump: L}.
 
 Now we state the requirements for  $\hn$. Later in this section, we demonstrate how to build estimators which satisfy such conditions.
\begin{condition}
\label{condition: on hn}
Let $y_n=o_p(1)$ be a random sequence.
The density estimator $\hn$ satisfies the following:
\begin{itemize}[noitemsep,topsep=0pt]
\item[(A)] 
$\|\hn-g_0\|_1\to_p 0$ and  $\sup_{x\in \RR}|\hn(x+y_n)-g_0(x)|\to_p 0$.
\item[(B)] For any $K\subset\iint(\dom(\psi_0))$, we have
 $\sup_{x\in K}|\hln(x+y_n)-\ps_0(x)|\to_p 0.$
\item[(C)]
Suppose $x\in\iint(\dom(\ps_0))$ is a continuity point of $\ps_0'.$  Then
\[\hln'(x+y_n)\to_p\ps_0'(x).\]
\end{itemize}
\end{condition}
Condition~\ref{condition: on hn} (A) implies $H(\hn,g_0)\to_p 0$ because $H(\hn,g_0)\lesssim \sqrt{\|\hn-g_0\|_1}$. However,  
we require stronger control over the rate of decay of the Hellinger error $H(\hn,g_0)$.
\begin{condition}
\label{cond: hellinger rate}
 There exists $p\in(0,1/2]$ so that $H(\hn, g_0)=O_p(n^{-p})$.
\end{condition}
The upper bound of $1/2$ on $p$ is natural because even in the parametric case, the Hellinger error rate is generally not faster than $O_p(n^{-1/2})$. Now we are ready to state our main theorem. The proof of Theorem~\ref{theorem: main: one-step: full} can be found in Appendix~\ref{app: proof of Theorem 1}.

%If $\theta_0$ was known, the conjectured minimax rate of nonparametrically estimating $g_0$ in $\mathcal{LC}$ is $n^{-2/5}$
% Of course, not all density estimators $\hn\in\mathcal{S}_0$  satisfies Conditions~\ref{condition: on hn} and \ref{cond: hellinger rate} because the requirements are quite non-trivial.
% Later we will demonstrate how to construct estimators which satisfy Conditions~\ref{condition: on hn} and \ref{cond: hellinger rate}.

\begin{theorem}\label{theorem: main: one-step: full}
Suppose $f_0\in\mP_0$ satisfies Assumption~\ref{assump: L} and $\bth$ is a $\sqrt{n}$-consistent estimator of $\th_0$.
Let $\hn\in\mathcal{SLC}_0$ be an estimator of $g_0$ which satisfies Conditions~\ref{condition: on hn} and \ref{cond: hellinger rate}.  Suppose  $\eta_n=C n^{-2p'/5}$, where $C>0$ is any constant, and $p'\in(0,p]$, where $p$ is as in Condition~\ref{cond: hellinger rate}. 
Then the  estimator $\hth$  defined in \eqref{def: one-step estimator: truncated} satisfies
\begin{equation*}
\sqn(\hth-\th_0)\to_{d}N(0,\I^{-1}).
\end{equation*}
\end{theorem} 

A couple of remarks  are in order. First, Theorem~\ref{theorem: main: one-step: full} requires $\hn\in\mathcal{SLC}_0$. This automatically rules out most nonparametric density estimators including the symmetrized kernel density estimator of \citeauthor{stone}. 
Second, Theorem~\ref{theorem: main: one-step: full} requires $\bth$ to be  $\sqn$-consistent.  \citeauthor{stone} and \citeauthor{beran} impose  similar conditions on their preliminary estimators. The $Z$-estimator of the shift in the logistic location shift model is $\sqn$-consistent under minimal regularity conditions \citep[cf. Example 5.40 and Theorem $5.23$,][]{vdv}. When $f_0\in\mP_0$,  the sample mean and the sample median  also satisfy  this  requirement.

Now we give example of two $\hn$'s, which  satisfy the conditions of Theorem~\ref{theorem: main: one-step: full}.

\vspace{1em}
\textbf{Partial MLE estimator $\widehat{g}_{\bth}$:}
For any $\theta\in\RR$, the density class $\mathcal{SLC}_{\theta}$ admits an MLE
 \citep[Theorem 2.1(C),][]{dosssymmetric}. When $\th=\bth$, the MLE  in the class $\mathcal{SLC}_{\bth}$ is a legitimate estimator of $f_0$. We denote the corresponding density by $\widehat{f}_{\bth}$. Then the centered density $\widehat{g}_{\bth}=\widehat{f}_{\bth}(\cdot+\bth)$ is a potential choice for $\hn$  because  $\widehat{g}_{\bth}\in\mathcal{SLC}_0$.
 We call this estimator a {\sl Partial MLE estimator} to distinguish it from the traditional MLE of $g_0$, which we will discuss in Section \ref{sec: MLE}.
 From  \cite{dosssymmetric} it follows that $\log\widehat{g}_{\bth}= \widehat{\ps}_{\bth}$ is a piecewise linear concave function with domain $[-a,a]$, where 
$a=|X|_{(n)}$.

\textbf{Geometric mean type symmetrized estimator $\hn^{geo,sym}$:}
We denote by $\hf$ the MLE of $f_0$ among the class of all log-concave densities, which exists by \cite{exist}. The finite sample and asymptotic properties of $\hf$ are well-established \citep{2009rufi, theory}.  In particular,  $\log\hf$ is piecewise linear with domain $[X_{(1)},X_{(n)}]$. However,
the estimator $\hf$ need not be symmetric about any $\theta\in\RR$.
A symmetrized version of $\hf$ is given by
\begin{equation}\label{est 5}
 \hn^{geo,sym}(z):=C_{n}^{geo}\lb \hf(\bth+z)\hf(\bth-z)\rb^{1/2},\quad z\in\RR
 \end{equation}
  where $C_{n}^{geo}$ is a random normalizing constant. Here ``geo" refers to the mode of symmetrization, which is the geometric mean in this case. Since addition preserves concavity, $\log(\hf(\bth+z))+\log(\hf(\bth-z))$ is concave, which entails that $
 \hn^{geo,sym}\in\mathcal{SLC}_0$. 
  The support of $\hn^{geo,sym}$ takes the form $[-a,a]$, where 
  \[a=\min(X_{(n)}-\bth,\bth-X_{(1)}).\] 
 Observe that the support of $\hn^{geo,sym}$  is smaller than that of $\widehat{g}_{\bth}$, and it may also exclude some data points. Simulations suggest that the performance of  $\hn^{geo,sym}$ can suffer, especially in small samples,  due to the exclusion of data points.

Proposition~\ref{prop: partial and geo suffices} states that $\widehat{g}_{\bth}$ and $\hn^{geo,sym}$ satisfy Conditions~\ref{condition: on hn} and  \ref{cond: hellinger rate}, as postulated. The proof of Proposition \ref{prop: partial and geo suffices} can be found in Appendix~\ref{app: proof of proposition 1}.
\begin{proposition}\label{prop: partial and geo suffices}
Suppose $f_0\in\mP_0$. Then $\hn=\widehat{g}_{\bth}$ and $\hn^{geo,sym}$ satisfy Condition~\ref{condition: on hn} and Condition~\ref{cond: hellinger rate} with $p=1/4$ and $2/5$, respectively.
\end{proposition}
 
The $\hln$ corresponding to $\hn=\widehat{g}_{\bth}$ and $\hn^{geo,sym}$ is non-smooth since $\hln$ is piecewise linear in both cases. Such an estimator may not be the best choice  in small samples. Although a smoothed version of $\hn$ may perform better in small samples,  tuning of the smoothing parameter in a data dependent way may be a non-trivial task. For the log-concave MLE $\hf$, however, \cite{smoothed} construct a well-behaved smoothing parameter in a  completely data-dependent way.  This smoothing parameter is given by 
  \begin{equation}\label{definition bn}
 \widehat{\lambda}_n:=\sqrt{\hs-\ts},
 \end{equation}
 where   $\hs$ is the sample variance and $\ts$ is the variance corresponding to  $\hf$, that is
  \[\hs=\dfrac{1}{n-1}\si (X_i-\bX)^2\quad\text{and}\quad\ts=\edint z^2\hf(z)dz-\lb\edint z\hf(z)dz\rb^2.\]
That the right hand side of \eqref{definition bn} is positive follows from (2.1) of \cite{smoothed}.
 In light of the above, we construct a smooth $\hn$ which is symmetric about zero although it is not log-concave.

 \vspace{1em}
 \textbf{Smoothed symmetrized  estimator $\hn^{sym,sm}$:}
 Let us define the smoothed version of $\hf$ by
 \begin{equation}\label{def: smoothed log-concave MLE}
 \hf^{sm}(z)=\dfrac{1}{\widehat{\lambda}_n}\edint \hf(z-t)\gd(t/\widehat{\lambda}_n)dt,\quad z\in\RR,
 \end{equation}
 where $\gd$ is the standard normal density and $\smbn$ is as defined in \eqref{definition bn}.
We define the smoothed symmetrized  estimator by
   \begin{align}\label{representation of htsm}
\htsm(z)=\dfrac{\hts(\bth+z)+\hts(\bth-z)}{2}.
   \end{align}
%   \textbf{Remove if the following $\hn^{sym}$ is never needed}
%   \textcolor{blue}{The estimator \htsm\ can also be represented in terms  of the symmetrized density estimator $\hn^{sym}$  defined by
%   \begin{equation}\label{25eq3}
% \hn^{sym}(z)=\dfrac{1}{2}\lb\hf(\bth+z)+\hf(\bth-z)\rb,\quad z\in\RR,
% \end{equation}
% noting
%   \begin{equation}\label{est 1}
% \hn^{sym,sm}(z):=\dfrac{1}{\widehat{\lambda}_n}\edint\hn^{sym}(z-t)\gd(t/\widehat{\lambda}_n) dt,\quad z\in\RR. 
% \end{equation}
% Clearly, $\hnss$ is a smoothed version of the estimator $\tilde{g}_n^{sym}$. 
% }
   
% Since smoothing  evens out the irregularities in the tail of $\hf$, we expect a smoothed version of $\hn^{sym}$ to perform better than $\hn^{sym}$ in  smaller samples.

 It is natural to ask if similar data-dependent smoothing parameters exist for $\widehat{g}_{\bth}$ and $\hn^{geo,sym}$ as well. Although  a quantity analogous to $\widehat{\lambda}_n$ can be defined for these estimators, there is no guarantee that the former will be   positive. Nevertheless, data dependent smoothing of $\widehat{g}_{\bth}$ can be an interesting direction for future research.
 
 It can be shown that $\hn^{sym, sm}$   satisfies Condition~\ref{condition: on hn} and Condition~\ref{cond: hellinger rate} with $p=1/5$.
 Moreover, although  $\hn^{sym, sm}$ is not log-concave, it leads to an 
 adaptive estimator of $\theta_0$ for suitably chosen $\eta_n$. The proof of Theorem~\ref{theorem: main: one-step: hnss} can be found in Appendix~\ref{app: proof of Theorem 2}.
 \begin{theorem}\label{theorem: main: one-step: hnss}
Suppose $f_0\in\mP_0$ satisfies Assumption~\ref{assump: L}, and $\bth$ is a $\sqrt{n}$-consistent estimator of $\th_0$.
Let $\hn=\hnss$ and $\eta_n=C n^{-2p'/5}$, where $C>0$ and  $p'\in(0,1/5]$.
Then
the  estimator $\hth$  defined in \eqref{def: one-step estimator: truncated} satisfies
\begin{equation*}
\sqn(\hth-\th_0)\to_{d}N(0,\I^{-1}).
\end{equation*}
\end{theorem} 

\begin{remark}
We suspect that  the rate of  decay of the Hellinger error of the estimators $\widehat{g}_{\bth}$  and $\hnss$ is  faster than our obtained rates, which are $O_p(n^{-1/4})$ and $O_p(n^{-1/5})$, respectively.  Our guess is based on the fact that  the geometric symmetrized estimator $\hn^{geo, sym}$  and the full MLE    in $\mP_0$ (see Theorem~\ref{MLE: Rate results}) are Hellinger consistent at the rate $O_p(n^{-2/5})$. The latter indicates that   $H(\hn,g_0)$ is possibly $O_p(n^{-2/5})$ if  $\hn$ is an equally good estimator of $g_0$. However, the knowledge of $p$ does not contribute much in the tuning of $\eta_n$ for practical implementation.  Therefore, we do not pursue further theoretical investigation on  the best possible rate of $\eta_n$ in this paper.   
\end{remark}

For convenience, we list the key differences among our three main estimators of $g_0$  in Table~\ref{Table: 1}.

%  Our empirical results indicate that the $\hth$ constructed using $\hn^{sym,sm}$ is adaptive. 
%  Although $\hn^{sym,sm}$ fails to be log-concave,  it can still be represented as the average of two log-concave densities. In this light, one could hope that the proof of Theorem~\ref{theorem: main: one-step: full} can be mimiced to establish the asymptotic efficiency of the $\hth$ based on $\hn^{sym,sm}$. It turns out, however, that the biggest bottleneck in proving Theorem~\ref{theorem: main: one-step: full} for  $\hn^{sym,sm}$ is Condition~\ref{cond: hellinger rate}. The  Hellinger error depends on the rate of $b_n$, which is yet unknown.
 
%   does not apply to $\hn^{sym,sm}$ because a) it is not log-concave b) deriving the rate of $H\hn^{sym,sm}, g_0)$ is difficult with existing tools since this estimator does not share a close proximity with any MLE.

  \begin{table*}[ht]
\caption{Comparison of different $\hn$'s: here $\hf$ is the log-concave MLE, $\hts$ is the smoothed log-concave MLE as defined in \eqref{def: smoothed log-concave MLE}, and $C_n^{geo} $ is the normalizing constant in \eqref{est 5}.
  }
     \begin{tabular}{@{}llll@{}}
\toprule
 Estimator ($\hn$)   &  $\hn^{sym,sm}$ & $\widehat{g}_{\bth}$ & $\hn^{geo,sym}$ \\ 
 \midrule

Summary &  Smoothed  & Partial MLE & GM type\\
& symmetrized   &   & symmetrized \\
 Formula & $2^{-1}(\hts(\bth+z)$ & $\widehat{g}_{\bth}(z)$ & $C_{n}^{geo}( \hf(\bth+z)$\\ 
of  $\hn(z)$  & $+\hts(\bth-z))$ & & $\times \hf(\bth-z))^{1/2}$\\
 & & &\\
 Log-concave & No & Yes & Yes\\ 
 
 Smooth  & Yes & No  &  No  \\ 
 
 Support & $\RR$ & $[-|X|_{(n)},|X|_{(n)}]$ & $[-a,a],\ a=\min(X_{(n)}-\bth,$ \\
  &  &  & $\bth-X_{(1)})$ \\
   \bottomrule
 \end{tabular} 
  
   \end{table*}\label{Table: 1}
   We close this section with a conjecture. It has previously been mentioned that the lack of control on $\tp$ at the tails make asymptotic analysis of the untruncated estimator difficult. However, we conjecture that the untruncated estimator $\hthf$ is also adaptive, i.e. $\sqn(\hthf-\theta_0)\to_d N(0,\I^{-1})$. Our simulations in Section~\ref{sec: simulation} do not refute this conjecture.

%-------------------------------------------------------------

\section{Maximum likelihood estimator (MLE)}
\label{sec: MLE}
In this section, we prove that the MLE of $(\th_0,g_0)$ exists, and explore some of its properties. Before going into further details, we introduce some new terminologies. Recall that by our definition of $\mathcal{SC}_\theta$, the class $\mathcal{SC}_0$ consists of all properd closed concave functions symmetric about the origin.
  For $\ps\in\mathcal{SC}_0$ and $\th\in\RR$, following \cite{dumbreg} and \cite{xuhigh}, we define the criterion function for  maximum likelihood estimation  by
 \begin{equation}\label{criterion function: xu samworth}
 \Psi(\th,\ps,F)=\edint \ps(x-\th)dF(x)-\edint e^{\ps(x-\th)}dx.
 \end{equation}
Following \cite{silverman1982}, we  included a Lagrange term  to get rid of the normalizing constant  involved in density estimation. This is a common device in log-concave density estimation literature \citep[cf.][]{2009rufi, dosssymmetric}.
% Dumbgen and Rufibatch (2009),  Doss and Wellner (2018)

 We use the notation $\Psi_n(\th,\ps)$ to denote the sample version  $\Psi(\th,\ps,\Fm)$ of $\Psi(\th,\ps,F)$. Thus,
  \begin{equation}\label{criterion function: Doss}
  \Psi_n(\th,\ps)=\edint \ps(x-\th)d\Fm(x)-\edint e^{\ps(x-\th)}dx.
  \end{equation}
%  involved in density-classes like $\mathcal{SLC}_{0}$. 
Let us denote  the MLE of $(\th_0,g_0)$ by $(\hthm,\hgm)$ when they exist.
We also denote $\hpm=\log\hgm$.  Observe that provided they exist,  $(\hthm,\hpm)$  satisfies
 \[(\hthm,\hpm)=\argmax_{\th\in\RR,\ps\in\mathcal{SC}_0}\Psi_n(\th,\ps).\]
For fixed $\theta\in\RR$, denote by $\widehat \ps_{\th}$ the maximizer of $   \Psi_n(\th,\ps)$ in $\ps\in\mathcal{SC}_0$.  Theorem 2.1(C) of \cite{dosssymmetric} implies the maximizer $\widehat{\ps}_{\th}$ exists, is unique, and that it satisfies
\[\edint e^{\widehat{\ps}_{\th}(x)}dx=1.\]
 It is not hard to see that if the MLE $(\hthm,\hpm)$ exists, then
   \[\hthm=\argmax_{\th\in\RR}\Psi_n(\th,\pst)\quad \text{and}\quad\hpm=\widehat{\ps}_{\hthm}.\]
   Note that $\hgm=e^{\hpm}$ is the MLE of $g_0$,  and $\hgf=\hgm(\mathord{\cdot}-\hthm)$ is the MLE of $f_0$.
Theorem~\ref{theorem: existence} implies that the   the MLE $(\hthm,\hpm)$ exists when $\Fm$ is non-degenerate. The proof of Theorem\ref{theorem: existence} an be found in Appendix~\ref{app: existence theorem}.
\begin{theorem}\label{theorem: existence}
When $\Fm$ is non-degenerate, the MLE $(\widehat{\theta}_n , \widehat{\psi}_{n} )$ of $(\theta_0 , \psi_0 )$ exists. If $\hthm$ is unique, then $\hthm\in[X_{(1)},X_{(n)}]$. Otherwise, we can find at least one $\hthm\in[X_{(1)},X_{(n)}]$.
\end{theorem}

Observe that Theorem~\ref{theorem: existence}  does not claim that \hthm\ is unique. Since $\Psi_n(\th,\ps)$ may not be jointly concave in $\theta$ and $\ps$, existence of a maximizer does not lead automatically to its uniqueness. For a particular choice of $\hthm$ however, the estimator   $\hpm=\widehat{\ps}_{\hthm}$ is unique by Theorem~$2.1(c)$ of \cite{dosssymmetric}. Therefore, if $(\th,\ps_1)$ and $(\th,\ps_2)$ both are MLEs of $(\th_0,\ps_0)$, we must have $\ps_1=\ps_2$.

{Although we can not theoretically prove the uniqueness of $\hthm$, we are yet unaware of any set up which leads to non-unique MLE.}
Moreover, in all our simulations, \hthm\ turned out to be unique, even when the underlying density $f_0$ was skewed or non-log-concave. Considering this fact, in what follows, we  refer to \hthm\ as ``the MLE" instead of ``an MLE". We must remark that even if $\hthm$ is not unique, all our theorems  still hold for each version of \hthm.

 On the other hand, when $\Fm$ is degenerate, Lemma~\ref{lemma: mle does not exist degenerate} entails that the MLE  does not exist. 
However,  for distributions with a density, probability of  $\Fm$ being degenerate is zero. Therefore  we will not worry about this  particular situation. The proof of Lemma~\ref{lemma: mle does not exist degenerate} is given in Appendix~\ref{app: mle lemma}.
 \begin{lemma}\label{lemma: mle does not exist degenerate}
  Suppose $\Fm$ is degenerate, i.e.  $\Fm\{x_0\}=1$ for some $x_0\in\RR$. Then the MLE of $(\theta_0,g_0)$ in $\mathcal P_0$ does not exist.
 \end{lemma}
   
% Doss and Wellner (2018)

  The following theorem sheds some light on the structure of $\hpm$.  This theorem is a direct consequence of Theorem~$2.1(c)$ of \cite{dosssymmetric}, and hence we skip the proof.
% Doss and Wellner (2018)
%
%Theorem on the domain, structure etc. for the MLE of the density g_0
%
\begin{theorem}\label{theorem: structure of the MLE of the density}
Suppose $(\hthm,\hpm)$ is the MLE.
For $\Fm$ non-degenerate,  $\hpm$  is piecewise linear with knots belonging to a subset of the set $\{0,\pm|X_1-\hthm|,\ldots,\pm|X_{n}-\hthm|\}$. Also,  for $x\notin[-|X-\hthm|_{(n)},|X-\hthm|_{(n)}]$, we have $\hpm(x)=-\infty$.
Moreover if $0\notin \{\pm|X_1-\hthm|,\ldots,\pm|X_{n}-\hthm|\}$, then $\hpm'(0\pm)=0.$
\end{theorem}
The MLE can be computed using our R package \texttt{log.location}, which implements a grid search method to optimize $\Psi_n(\theta,\widehat \psi_{\theta})$ in $\theta$. 

\subsection{Asymptotic properties of the MLE}
For $f_0\in\mP_0$,  we  showed that  the one-step estimators are consistent. Theorem~\ref{MLE: Rate results} (A) below shows that the MLE $\hthm$  enjoys similar consistency 
property. In fact, $\hthm$ is strongly consistent for $\th_0$. Part A of Theorem~\ref{MLE: Rate results} also entails that $\hgm$ and $\hgf$ are strongly Hellinger consistent. Part B of Theorem~\ref{MLE: Rate results} concerns the rate of convergences.
 The proof of Theorem~\ref{MLE: Rate results} is delegated to Appendix~\ref{app: rate results}.
\begin{theorem}\label{MLE: Rate results}
Suppose $f_0\in\mP_0$.
   Then the following assertions hold:
   \begin{compactenum}[label=(\Alph*),ref=(\Alph*)]
  \item\label{lemma 2: E} As $n\to\infty$, $\hthm\as \theta_0$, $H(\hgf,f_0)\as 0$, and 
  $H(\hgm,g_0)\as 0$.
  \item\label{lemma 2: F} Furthermore, ${|\hthm-\th_0|=O_p(n^{-2/5})}$, $H(\hgf,f_0)=O_p(n^{-2/5})$, and\\ $H(\hgm,g_0)=O_p(n^{-2/5})$.
   \end{compactenum}
\end{theorem}
The rate of $H(\hgm,g_0)$ as given by Theorem~\ref{MLE: Rate results}  is standard for  log-concave density estimators. The MLEs in  $\mathcal{SLC}_{0}$ and $\mathcal{LC}$ have the same rate of  Hellinger error decay  \citep[see Theorem~4.1(c) of][]{dosssymmetric}. Moreover, this rate 
probably can not be improved by any other  estimator of $g_0$. To see why, first note that Theorem 1 of \cite{dossglobal} proves that the minimax rate of   Hellinger error decay in  $\mathcal{LC}$ is $O_p(n^{-2/5})$.
Remark 4.2 of \cite{dosssymmetric} conjectures that the minimax rate of estimation in the constrained class $\mathcal{SLC}_0$  stays the same. Since estimation of $g_0$ in $\mP_0$ can not be  easier than  estimation in the smaller class $\mathcal{SLC}_0$, it is likely  that the  minimax rate of estimating $g_0$ in $\mP_0$ is also $O_p(n^{-2/5})$.

However, the MLE $\hthm$ probably  convergences to $\th_0$ at a rate faster than  $O_p(n^{-2/5})$. Our simulations suggest that $\hthm$ is $\sqn$-consistent, based on which, we  conjecture that $\hthm$ is also an adaptive estimator of $\theta_0$. In our model, the low dimensional parameter of interest, i.e. $\theta_0$, is bundled with  the infinite dimensional nuisance parameter.  Obtaining the precise rate of convergence for the MLE  in such semiparametric models is typically difficult \citep{murphy}.  Nevertheless, since the MLE is tuning parameter free, finding its exact asymptotic distribution will be an interesting future research direction.

% See  Appendix~\ref{sec: appendix: mle finite} in the supplement for more discussion on finite sample properties of  $\hthm$ and $\hpm$.
% Theorem~\ref{theorem: structure of the MLE of the density} implies that if $0\notin \{\pm|X_1-\hthm|,\ldots,\pm|X_{n}-\hthm|\}$, by our definition of a knot in section \ref{sec: Preliminaries}, $0$ is not a knot of $\hpm$.

%%%% Theorem concerning the survival function

%  Also notice that 
%\[\edint (x-\hthm)^2d\Fm(x)=S_X^2+(\hthm-\bar{X})^2\]
%where $\bar{X}$ and $S_X^2$ are the sample mean and variance respectively. 
% 

\section{Simulation study}\label{sec: simulation}
   
   \begin{figure}[ht]
\includegraphics[width=\textwidth, height=.38\textwidth]{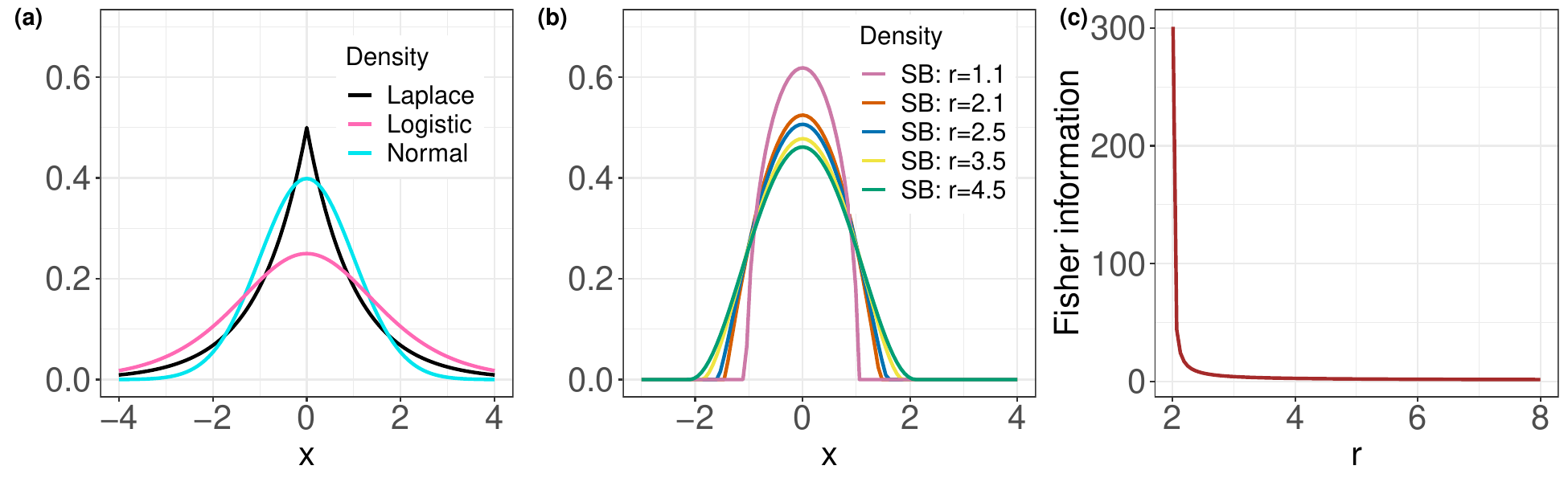}
\caption{
(a) Plot of the standard Laplace, standard normal and standard logistic densities.
(b) Plot of the
 symmetrized beta  density $f_{0,r}$,  defined in \eqref{definition: Symmetrized beta}, for  different values of $r$.
 (c) Plot of Fisher information $\mathcal{I}_{f_{0,r}}$ vs $r$ where $f_{0,r}$ is the symmetrized beta density. }
\label{Figure: comparison of truncated information densities 1}
\end{figure}

% We benchmark the performance of our estimators against that of   \cite{stone} and \cite{beran}'s estimators.
   This section compares the efficiency of our estimators and the coverage of the resulting confidence intervals with  that of \citeauthor{stone} and \citeauthor{beran}.
   The general set-up of the simulation is as follows.
   We consider as $g_0$ the standard normal, standard logistic, and standard Laplace density. We also  consider a fourth density, namely the symmetrized beta density,  which is defined as follows:
   \begin{equation}\label{definition: Symmetrized beta}
f_0(x)\equiv f_{0,r}(x)=\dfrac{\Gamma\lb(3+r)/2\rb}{\sqrt{\pi r}\Gamma(1+r/2)}\lb 1-\dfrac{x^2}{r}\rb^{r/2}1_{[-\sqrt{r},\sqrt{r}]}(x),\quad r> 0.
\end{equation}
Here $\Gamma$ is the usual Gamma function.
 It is straightforward to verify that in this case 
 \[\phi_0'(x)=\dfrac{-x}{1-x^2/r}1_{[-\sqrt{r},\sqrt{r}]}(x)\quad\text{and }\quad\phi_0^{\prime\prime}(x)=-\frac{(1+x^2/r)}{(1-x^2/r)^2}1_{[-\sqrt{r},\sqrt{r}]}(x).\]
 Some computation shows  that $r\leq 2$ leads to $\mathcal{I}_{f_{0,r}}=\infty.$ However   for $r>2$, $\mathcal {I}_{f_{0,r}}<\infty$, and
%  \[\mathcal{I}_{f_{0,r}}=\dfrac{r\Gamma\lb \dfrac{r}{2}-1\rb \Gamma\lb \dfrac{3+r}{2}\rb}{2\Gamma\lb \dfrac{r}{2}+1\rb \Gamma\lb \dfrac{1+r}{2}\rb}<\infty,\]
  $f_0\in\mP_0$. This is an example of a case where Assumption~\ref{assump: L} fails to hold because $\phi_0^{\prime\prime}$ is unbounded. We consider the symmetrized beta density with $r=2.1$ and $4.5$.
  
   See  Figure~\ref{Figure: comparison of truncated information densities 1}a and \ref{Figure: comparison of truncated information densities 1}b  for a pictorial representation of the above-mentioned densities. 
 Figure~\ref{Figure: comparison of truncated information densities 1}c  displays the plot of $\mathcal{I}_{f_{0,r}}$ versus $r$ for the symmetrized beta density, which depicts that $\mathcal{I}_{f_{0,r}}$  decreases steeply for $r>2$. This finding is consistent with $\I$ being $\infty$ when $f_0$ is the uniform density on $[-1,1].$

We set $\th_0=0$, and generate $3000$ samples of size $n=30$, $100$, $200$, and $500$ from each of the above-mentioned densities. We define the efficiency of an estimator $\theta_n$ by 
     \begin{equation}\label{Def: efficiency}
   \text{Efficiency}(\th_n)=\dfrac{1/(n\I)}{ Var(\th_n)}.
    \end{equation} 
 In practice, we replace $Var(\th_n)$ by its Monte Carlo estimate.

   \subsubsection*{\textbf{The shape-constrained estimators:}}
    Along with the MLE and the untruncated one step estimator defined in \eqref{def: one-step estimator full}, we consider the truncated one step estimators with truncation level $\eta=10^{-2}$, $10^{-3}$, and $10^{-5}$.  We select the sample mean as the preliminary estimator $\bth$ because it exhibited slightly better overall performance than other potential choices of $\bth$, e.g. the median and the trimmed mean.   We choose the partial MLE estimator and the smoothed symmetrized estimator $\tilde g_n^{sym,sm}$ as the  estimator of $g_0$ because simulations suggest that they perform significantly better than  $\hn^{geo, sym}$.
  
  \subsubsection*{\textbf{Comparators: \citeauthor{stone} and \citeauthor{beran}'s estimators:}}
%   We benchmark the performance of our estimators against the nonparametric estimators   . 
   As mentioned earlier, \citeauthor{stone}'s estimator is a truncated one step estimator which uses symmetrized Gaussian kernels  to estimate $g_0$. Similar to \citeauthor{stone}, we let
  the corresponding truncation parameter  and the kernel bandwidth parameter  to be $d_n s_n$ and $t_n s_n$, respectively, where $s_n$ is the median absolute deviation (MAD), and $d_n>0$ and $t_n>0$ are  tuning parameters. 
  Following \citeauthor{stone}, we  take the preliminary estimator to be   the sample median.
  
  As previously stated, \citeauthor{beran}'s estimator is a  rank-based estimator which depends on the scores. \citeauthor{beran} uses  Fourier series expansion to estimate the scores, which requires choosing (a) the number of  basis functions ($b_{c,n}$), and (b) a scaling parameter $\rho_n$, which is used to approximate a derivative term by quotients during the estimation of the Fourier coefficients of the score.  This  estimator uses a preliminary estimator of $\theta$, which we take to be the sample median following \citeauthor{beran}'s suggestion.
  In this case, the sum of squares of the estimated Fourier coefficients
  is a  consistent estimator of $\I$ \citep[see  (3.3) of ][]{beran}.
  
  For sample size $n=40$, \citeauthor{stone} uses $d_n=20$ and $t_n=0.60$, but \citeauthor{beran} does not give any demonstration on how to choose the tuning parameters.
  To choose some reliable values for the associated tuning parameters, we start with some pre-selected grids, and employ a grid search procedure (see Appendix~\ref{app: stone} for more details). The selected tuning parameter is the maximizer of the estimated efficiency among the grid, where the   efficiency is estimated using  one hundred Monte Carlo replications. Of course, this procedure requires the knowledge of the unknown distribution, and hence, not implementable in practice. However, our procedure at least guarantees a reliable benchmark  to compare the performance of our estimators. 
  We refer to the resulting tuning parameters as ``optimal" for the sake of simplicity. However, it should be kept in mind that these tuning parameters depend on the chosen grid, and therefore, may be different from the globally optimal tuning parameters if the grid selection is not accurate enough. This could have been overcome by an exhaustive search but that is beyond the scope of the current paper.   
 
%   It may be possible to construct data-dependent  methods for finding reliable tuning parameters. However, that is out of scope of the present paper.
  
  For each distribution and each sample size, we construct  two versions of the nonparametric estimators. The first version is based on the aforementioned optimal tuning parameter, and the other version  
  uses  tuning parameters slightly away from the optimal region. For convenience, we will refer to the second set of tuning parameters as ``non-optimal".  See Appendix~\ref{app: stone} for more details on these tuning parameters.
  
  We should mention that neither \citeauthor{stone} nor \citeauthor{beran} construct  confidence intervals. However,   both estimators rely on consistent estimators of $\I$, namely, the estimator $\widehat A_n(r_n,c_n)$   of \citeauthor{stone} (see  (1.10) of \citeauthor{stone}),  and the squared $L_2$ norm of the estimated score in  \citeauthor{beran}.
  We use the above estimators of $\I$  to build the respective confidence  intervals of \citeauthor{stone} and \citeauthor{beran}.
  
  \subsection*{\textbf{Results:}}
  Figure~\ref{fig: IR} implies that  \citeauthor{stone} and \citeauthor{beran}'s estimators have high  efficiency  when they are equipped with the optimal tuning parameters. In fact Stone's estimator has better efficiency than all other estimators in case of logistic and normal distribution. However, even with the optimal tuning parameter, the coverage of \citeauthor{stone}'s confidence interval is quite low (see Figure~\ref{fig: coverage: bs}). The coverage of \citeauthor{beran}'s confidence interval is comparatively better but  still not as good as  the shape-constrained estimators (see Figure~\ref{fig: coverage}). The poor coverage of the nonparametric confidence intervals is probably  due to their smaller width, as shown by Figure~\ref{fig: length}. We suspect that for our tuning parameters, the nonparametric estimators of $\I$  overestimate $\I$, leading to narrow confidence intervals.
  When the tuning parameters are non-optimal, the  nonparametric estimators suffer  in terms of both efficiency and the coverage. This is most evident in large samples because in this case,  their performance does not significantly improve  with the sample size. Figure~\ref{fig: IR} and \ref{fig: coverage} entail that
  all estimators have markedly poor performance in the symmetrized beta case when $r=2.1$.
  
%   To examine this issue more closely, we single out one distribution, namely standard Gaussian, and compare the efficiency and coverage of the nonparametric estimators for different values of $t_n$ for fixed $d_n$. We remind the readers that if the underlying Gaussian distribution has higher variance, then the range of allowable $d_n$ shrinks. For example, when the standard deviation $\sigma=$, then.... So for a fixed $t_n$, the underlying distribution  $d_n$ the estimator performance starts denigrating depends on the distribution. In the following paragraph, we give a more detailed account of the behavior of the nonparametric estimators when the tuning parameters deviate from the optimal value. 
   
   Let us turn our attention to the one step estimators now. Figure~\ref{fig: IR} underscore that
   the efficiency of the one step estimators  monotonously decreases with the truncation, with the highest efficiency being observed  at the truncation level zero. However, the difference becomes smaller as the truncation level decreases. In particular, at  truncation level $10^{-5}$, the difference almost vanishes. The one step estimators with lower truncation level, i.e.  $\eta\leq 0.001$, exhibit satisfactory performance in terms of both efficiency and coverage (see Figure~\ref{fig: IR} and Figure~\ref{fig: coverage}). The  estimators with higher truncation level lag in terms of efficiency as expected, although they  exhibit superior coverage in some cases.
  
   The additional gain in  coverage that sometimes accompany  higher levels of truncation is
   probably due to slightly wider confidence intervals (see Figure~\ref{fig: length}). Wider confidence intervals are expected with high levels of truncation since the length of the confidence intervals, which is a constant multiple of $\hin(\eta)^{-1}$, increases in $\eta$.  However, higher level of truncation may not always lead to a better coverage, especially since high truncation level  can also result in significant loss of efficiency. See for instance the case of symmetrized beta with $r=2.1$, where the one step estimators with truncation $0.01$ lags behind the other one-step estimators   in terms of both efficiency and coverage. This case clearly demonstrates that the one step estimators with higher level of truncation are not always  reliable. In contrast, the one step estimators with low level of truncation, particularly the untruncated one step estimator, always exhibit  satisfactory performance. In view of above, we propose the untruncated estimator for practical implementation.

   Close inspection shows  the  smoothed symmetrized estimators have better overall performance than the partial MLE estimators  with the obvious exception of Laplace distribution, which has a non-smooth density.
  Finally,  we note that the one step estimators with lower truncation level  have better efficiency than the MLE under all distributions except  Laplace. However, when it comes to the coverage of the confidence intervals, the MLE can be competitive with the best one step estimators, especially in small samples.

  %The smoothed symmetrized one step  estimator without  truncation exhibits the best efficiency  in this case.

%   Figure~\ref{fig: IR} implies that Stone's estimator with the best tuning parameter has higher efficiency than all other estimators. However, when we change the tuning parameters to $d_n=10$ and $t_n=0.10$, the performance of Stone's estimator suffers heavily. The same can be said about Beran's estimator except even with the best tuning, it lags behind some one step estimators in case of non-Gaussian distributions. 
  
%   Figure~\ref{fig: coverage old} indicates that the coverage of nonparametric confidence intervals  is less than the coverage of most shape-constrained estimators. Especially, the confidence intervals constructed using Stone's estimators have coverage less than $0.60$ for all sample sizes, except for the case of Laplace. 

  In summary, the coverage of the nonparametric confidence intervals is not satisfactory for the tuning parameters considered here, and  the efficiency of the nonparametric estimators depends crucially on the tuning parameters.  For some choices of tuning parameter,  these estimators may exhibit excellent efficiency but for other choices, they severely underperform. In contrast, our untruncated one step estimator and the resulting confidence interval perform reasonably well under all scenarios.
  The performance of the untruncated one-step estimator also speaks in favor of our conjecture that it is an adaptive estimator.
  Although we do not show the plots of the mean squared error (MSE) here, they depict the same
  patterns as the efficiency plots in Figure~\ref{fig: IR}.
  
  We close this section with a remark on the necessity of Assumption~\ref{assump: L}. The symmetrized beta distributions do not satisfy Assumption~\ref{assump: L}, but the one step estimators still seem to be efficient when $r=4.5$. Although the one step estimators 
   perform poorly in case of $r=2.1$, they still perform better than \citeauthor{stone} and \citeauthor{beran}'s estimators, whose  asymptotically efficiency under this distribution is theoretically validated. Thus,  our simulations do not refute the possibility that Assumption~\ref{assump: L} might be unnecessary. 
   
%   Our simulations suggest (see Stone table) that the optimal choice of the tuning parameters (in terms of efficiency)  varies with the data distribution. So probably there is no universal way to garuantee 
  
%   On the other hand, choosing the truncation parameter for our one step estimator is straightforward because both the efficiency and the coverage  monotonically increase as  the level of truncation parameter shrinks towards zero. Since a smaller truncation parameter is always preferable,  the untruncated version is always superior to its truncated counterparts. 
%   The performance of the tuning free MLE is unimpressive  in all our settings, however.
   
  \begin{figure}[ht]
       \centering
       \begin{subfigure}{\textwidth}
        \includegraphics[width=\textwidth, height=2.8 in]{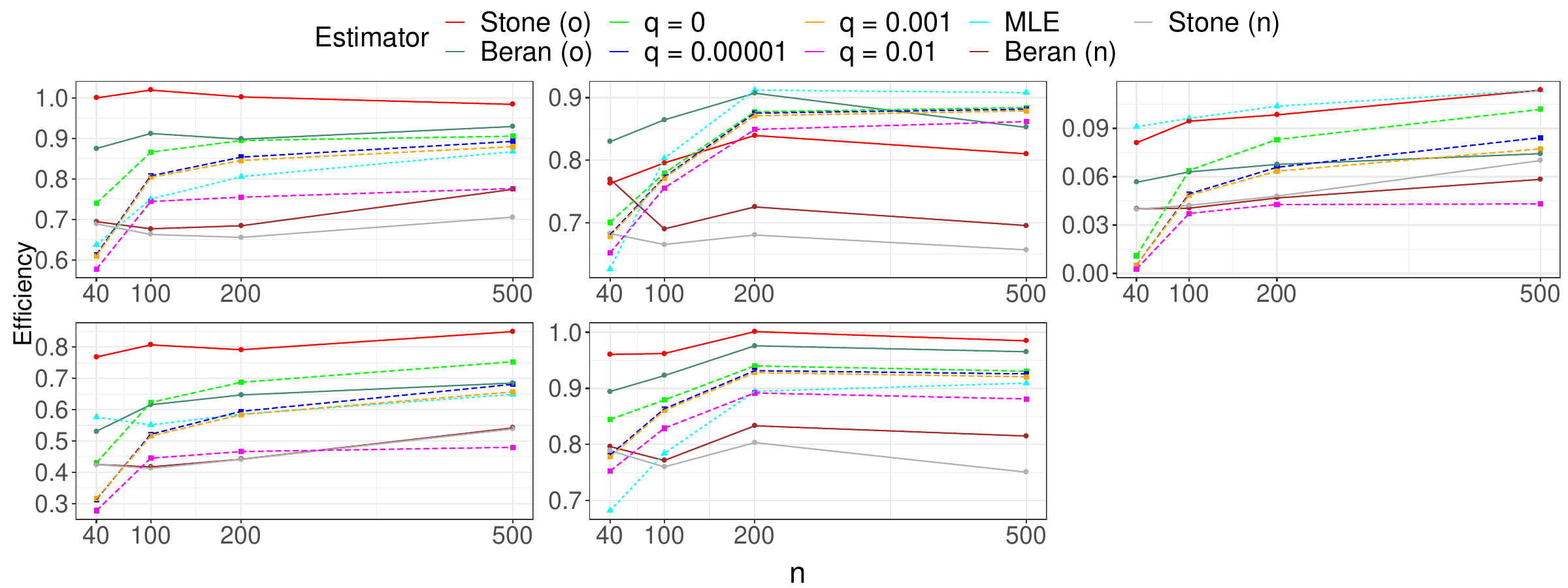}
        \caption{Comparison plot when the one step estimators are the Partial MLE estimator}
       \end{subfigure}\\
       \begin{subfigure}{\textwidth}
        \includegraphics[width=\textwidth, height=2.8 in]{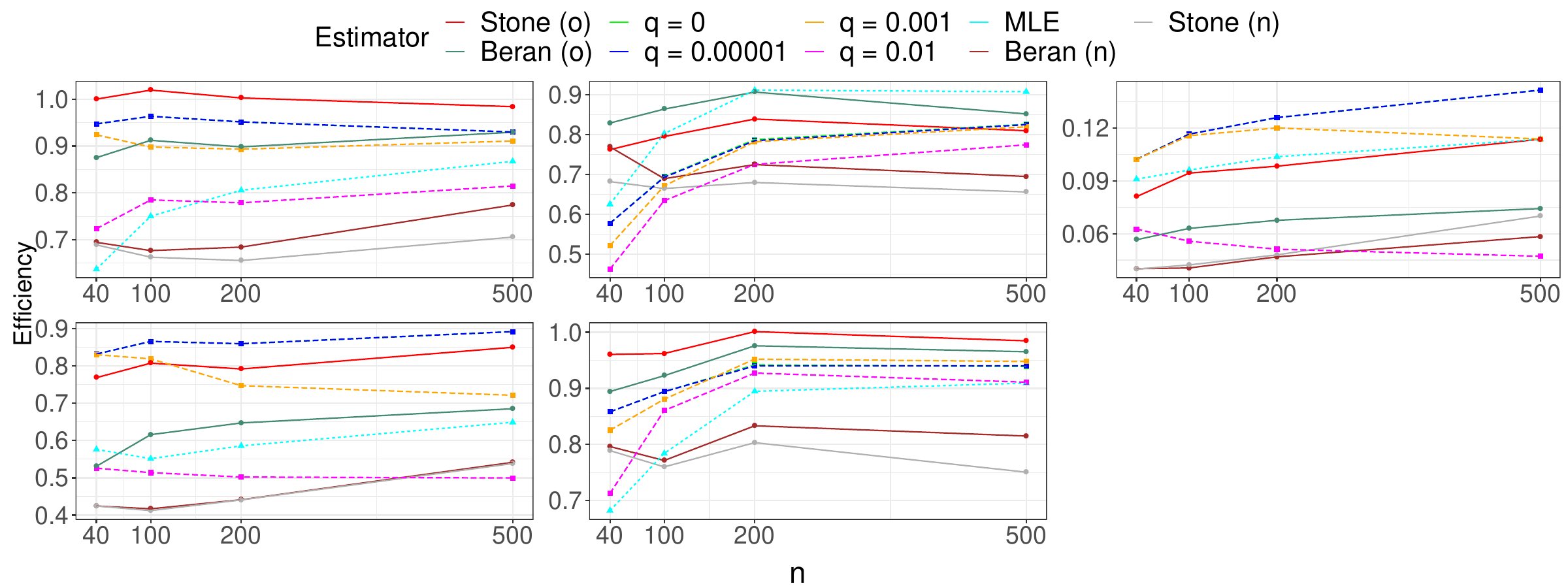}
        \caption{Comparison plot when the one step estimators are the smoothed symmetrized estimator}
       \end{subfigure}
       \caption{comparison of  efficiency: the data-generating distributions are normal (topleft), Laplace (topmiddle), Symmetrized beta with $r=2.1$ (topright), and $r=4.5$ (bottomleft), and logistic (bottommiddle).  For Stone's and Beran's estimators (in solid lines), (o) stands for the optimal tuning parameter, and (n) corresponds to the  non-optimal tuning parameter. Here $q$ stands for the truncation parameter $\eta$ in our one-step estimators (in dashed lines). }
       \label{fig: IR}
   \end{figure}

 \begin{figure}[ht]
       \centering
       \begin{subfigure}{\textwidth}
        \includegraphics[width=\textwidth, height=2.8 in]{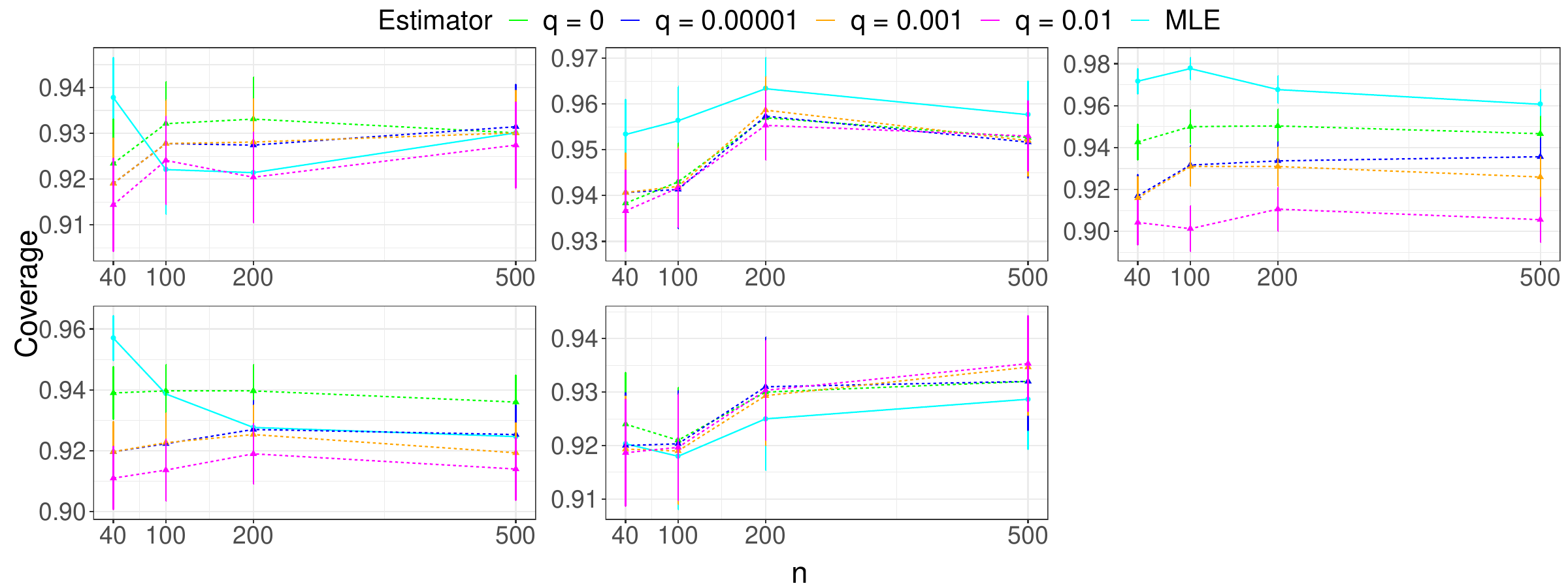}
        \caption{Comparison plot when the one step estimators are the Partial MLE estimator}
       \end{subfigure}\\
       \begin{subfigure}{\textwidth}
        \includegraphics[width=\textwidth, height=2.8 in]{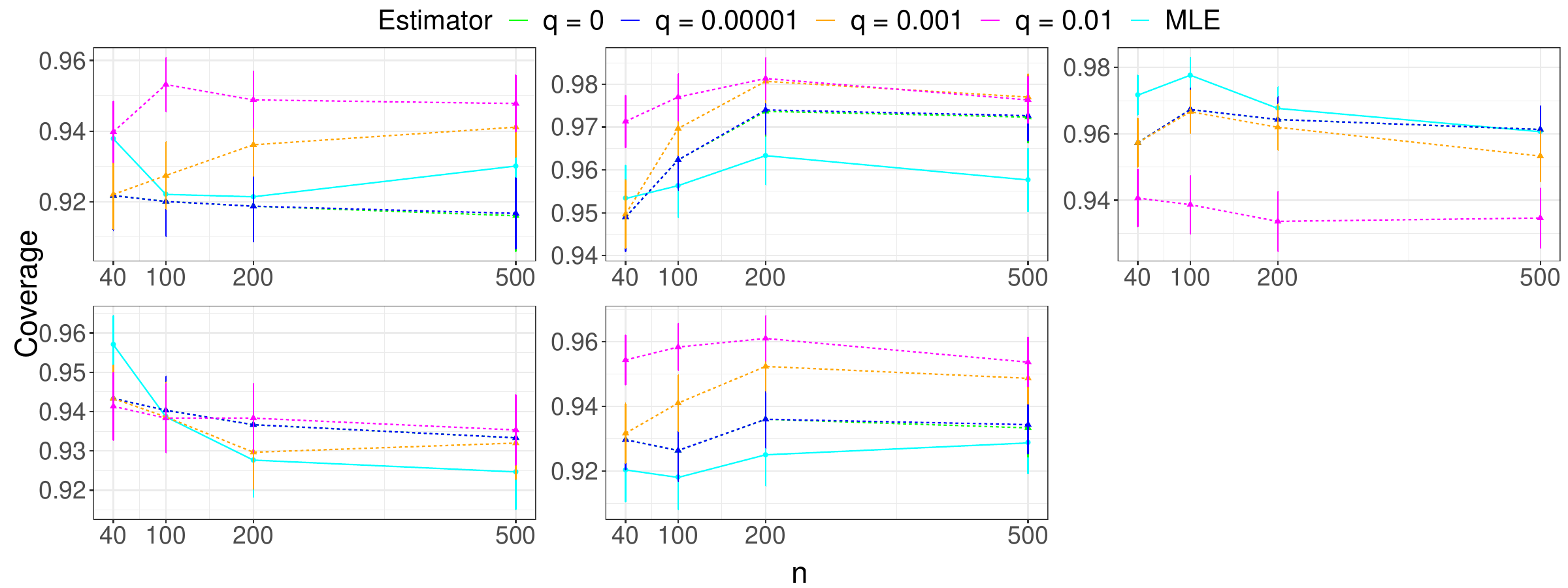}
        \caption{Comparison plot when the one step estimators are the smoothed symmetrized estimator}
       \end{subfigure}
       \caption{Comparison of the coverage of the 95\% confidence intervals:  the data-generating distributions are normal (topleft), Laplace (topmiddle), Symmetrized beta with $r=2.1$ (topright), and $r=4.5$ (bottomleft), and logistic (bottommiddle).
       Here $q$ stands for the truncation parameter $\eta$ in our one-step estimators. The errorbars are given by $\pm 2$ standard deviation.}
       \label{fig: coverage}
   \end{figure}

   \begin{figure}[ht]
       \centering
       \begin{subfigure}{\textwidth}
        \includegraphics[width=\textwidth, height=2.8 in]{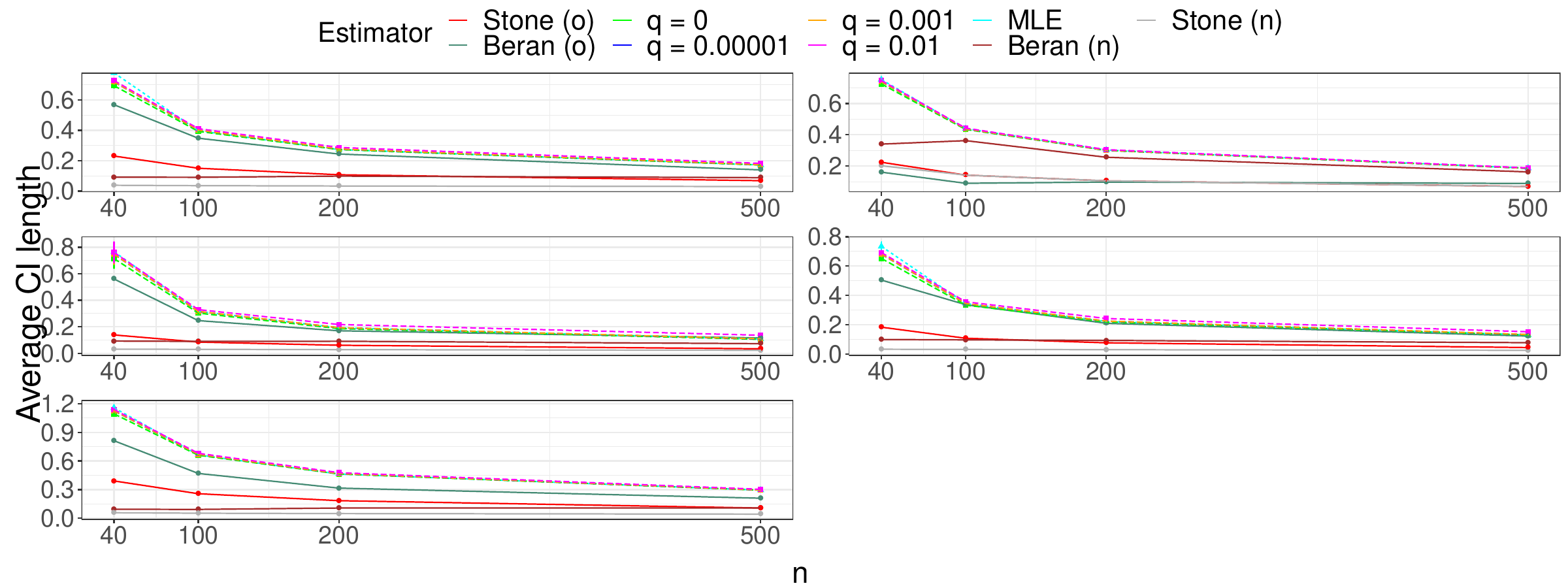}
        \caption{Comparison plot when the one step estimators are the Partial MLE estimator}
       \end{subfigure}\\
       \begin{subfigure}{\textwidth}
        \includegraphics[width=\textwidth, height=2.8 in]{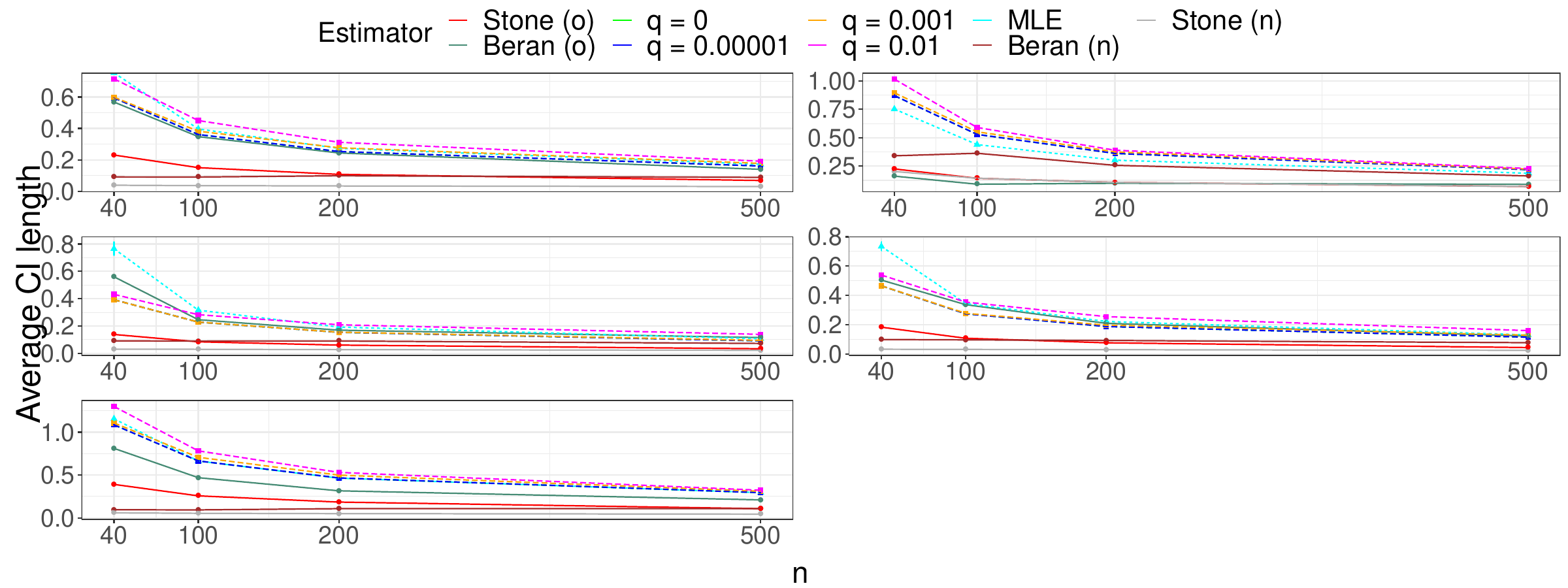}
        \caption{Comparison plot when the one step estimators are the smoothed symmetrized estimator}
       \end{subfigure}
       \caption{Comparison of the average confidence interval length (averaged across the 3000 Monte Carlo samples): the data-generating distributions are normal (topleft), Laplace (topmiddle), Symmetrized beta with $r=2.1$ (topright), and $r=4.5$ (bottomleft), and logistic (bottommiddle).  For Stone's and Beran's estimators (in solid lines), (o) stands for the optimal tuning parameter, and (n) corresponds to the  non-optimal  tuning parameter. Here $q$ stands for the truncation parameter $\eta$ in our one-step estimators (in dashed lines). The errorbars are given by $\pm 2$ standard deviation. }
       \label{fig: length}
   \end{figure}
   
   \begin{figure}[ht]
       \centering
        \includegraphics[width=\textwidth, height=2.8 in]{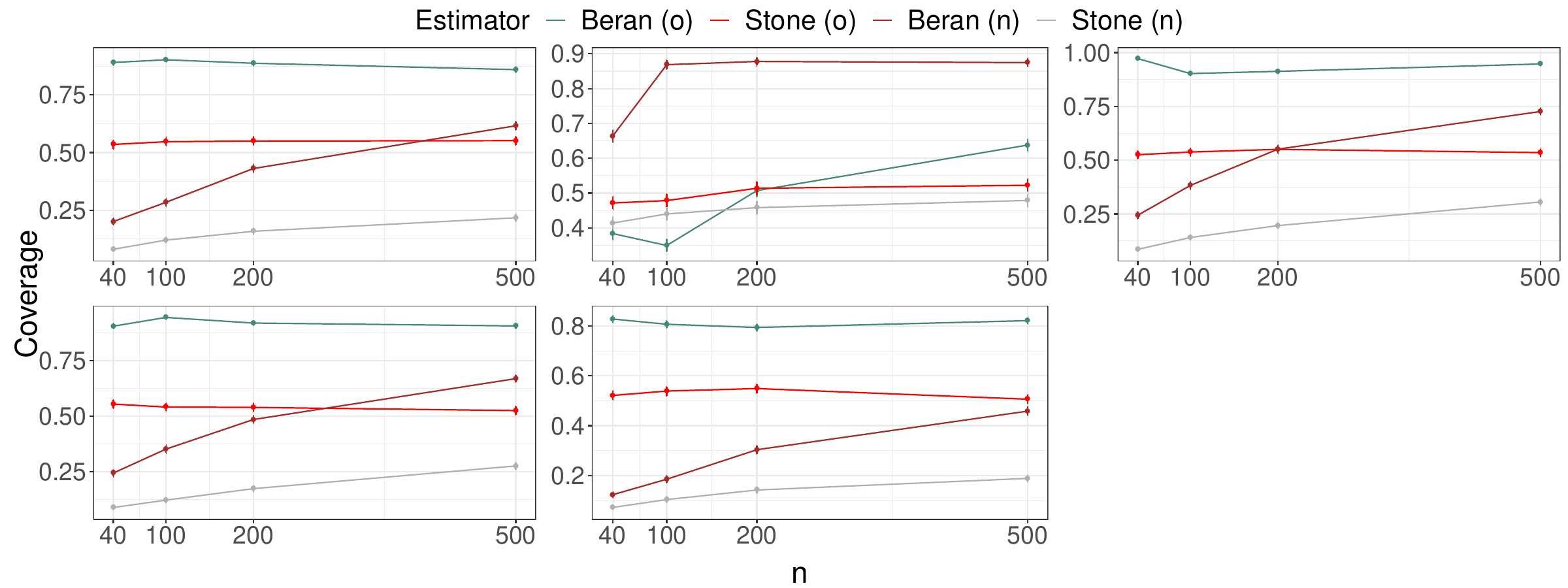}
       \caption{Comparison of the coverage of the 95\% confidence intervals for Beran's and Stone's estimators:  the data-generating distributions are normal (topleft), Laplace (topmiddle), Symmetrized beta with $r=2.1$ (topright), and $r=4.5$ (bottomleft), and logistic (bottommiddle). Here (o) stands for the optimal tuning parameter, and (n) corresponds to the non-optimal  tuning parameter. The errorbars are given by $\pm 2$ standard deviation.}
       \label{fig: coverage: bs}
   \end{figure}

\section{Discussion}
\label{sec: discussion}

   In this paper, we show that under the additional assumption of log-concavity, adaptive estimation of $\theta_0$ is possible with only one tuning parameter. 
   Our simulations suggest that the tuning parameter-free untruncated one step estimator may also be adaptive.  This demonstrates the usefulness of  log-concavity  assumption in semiparametric models in facilitating a simplified  estimation procedure. 
  It is natural to ask what happens if the above shape restriction fails to hold. For functionals of log-concave MLE type estimators, this question can be partially answered building on the log-concave projection theory developed by \cite{dumbreg}, \cite{theory}, \cite{xuhigh}, and \cite{barber2020}. See \cite{laha2019} for discussion of the case when the log-concavity assumption is violated in our model $\mP_0$.
  In particular, it can be shown that, even if $f\notin\mP_0$, as long as $f$ is symmetric about  $\th_0$,  the MLE and the truncated one step estimators are still consistent under mild conditions.

\section*{Acknowledgement}  
The author is  grateful to Jon Wellner for his help.
\bibliographystyle{natbib}
\bibliography{location_estimation}

\FloatBarrier
 \appendix
 \section*{Appendix}
 \setlength{\abovedisplayskip}{4pt}
\setlength{\belowdisplayskip}{4pt}
The appendix is organized as follows.
Appendices~\ref{app: proof of Theorem 1}, \ref{app: proof of proposition 1}, and \ref{app: proof of Theorem 2} contain the proofs for the one step estimators, where Appendices~\ref{app: mle lemma}, \ref{app: existence theorem}, and \ref{app: rate results} contain the proofs for the MLE. The proof of the main theorem is presented first, followed by the auxiliary lemmas required for the proof. Some common technical facts, which are used repeatedly in the proofs, are listed at the end in
Appendix~\ref{sec: technical facts}.
Appendix~\ref{app: stone} contains details on the selected tuning parameters for  \citeauthor{stone} and \citeauthor{beran}'s estimators.

 Before proceeding any further, we introduce some new notations and terminologies.
  For $i=1,\ldots,n$, consider  the pseudo-observations $Z_i=X_i-\theta_0$. Note that, if the $X_i$'s have density $ f_0$, then the $Z_i$'s have density $g_0$, and distribution function $G_0$. We will denote the log-densities corresponding to $\hts$, $\hnss$, $\widehat{g}_{\theta}$ and  ${\hn}^{geo,sym}$ by $\hlns$, $\hln^{sym, sm}$, $\widehat{\psi}_{\theta}$ and $\hln^{geo, sym}$, respectively.
   As usual, $\hlnsp$, $(\hln^{sym, sm})'$, $\widehat{\psi}_{\theta}'$,  and $(\hln^{geo, sym})'$ will denote the corresponding right derivatives. We remark in passing that there is nothing special about the right derivative, and any $L_1$ derivative would have worked. However, we fix one specific version  to avoid future confusion. We denote the distribution functions of $\hts$, $\hnss$, and  ${\hn}^{geo,sym}$ by $\Hsm$, $\Gs$, and $\tilde G_n^{geo, sym}$, respectively. 
 
 The  empirical process of the $X_i$'s will be denoted by $\mathbb G_n=\sqn(\Fm-F_0)$. 
 For any  function $h:\RR\mapsto\RR$, and a measure $Q$ on $\RR$, we write
$Q h:=\int_{\RR} h dQ$ provided $h$ is integrable with respect to $Q$.
 Suppose $\mathcal H$ is a  class of $Q$-measurable functions. We denote by
$ \|Q\|_{\mathcal H}$ the supremum $\sup_{h\in\mathcal H}\abs{Qh}$. 
For the sake of simplicity, we will denote $\td=\theta_0-\bth$ and $\delta_n=\th_0-\hthm$ in our proofs.

For a measure $P$ on $\RR$, we define the $L_{P,k}$ norm of the function $h$ as
\[\|h\|_{P,k}=\lb\edint |h(x)|^k dP(x)\rb^{1/k},\quad k\geq 1.\]
For any class of functions $\mathcal H$, we will denote
\[\|\mathcal H\|_{P,k}=\sup_{h\in\mathcal H}\|h\|_{P,k}.\]
For two distribution functions $F_1$ and $F_2$ with densities $f_1$ and $f_2$, the total variation distance between $F_1$ and $F_2$ is given by
$d_{TV}(F_1,F_2)=\|f_1-f_2\|_1/2$.
   We define the 
Wasserstein distance between two measures $\mu$ and $\nu$ on $\RR$ by
\begin{equation}\label{def: Wasserstein distance}
d_W(\mu,\nu)=\edint|F(x)-G(x)|dx,
\end{equation}
   where $F$ and $G$ are the distribution functions corresponding to $\mu$ and $\nu$ respectively. 
  This representation of $d_W(\mu,\nu)$ follows from  \cite{villani2003}, page $75$. By an abuse of notation, sometime we will denote the above distance by $d_W(F,G)$ as well. Suppose $\e>0$. For any class of functions $\mathcal H$ and a norm $\|\cdot\|$,  the bracketing entropy $N_{[\ ]}(\epsilon, \mathcal H,\| \cdot\|)$ is as in Definition 2.1.6, page 83 of \cite{wc}. The covering number $N(\epsilon, \mathcal H,\| \cdot\|)$ is as defined in page 83 of \cite{wc}.

  For two sets $A$ and $B$, $A\times B$ will represent the Cartesian product. For any set $A\subset \RR$, and $x\in\RR$, we use the usual notation $A+x$ to denote the translated set  $\{y+x\ :\ y\in A\}$. 
   The notation $\overline{A}$ will refer to the closure of the set $A$. For any function $h$,  $1_{[h(x)\leq C]}$ will denote the indicator function of the event $h(x)\leq C$. For any  set $A$, we let $1_A(x)$  be the indicator function of the event $x\in A$. As usual, we denote  by $\varphi$ the standard Gaussian density. 
   
   In some of our proofs, we will replace $\hnss$ by a more general mixture density which satisfies Condition~\ref{cond: hnss condition}.
   
   \begin{condition}[Condition for $\hnss$]\label{cond: hnss condition}
The density $\hn$ is symmetric about zero and satisfies
$\hn(x)=(g_{1n}(x)+g_{2n}(x))/2$, where $g_{1n}$ and $g_{2n}$ are log-concave densities. The densities $\hn$, $g_{1n}$, and $g_{2n}$ satisfy  Conditions~\ref{condition: on hn} and \ref{cond: hellinger rate}. Moreover, $\supp(g_{1n})=\supp(g_{2n})=\RR$, and the $p$ in Condition~\ref{cond: hellinger rate} is the same for $\hn$, $g_{1n}$, and $g_{2n}$.
\end{condition}

  We will later show in Lemma~\ref{lemma: L1 convergence: one step density estimators} that $\hts(\bth\pm\mathord{\cdot})$ and $\hnss$ satisfy Condition~\ref{condition: on hn}, and in Lemma~\ref{lemma: os: helli: hnss}, we will show that these densities satisfy Condition~\ref{cond: hellinger rate} with $p=1/5$. Since $\hts(\bth\pm\mathord{\cdot})$ is the convolution of two log-concave densities, it is log-concave. That  $\hnss$ satisfies Condition~\ref{cond: hnss condition} follows immediately from the above results.

We will frequently use the fact that if $f$ is a log-concave density, then $f>0$ on $J(F)=\{x\in\RR: 0<F(x)<1\}$ \citep[cf. Theorem 1(iv) of][]{dumbgen2017}. Therefore,  $\iint(\dom(\log f))=\iint(\supp(f))=\iint(J(F))$.  As a consequence,  $F^{-1}$ is  strictly increasing, and differentiable with derivative $1/f(F^{-1}(t))$ on $(0,1)$ by Fact~\ref{fact: bobkov big}. Also, a log-concave $f$ is thus  continuous on $\iint(J(F))$ by Theorem 10.1 of \cite{rockafellar}. When $f_0\in\mP_0$, furthermore, $f_0$ and $g_0$ are absolutely continuous  on $\RR$ by Theorem 3 of \cite{huber}.
Now we list below some  useful facts about log-concave densities. 

 \begin{fact}[Lemma 1 of \cite{theory}]
 \label{fact: Lemma 1 of theory paper}
 If $f$ is a univariate log-concave density, then there exists $\alpha>0$ and $\beta\in\RR$ so that  $f(x)\leq e^{-\alpha |x|+\beta} $.
 \end{fact}

 \begin{fact}\label{fact: f/F}
 If $f$ is log-concave, then $f/F$ is non-increasing on $J(F)$ and $f/(1-F)$ is non-decreasing on $J(F)$.
 \end{fact}
 
 \begin{proof}
This is a well known fact about log-concave densities. See Theorem 1 and Corollary 2 of \cite{bagnoli2005} or \cite{dumbgen2017}.
 \end{proof}
 
 The following two facts will be very useful to lower bound $\hn$ on $[-\xin,\xin]$.
 
  \begin{fact}\label{fact: f grtr than F: base}
 Suppose $\hn$ is a log-concave density satisfying Condition~\ref{condition: on hn}.   Then for any $x\in\RR$, 
 \[\hn(x)\geq \omega_n\min(\tilde{G}_n(x), 1-\tilde{G}_n(x)), \]
 where $\omega_n\geq 0$   satisfies $\omega_n\to_p \omega_0>0$.  Here $\omega_0>0$ is a constant depending only on $g_0$. 
 \end{fact}
 
  \begin{proof}
  If $\tilde G_n(x)$ is zero or one, then the statement  trivially holds. Therefore, we assume $x\in J(\tilde G_n)$, i.e. $0<\tilde G_n(x)<1$. For $q\in(0,1/2)$, Fact~\ref{fact: f/F} implies
  \[\hn(\tilde G_n^{-1}(q))\geq 2q\slb\hn(\tilde G_n^{-1}(1/2))\srb\geq 2\min(q,1-q)\hn(\tilde G_n^{-1}(1/2)).\]
 On the other hand, for $q\in(1/2,1)$,  Fact~\ref{fact: f/F} implies 
 \[\hn(\tilde G_n^{-1}(q))\geq 2(1-q)\hn(\tilde G_n^{-1}(1/2))\geq 2\min(q,1-q)\hn(\tilde G_n^{-1}(1/2)).\] Because $\tilde G_n(x)\in(0,1)$, replacing $q$ by $\tilde G_n(x)$  we obtain that
 \begin{align}\label{infact: f grtr than F}
  \hn(x)\stackrel{(a)}{=}\hn(\tilde G_n^{-1}(\tilde G_n(x)))\geq  2\hn(\tilde G_n^{-1}(1/2))\min\slb\tilde G_n(x),1-\tilde G_n(x)\srb.   
 \end{align}
 Here (a) uses the fact that $\hn(x)>0$ which follows since $x\in\iint(J(\tilde G_n))=\iint(\dom(\hln))$. 
 The rest of the proof follows setting $\omega_n=2\hn(\tilde G_n^{-1}(1/2))$, which converges in probability to $\omega_0=2g_0(0)$ by Condition~\ref{condition: on hn} and Fact~\ref{fact: consistency of the quantiles}.
  
%   Note that $0\in \iint(\dom(\psi_0))$. Since  $d_{TV}(G_0,\tilde G_n)\to 0$ by Condition~\ref{condition: on hn}, it an be shown $P(0<\tilde G_n(0)< 1)\to 1$. Therefore, $\hn(0)/\tilde G_n(0)$ and $\hn(0)/(1-\tilde G_n(0))$ are well defined with probability tending to one.  
%   Then  by Fact~\ref{fact: f/F}, it follows that for  $x\leq 0$, $\hn(x)/\tilde G_n(x)\geq \hn(0)/\tilde G_n(0)$. Therefore, for $x\leq 0$, we can write
% \[\hn(x)\geq \tilde G_n(x)\frac{\hn(0)}{\tilde G_n(0)}\geq \min(\tilde G_n(x), 1-\tilde G_n(x))\frac{\hn(0)}{\tilde G_n(0)}.\]
%  Similarly for $x\geq 0$, we can show that
% \[{\hn(x)}\geq(1-\tilde G_n(x))\frac{\hn(0)}{1-\tilde G_n(0)}\geq \min(\tilde G_n(x), 1-\tilde G_n(x))\frac{\hn(0)}{1-\tilde G_n(0)}.\]
% Let us set
% \[ \omega_n=\min\lb \frac{\hn(0)}{\tilde G_n(0)}, \frac{\hn(0)}{1-\tilde G_n(0)}\rb.\]
%  In particular, if $0\in\supp(\hn)$ for all $n\geq 1$ then $\omega_n$ is finite for all $n\geq 1$. 
% Condition~\ref{condition: on hn} implies that  $\omega_n\to_p 2g_0(0)$, which completes the proof.
\end{proof}

 \begin{fact}\label{fact: f grtr than F}
 Suppose either $\hn$ satisfies  Condition~\ref{cond: hnss condition}, or $\hn$ is a log-concave density satisfying Condition~\ref{condition: on hn}. Then the assertions of Fact~\ref{fact: f grtr than F: base}  hold.
%  \begin{enumerate}
%      \item $\hn$ is a log-concave density and satisfies Condition~\ref{condition: on hn}. Also $0\in\iint(\supp(\hn))$.
%      \item $\hn(x)=(g_{1n}(x)+g_{2n}(x))/2$ where $g_{1n}$ and $g_{2n}$ are log-concave densities and satisfies Condition~\ref{condition: on hn}. Also $0\in\iint(\supp(g_{in}))$ for $i=1,2$.
%  \end{enumerate}
%   Then for any $x\in\RR$, 
%  $\hn(x)\geq \omega_n\min(\tilde{G}_n(x), 1-\tilde{G}_n(x)) $
%  where $\omega_n\to_p \omega_0$, and $\omega_0>0$ is a constant depending on $g_0$. If $\bth\as \th_0$, then $\omega_n\as \omega_0$.
 \end{fact}
 
  \begin{proof}
  If $\hn\in\mathcal{LC}$, the proof follows from Fact~\ref{fact: f grtr than F: base}.
%   If $\tilde G_n(x)=0$ or one, then the statement is trivially. Therefore, we suppose $x\in\{x: 0<\tilde G_n(x)<1\}$.
%   Suppose $\hn$ is log-concave. Then  by Fact~\ref{fact: f/F}, it follows that for  $x\leq 0$, $\hn(x)/\tilde G_n(x)\geq \hn(0)/\tilde G_n(0)$. Note that the denominator is non-zero because $0\in\iint(\supp(\hn))$. Therefore, for $x\leq 0$, we can write
% \[\hn(x)\geq \tilde G_n(x)\frac{\hn(0)}{\tilde G_n(0)}\geq \tilde G_n(x)(1-\tilde G_n(x))\frac{\hn(0)}{\tilde G_n(0)}.\]
%  Similarly for $x\geq 0$, we have
% \[{\hn(x)}\geq(1-\tilde G_n(x))\frac{\hn(0)}{1-\tilde G_n(0)}\geq \tilde G_n(x)(1-\tilde G_n(x))\frac{\hn(0)}{1-\tilde G_n(0)}.\]
% Note that $\tilde G_n(0)<1$ because $0\in\iint(\supp(\hn))$.
% Therefore, letting 
% \[\tilde a_n=\min\lb \frac{\hn(0)}{\tilde G_n(0)}, \frac{\hn(0)}{1-\tilde G_n(0)}\rb,\]
% we have
% \[\hn(x)\geq \tilde  a_n \tilde G_n(x)\slb 1-\tilde G_n(x)\srb\geq \frac{\tilde a_n}{2}\min\slb\tilde G_n(x), 1-\tilde G_n(x)\srb.\]
% By Condition~\ref{condition: on hn}, $\tilde{a}_n\to_p 2g_0(0)$, where the convergence is almost sure if $\bth\as\theta_0$. The proof then  follows  taking $w_n=\tilde a_n/2$ and $w_0= g_0(0)$.
Therefore we consider the case when $\hn$ satisfies Condition~\ref{cond: hnss condition}. Since the component densities $g_{1n}$ and $g_{2n}$ in  Condition~\ref{cond: hnss condition} are log-concave, Fact~\ref{fact: f grtr than F: base} applies to them. Denote by $G_{1n}$ and $G_{2n}$ the corresponding distribution functions. Equation \ref{infact: f grtr than F} in the proof of Fact~\ref{fact: f grtr than F: base} implies
 \[g_{1n}(x)+g_{2n}(x)\geq b_n \lbs  G_{1n}(x)\slb 1- G_{1n}(x)\srb+ G_{2n}(x)\slb 1- G_{2n}(x)\srb\rbs,\]
 where 
 \[b_n=2\min\lb g_{1n}(G_{1n}^{-1}(1/2)),  g_{2n}(G_{2n}^{-1}(1/2))\rb\to_p 2g_0(0)\]
by the fact that $g_{1n}$ and $g_{2n}$ satisfy Condition~\ref{condition: on hn} and Fact~\ref{fact: consistency of the quantiles}.  Fact~\ref{fact: concave-convex}  implies that 
\begin{align*}
 \MoveEqLeft   G_{1n}(x)\slb 1- G_{1n}(x)\srb+ G_{2n}(x)\slb 1- G_{2n}(x)\srb\\
    \geq &\ \min \lb \frac{G_{1n}(x)+G_{2n}(x)}{2}, 1-  \frac{G_{1n}(x)+G_{2n}(x)}{2}\rb\\
    =&\ \min(\tilde G_n(x), 1-\tilde G_n(x)),
\end{align*}
where $\tilde G_n$ is the distribution function corresponding to $\tilde g_n=(g_{1n}+g_{2n})/2$.
Letting  $\omega_n=b_n/2$, we have
 \[\hn(x)\geq \frac{b_n}{2} \min(\tilde G_n(x), 1-\tilde G_n(x)),\]
which completes the proof.
  \end{proof}

 \section{Proof of Theorem \ref{theorem: main: one-step: full}}
 \label{app: proof of Theorem 1}
% \textcolor{red}{As in the proof of Lemma~\ref{lemma: L1 convergence: one step density estimators}, one can show that it suffices to prove Theorem \ref{theorem: main: one-step} when $\bth\as\th_0$. Therefore, in what follows, we assume that $\bth\as\th_0$.}
We first argue that it suffices to prove the theorem only for the case when $\eta_n$ equals $Cn^{-2p/5}$. 
In the latter case, we would show $\sqn(\hth-\th_0)\to_d N(0,\I^{-1})$ for $\eta_n=Cn^{-2p/5}$ when $H(\hn, g_0)=O_p(n^{-p})$. Note that for any $p'\in(0,p]$,  $H(\hn, g_0)=O_p(n^{-p'})$ trivially holds since
$H(\hn, g_0)=O_p(n^{-p})$. Therefore, replacing $p$ by $p'$ in what we  just proved,  $\sqn(\hth-\th_0)\to_d N(0,\I^{-1})$  would follow identically for $\eta_n=Cn^{-2p'/5}$. Thus, it is enough to consider the case when $\eta_n=Cn^{-2p/5}$.

From  \eqref{def: one-step estimator: truncated} we obtain that
\begin{align*}
-(\hth-\bth)= \dint_{\bth-\xin}^{\bth+\xin}\dfrac{\hln'(x-\bth)}{\hin(\eta_n)}d\Fm(x)=\dint_{-\xin}^{\xin}\dfrac{\hln'(z)}{\hin(\eta_n)}d\Fm(z+\bth).
\end{align*}
Denoting $\td=\th_0-\bth$, we observe that the above expression  writes as
\begin{align}\label{theorem: main: one-step: first split}
\MoveEqLeft\underbrace{\dint_{-\xin}^{\xin}\dfrac{\tp(z)-\psp'(z-\td)}{\hin(\eta_n)}d(\Fm(z+\bth)-F_0(z+\bth))}_{T_{1n}}\nonumber\\
&\ +\underbrace{\dint_{-\xin}^{\xin}\dfrac{\tp(z)}{\hin(\eta_n)}\lb f_0(z+\bth)-g_0(z)\rb dz
 }_{T_{2n}}+\underbrace{\dint_{-\xin}^{\xin}\dfrac{\tp(z)-\psp'(z)}{\hin(\eta_n)}g_0(z)dz}_{T_{3n}}\nonumber\\
 &\ +  \underbrace{\dint_{-\xin}^{\xin}\dfrac{\psp'(z)}{\hin(\eta_n)}g_0(z)dz}_{T_{4n}} + \underbrace{\dint_{-\xin}^{\xin}\dfrac{\psp'(z-\td)}{\hin(\eta_n)}d(\Fm(z+\bth)-F_0(z+\bth))}_{T_{5n}}
\end{align}
Observe that  $T_{3n}$ and $T_{4n}$ vanish since $\tp$ and $\psi'_0$ are odd functions while $g_0$ is an even function. 
%   \[\sqn T_{5n}=\int_{\th_0-\xi_0}^{\th_0+\xi_0}\dfrac{\psi'_0(x-\th_0)}{\Ig(\eta_n)}d\mathbb{G}_n(x)+o_p(1).\]

The proof of Theorem~\ref{theorem: main: one-step: full} has three main steps. The first step uses Donsker Theorem to show that the empirical process term  $T_{1n}$ is $o_p(n^{-1/2})$. The term $T_{2n}$ accounts for the bias due to the use of $\bth$ instead of the true center $\theta_0$ in the construction of the scores. The second step of the proof shows that the order of $T_{2n}$ is same as $\td=\theta_0-\bth$. In particular, we will show that $T_{2n}=-\td(1+o_p(1))$.
   Since $\td=O_p(n^{-1/2})$, the above two steps lead to
  \[\sqn(\bth-\hth)=o_p(1)+\sqn(\bth-\th_0)+\sqn T_{5n}.\]
  The third  step of the proof shows that the term $\sqn T_{5n}$ is asymptotically normal
  with variance $\I^{-1}$.
  A rearrangement of the terms in the above display then establishes the desired asymptotic convergence of $\sqn(\hth-\theta_0)$. The rest of the proof is devoted to proofs of the above-mentioned three steps. 
%  Therefore it follows that
%\[\sqn(\hth-\th_0)=\sqn(1-\gamma_{\eta_n})(\bth-\th_0)+ Z_{0}+o_p(1),\]
%where
%$\gamma_{\eta_n}$ is as defined in \eqref{definition of gamma},  and
% $Z_{0}\sim N(0,\Ig(\eta_n)^{-1})$.
 
  \subsubsection*{\textbf{First step: asymptotic negligibility of $\sqn T_{1n}$:}}
 First, let us denote $\mathcal T_n= [\bth-\xin,\bth+\xin]$.
  Recall that in Section \ref{sec: Preliminaries} we denoted the empirical process $\sqn(\Fm-F_0)$ by $\mathbb{G}_n$.
   Note that $\sqn T_{1n}$ also writes as
 \begin{align}\label{intheorem: T1n: representation}
 \sqn  T_{1n} = \sqn\dint_{\bth-\xi_n}^{\bth+\xi_n}\dfrac{\tp(x-\bth)-\php'(x)}{\hin(\eta_n)}d(\Fm-F_0)(x)=\edint\dfrac{ h_n(x)}{\hin(\eta_n)}d\mathbb{G}_n(x),
 \end{align}
 where by $h_n$ we denote the function
 \begin{equation}\label{inlemma: def: main: hn}
     h_n(x)=(\tp(x-\bth)-\phi_0'(x))1_{\mathcal T_n}(x),\quad x\in\RR.
 \end{equation}

 Because $\eta_n=O( n^{-2p/5})$, Lemma~\ref{lemma: tilde psi prime bound} implies
 \begin{equation}\label{inlemma: t1: bounding hln prime}
 \sup_{x\in\mathcal T_n}|\hln'(x-\bth)|=O_p(n^{p/5}).    
 \end{equation}
 Thus $\hln'$ restricted to the compact set $\mathcal T_n$ is bounded. We can extend the function $x\mapsto \hln'(x) 1_{\mathcal T_n}(x)$ to $\RR$ in a way such that the resulting function $\widehat u_n$ is still monotone and has the same bound. This can be done by setting $\widehat u_n$ to be $\hln'(\bth-\xin)$ and  $\hln'(\bth+\xin)$ on the intervals  $(-\infty,\bth-\xin]$ and $[\bth+\xin,\infty)$, respectively. Note also that we can replace $\hln'$ by $\widehat u_n$ in the definition of $h_n$, i.e.
 \begin{align}
     \label{inthm: 1: def: hn in }
     h_n(x)=(\widehat u_n(x-\bth)-{\phi}_0'(x))1_{\mathcal T_n}(x).
 \end{align}
  Let us denote  $M_n=Cn^{p/5}$ for some $C>0$ and define
\begin{equation}\label{intheorem: def: U n}
    \mathcal{U}_{n}(M_n)=\bigg\{u:\RR\mapsto[-M_{n},M_{n}]\ \bl\ \ u\text{ is non-increasing}\bigg\}.
\end{equation}
Since $\|\widehat u_n\|_{\infty}=O_p(n^{p/5})$, for sufficiently large $C$,  $\hat u_n(\mathord{\cdot}-\bth)\in \mathcal{U}_n(M_n)$ with high probability.
Now  define the class $\mathcal{H}_n(C)$ by
  \begin{align*}
 \mathcal{H}_n(C)=  \bigg\{h:\RR\mapsto  \RR\ \bl \ &\ h(x)=(u(x)-{\phi}_0'(x))1_{[r_1,r_2]}(x),\ u\in\mathcal{U}_{n}(M_n), \\
 &\
 \|h\|_{P_0,2}\leq Cn^{-2p/5}(\log n)^{3},\ \ \|h\|_\infty\leq M_n,\\
 &\ [r_1,r_2]\subset [\th_0-C\log n,\th_0+C\log n]\cap\iint(\dom(\phi_0)) \bigg \}.
 \end{align*}
 The notation $\mathcal H_n(C)$ does not depend on $M_n$ because $M_n=Cn^{p/5}$ is also a function of $C$.

%  that given any $\e>0$, here exist $C_{\e}>0$ so that 
% \[\liminf_{n\to\infty}P\lb\sup_{x\in[\bth-\xi_n,\bth+\xi_n]}|\tp(x-\bth)| \leq C_{\e} n^{ p/5}\rb\geq 1-\e.\]
% Letting $M_n= C_{\e} n^{p/5}$, we denote
% \[\mathcal D(\e)=\lbs \sup_{x\in[\bth-\xi_n,\bth+\xi_n]}|\tp(x-\bth)| \leq M_n\rbs.\]
% On the set $\mathcal D(\e)$ therefore, the first term of $h_n(x)$, i.e.  $\hln'(x-\bth)1_{[\bth-\xin,\bth+\xin]}(x)$ is bounded. 
We want to show that $h_n\in\mathcal H_n(C)$ with high probability for large $n$.
Note that
\[\sup_{x\in\mathcal T_n}|\phi'_0(x)|=\sup_{x\in [-\d-\xi_n,-\d+\xi_n]}|\psi'_0(x)|.\]
 Lemma~\ref{lemma: bound: psi knot prime} in conjunction with the fact that $\eta_n=O(n^{-2p/5})$ implies
 \begin{equation}\label{intheorem: boundedness of psp}
 \sup_{x\in \mathcal T_n}|\phi_0'(x)|=O_p(\log n).
 \end{equation}
 Thus \eqref{inlemma: t1: bounding hln prime} and \eqref{intheorem: boundedness of psp} imply $\|h_n\|_{\infty}=O_p(n^{p/5})$. Lemma~\ref{Lemma: L2 norm of hn} bounds the $L_{P_0,2}$ norm of $h_n$ entailing $\|h_n\|_{P_0,2}=O_p(n^{-2p/5}(\log n)^{3})$.
  Lemma~\ref{lemma: An inclusion} implies, on the other hand,
\[
\lim_{n\to\infty}  P\slb  \mathcal T_n\subset [\th_0-C\log n,\th_0+C\log n]\cap \iint(\dom(\phi_0))\srb= 1.
\]
Therefore, we conclude that given $t>0$, we can choose $C>0$ so large such that $P(h_n\in\mathcal{H}_n(C))>1-t$.

% \textcolor{red}{  Therefore, we obtain that
% \begin{equation}\label{inlemma: boundedness: hn}
% \limsup_n\sup_{x\in [\bth-\xi_n,\bth+\xi_n]}|h_n(x)|<M_{\eta_n} \quad a.s.
% \end{equation}
% }
% Let us define the monotone function $\widehat u_n:\RR\to$ so that $\widehat u_n=\tp$ on $[-\xi_n,\xi_n]$ but $\widehat u_n=\tp(-\xi_n)$ for $x<-\xi_n$ and $\widehat u_n=\tp(\xi_n)$ for $x>\xi_n$.
% Note that $\|\widehat u_n\|_{\infty}=\|\tp\|_\infty$
% We can similarly define $\overline{\phi}_0':\RR\mapsto\RR$ so that $\overline{\phi}_0'$ is non-increasing,
% \[\|\overline{\phi}_0'\|_\infty=O_p(\log n),\quad  \overline{\phi}_0'(x)=\phi_0'(x)\quad\text{ for       }x\in[\bth-\xi_n,\bth+\xi_n].\]

% We  want to establish that $|h_n(x)|$ is bounded by some number $M_{\eta_n}>0$ almost surely. To this end, we show that $\tp$ and $\php'$ are bounded on the intervals $[-\xi_n,\xi_n]$ and $[\bth-\xi_n,\bth+\xi_n]$ respectively.   Now turning to $\tp$, we observe that
% Lemma \ref{lemma: boundedness of tp} assures the existence of $c_{\eta_n}>0$ such that on $[-\xin,\xin]$, the derivative $|\tp|$ is bounded by $c_{\eta_n}$.
%    Setting $M_{\eta_n}=c_{\eta_n}+c_{\eta_n}'$,

 Theorem~2.7.5 of \cite{wc} (pp. $159$) states that there exists an absolute constant $C'>0$ so that for any $\e>0$ and any probability measure $\RR$ on the real line,  
   \begin{equation}\label{inlemma: finite entropy increasing}
\log N_{[\ ]}(\e,\mathcal{U}_{n}(M_n),L_2(Q))\leq C'M_n\e^{-1}.
   \end{equation}
    On the other hand, using Theorem~2.7.5 of \cite{wc}, it can also be shown that the
 class $\mathcal{F}_I$ of all indicator functions of  the form $1_{[z_1,z_2]}$, where $z_1\leq z_2$ with $z_1,z_2\in\RR$,  satisfies
   \begin{equation}\label{inlemma: finite entropy indicator functions}
 \log N_{[\  ]}(\e,\mathcal{F}_I,L_2(Q))\leq C' 2\e^{-1}.
   \end{equation}
Using \eqref{inlemma: finite entropy increasing} and \eqref{inlemma: finite entropy indicator functions} we derive that
\[\log N_{[\ ]}(\e,\mathcal{H}_n(C),L_2(P_0))\lesssim {M_n}{\e}^{-1}.\]

For $x<1$, the bracketing integral
\begin{align*}
    \mathcal{J}_{[\ ]}(x,\mathcal{H}_n(C),L_2(P_0))=&\ \dint_{0}^{x}\sqrt{1+\log N_{[\ ]}(\e,\mathcal{H}_n(C),L_2(P_0))}d\e\\
    \lesssim &\  2M_n\dint_{0}^{x/M_n}\e^{-1/2}d\e,
\end{align*}
which equals $\sqrt{x M_n}$. Let us also denote $K_n=Cn^{-2p/5}(\log n)^{3}$. Note that \\
 \[\|\mathcal{H}_n(C)\|_{P_0,2}=\sup_{h\in\mathcal{H}_n(C)}\|h\|_{P_0,2}=K_n.\]
Then from Fact~\ref{fact: empirical process of m-p's}
it follows that
\begin{align*}
   E\left[\|\mathbb{G}_n\|_{\mathcal{H}_n(C)}\right]\lesssim &\   \mathcal{J}_{[\ ]}(K_n,\mathcal{H}_n(C),L_2(P_0))\slb 1+\frac{  \mathcal{J}_{[\ ]}(K_n,\mathcal{H}_n(C),L_2(P_0))}{K_n^2\sqn}M_n\srb
   \end{align*}
   which is bounded by a constant multiple of $ \sqrt{K_nM_n}+K_n^{-1}M_n^2n^{-1/2}$. Since $K_n=Cn^{-2p/5}(\log n)^3$ and $M_n=Cn^{p/5}$, 
%    \[
%   \sqrt{K_nM_n}+K_n^{-1}M_n^2n^{-1/2} \lesssim (\log n)^{3/4}n^{-p/10}+(\log n)^{-3/2}n^{-1/2}.
% \]
\begin{equation}
\sqrt{K_n M_n} = \sqrt{C^2 n^{p/5} n^{-2p/5} (\log n)^{3}} = Cn^{-p/10} (\log n)^{3/2} = o(1).
\end{equation}
On the other hand,
\begin{align*}
\frac{M_n^2}{K_n\sqn} = \frac{C^2n^{2p/5}}{Cn^{-2p/5} (\log n)^{3}\sqrt{n}}
= \frac{Cn^{4p/5} n^{-1/2}}{(\log n)^3} 
= \frac{C n^{(8p-5)/10}}{(\log n)^{3}} 
= o(1),
\end{align*}
where the last step follows because $p\leq 1/2 < 5/8$ by Condition \ref{cond: hellinger rate}. Hence, we have shown that 
\[ E\left[\|\mathbb{G}_n\|_{\mathcal{H}_n(C)}\right]\lesssim  \sqrt{K_nM_n}+K_n^{-1}M_n^2n^{-1/2}=o(1).\]
Now fix $t'>0$  and $\xi>0$. We can choose $C$ so large such that $P( h_n\notin \mathcal{H}_n(C))<\xi/2$. Therefore
\begin{align*}
  \MoveEqLeft  P\lb \edint h_n(x)d\mathbb{G}_n(x)>t' \rb\\
  \leq &\  P\lb \edint h_n(x)d\mathbb{G}_n(x)>t' , h_n\in \mathcal{H}_n(C)\rb+P\lb h_n\notin \mathcal{H}_n(C)\rb\\
  \stackrel{(a)}{\leq} &\  E\lbt\sup_{h\in\mathcal{H}_n(C)}\bl\edint h(x)d\mathbb{G}_n(x)\bl\rbt/t'+\xi/2\\
  =&\ o(1)/t'+\xi/2,
\end{align*}
which is less than $\xi$ for sufficiently large $n$. Here (a) follows from
 Markov's inequality. Since $t'$ and $\xi$ are arbitrary, we conclude that $\int h_n d\mathbb{G}_n$ is $o_p(1)$. 
% , thereby establishing that $\mathcal{H}_n(C)$ is $P_0$-Donsker. 
 Finally 
an application of Lemma~\ref{lemma: consistency of FI} leads to $\hin(\eta_n)\to_p\I$, and thus from \eqref{intheorem: T1n: representation},  $\sqn T_{1n}=o_p(1)$ follows.

%%%%%%%%%%%%%%%%%%%%%%%%%% T_2 %%%%%%%%%%%%%%%%%%%%%%%%%%%%%%%%%%%%%%%%

% Now we turn our attention to $T_{2n}$. Here we aim to prove \eqref{intheorem: one step: T2 convergence}. 

\subsubsection*{\textbf{Second step: asymptotic limit of $T_{2n}/\td$:}}
Let us define $\mathcal A_n=[-\xin,\xin-\td]$,
Observe that  $T_{2n}/\td$ can be written as
\begin{align}\label{intheorem:t2n:representation}
\dint_{-\xin}^{\xin}\dfrac{\tp(z)}{\hin(\eta_n)}\dfrac{\lb g_0(z-\td)-g_0(z)\rb}{\td} dz\ =&\ \dint_{-\xin}^{\xin}\dfrac{\tp(z)}{\hin(\eta_n)}\dfrac{\dint_{z}^{z-\td}g_0'(t)dt}{\td} dz\\
= &\ -\dint_{\RR}\underbrace{1_{\mathcal A_n}(t)g_0'(t)\dfrac{\dint_{t}^{t+\td}\tp(z)dz}{\td\hin(\eta_n)}}_{b_n(t)}dt,\nn
\end{align}
where the last equality  follows by Fubini's Theorem since $g_0$ is absolutely continuous. 
% \textcolor{blue}{\textbf{This part will be removed:}
% Also since $g_0'(t)=\psi_0'(t) g_0(t)$, we have
% \begin{align*}
% \MoveEqLeft  \dint_{-\xin-\td}^{\xin}|g_0'(t)|\dfrac{\left |\dint_{t}^{t+\td}\tp(z)dz \right |}{|\td|\hin(\eta_n)}dt\\
% \leq &\  \dint_{-\xin-\td}^{\xin} |\psi_0'(t)|\sqrt{g_0(t)}\frac{\sqrt{g_0(t)}}{\sqrt{\hn(t)}}  \frac{|\tp(t)|\sqrt{\hn(t)}}{\hin(\eta_n)}dt\\
% &\ +\dint_{-\xin-\td}^{\xin} |\psi_0'(t)|\frac{g_0(t)}{\sqrt{g_0(t+\td)}}\frac{\sqrt{g_0(t+\td)}}{\sqrt{\hn(t+\td)}}  \frac{|\tp(t+\td)|\sqrt{\hn(t+\td)}}{\hin(\eta_n)}dt\\
% =&\  \dint_{-\xin-\td}^{\xin} \underbrace{|\psi_0'(t)|\sqrt{g_0(t)}\lb\frac{\sqrt{g_0(t)}-\sqrt{\hn(t)}}{\sqrt{\hn(t)}}\rb \frac{|\tp(t)|\sqrt{\hn(t)}}{\hin(\eta_n)}}_{T_{21,n}(t)}dt\\
% &\ + \dint_{-\xin-\td}^{\xin} \underbrace{|\psi_0'(t)|\sqrt{g_0(t)} \frac{|\tp(t)|\sqrt{\hn(t)}}{\hin(\eta_n)}}_{T_{22,n}(t)}dt\\
% &\ + \dint_{-\xin-\td}^{\xin} \underbrace{|\psi_0'(t)|\frac{{g_0(t)}}{\sqrt{g_0(t+\td)}}\lb\frac{\sqrt{g_0(t+\td)}-\sqrt{\hn(t+\td)}}{\sqrt{\hn(t+\td)}}\rb \frac{|\tp(t+\td)|\sqrt{\hn(t+\td)}}{\hin(\eta_n)}}_{T_{23,n}(t)}dt\\
% &\ +\dint_{-\xin-\td}^{\xin} \underbrace{|\psi_0'(t)|\frac{{g_0(t)}}{\sqrt{g_0(t+\td)}} \frac{|\tp(t+\td)|\sqrt{\hn(t+\td)}}{\hin(\eta_n)}}_{T_{24,n}(t)}dt.
% \end{align*}
% }

% \begin{equation}\label{intheorem: def: bn:}
%     b_{n}(t)=1[t\in\mathcal A_n]\dfrac{\dint_{t}^{t+\td}\tp(z)dz }{\td\hin(\eta_n)}g_0'(t),
% \end{equation}
Note that
\eqref{intheorem:t2n:representation} implies $T_{2n}=-\int_{\RR} b_n(t)dt$. The following lemma, which is proved in Appendix~\ref{sec: key lemma: main}, establishes $T_{2n}\to_p -1$, thus completing the proof of the second step.

\begin{lemma}\label{lemma: main: key lemma for T2}
Under the set up of Theorem~\ref{theorem: main: one-step: full},
$\mathbb Y_n\equiv\edint b_n(t)dt\to_p 1$ where $b_n(t)$ is as defined in \eqref{intheorem:t2n:representation}.
\end{lemma}

\subsubsection*{\textbf{Third step: showing the asymptotic normality of $T_{5n}$:}}
 A change of variable leads to
\begin{align}\label{intheorem: T3n split new}
\sqn T_{5n}=&\ \sqn\int_{\bth-\xin}^{\bth+\xin}\dfrac{\psp'(x-\th_0)}{\hin(\eta_n)}d(\mathbb{F}_n-F_0)(x)\nonumber\nn\\
=&\  \int_{-\infty}^{\infty}\dfrac{\phi'_0(x)}{\hin(\eta_n)}d\mathbb{G}_n(x) -\dint_{C_n} \dfrac{\phi_0'(x)}{\hin(\eta_n)}d\mathbb{G}_n(x),
\end{align}
where 
$C_n=(-\infty,\bth-\xin]\cup[\bth+\xin,\infty]$.
The central limit theorem yields
\[\edint\phi'_0(x)d\mathbb{G}_n(x)=\sum_{i=1}^n\frac{\phi_0'(X_i)-E[\phi_0'(X_i)]}{\sqn}\to_d N(0,\I).\]
Then from Lemma~\ref{lemma: consistency of FI} and  Slutsky's theorem it follows that 
\[ \int_{-\infty}^{\infty}\dfrac{\phi'_0(x)}{\hin(\eta_n)}d\mathbb{G}_n(x)\to_d N(0,\I^{-1}).\]
% because  the central limit theorem and \eqref{definition: I(eta)} imply   that the first term on the above display is $O_p(1)$.
Thus it suffices to show that
the second term on the right hand side of \eqref{intheorem: T3n split new} is $o_p(1)$. To that end, observe that  $1-1_{C_n}=1_{C_n^c}$ belongs to  the class of all indicator functions of  the form $1_{[z_1,z_2]}$,  where $z_1\leq z_2$ with $z_1,z_2\in\RR$. Since the latter class is Donsker by \eqref{inlemma: finite entropy indicator functions},  Theorem 2.1 of \cite{epindex} entails that  the second term on the right hand side of \eqref{intheorem: T3n split new} is of order $o_p(1)$ provided 
\[\hin(\eta_n)^{-2}\edint 1_{C_n}(x)\phi'_0(x)^2f_0(x)dx\to_p 0.\]
Since $\hin(\eta_n)\to_p\I>0$ by Lemma~\ref{lemma: consistency of FI}, we only need to show that the integral in the last display is $o_p(1)$. Because $\I<\infty$, Fact~\ref{fact: condition for integrability}  implies that given any $\e>0$, there exists $\sigma>0$ so that 
$P_0(\mathcal B)<\sigma$ implies $\int_{\mathcal B}\phi_0'^2(x)f_0(x)dx<\e$ for any $P_0$-measurable set $\mathcal{B}\subset\RR$. Thus the proof follows if we can  show that $\int_{\mathcal C_n}f_0(x)dx=o_p(1)$. To that end, observe that
\begin{align*}
   \dint_{C_n}f_0(x)dx=&\ 1-F_0(\bth+\xin)+F_0(\bth-\xin) \\
   \to_p &\  1-F_0(\th_0+G_0^{-1}(1))+F_0(\th_0+G_0^{-1}(0))
\end{align*}
by continuous mapping theorem because (a) $\bth\to_p\theta_0$, (b) $\xin\to_p G_0^{-1}(1)$ by Lemma~\ref{lemma: xi: xi goes to 1}, and (c) $F_0$ is continuous. Since
 $\theta_0+G_0^{-1}(1)=F_0^{-1}(1)$ and $\theta_0+G_0^{-1}(0)=F_0^{-1}(0)$, the proof follows.
% Finally, since $G_0$ is continuous,  continuous mapping yields
% \[1-F_0(\bth+\xin)+F_0(\bth-\xin)\to_p 1-F_0(F_0^{-1}(1))+F_0(F_0^{-1}(0))=0.\]
% Therefore, for any $\sigma>0$,  $P(\int_{C_n}f_0(x)dx<\sigma)\to 1$ . In that case, Fact~\ref{fact: condition for integrability} and the integrability of $(\phi'_0)^2$ with respect to the induced measure $F_0$ imply that given any $\e>0$, 
% \[P\slb \dint_{C_n}\phi'_0(x)^2 f_0(x)dx<\e\srb\to 0\quad\text{as}\ n\to\infty,\]

% To this end note that since $\psp'$ is non-increasing, \eqref{convergence of inverses : one-step: 1}, \eqref{convergence of inverses : one-step: 2}, \eqref{intheorem: boundedness of psp}, and Lemma~\ref{lemma: consistency of Fisher information} entail that for all sufficiently large $n$,
% \[\edint 1_{C_n}(x)\dfrac{\psp'(x-\th_0)^2}{\hin(\eta_n)^2}dF_0(x)\leq \dfrac{\psp'(\pG^{-1}(\eta_n/2))^2}{\Ig^2(\eta_n)}|\bth-\th_0|\lb\sup_x f_0(x)\rb\quad a.s.\] 
%  Since $f_0\in\mP_1$, from the definition of $\mP_1$ in \eqref{definition: model P1} and the strong consistency of $\bth$ we obtain that the term on the right hand side of the last display
%  converges to zero almost surely. 
\hfill $\Box$
 \
 
 \subsection{Proof of key lemmas for Theorem~\ref{theorem: main: one-step: full}}
\label{sec: key lemma: main}

% \subsubsection{Proof of Lemma~\ref{lemma: main: key lemma for T2}}
\begin{proof}[Proof of Lemma~\ref{lemma: main: key lemma for T2}]
% Fix $\e>0$,
%  Using Lemma \ref{lemma: bound: xi n}, and \ref{lemma on g/g(-)}, we can show that the random sequence
%  \[\mathbb{W}_n=\frac{\xin}{\log n}+ \sup_{\mathcal A_n}\frac{g_0(t)}{g_0(t+\td)}\]
%  is tight. Therefore, by Prohorov's theorem \citep[\  ][Theorem 1.5]{shorack2000} every sequence of $n$ has a further subsequence $n'$ along which $\mathbb W_{n'}\to_d W$ for some random variable $W$. 
 
%  there exists $C>0$ so that the set
%  \[\mathcal D_n(C)=\lbs [-\xin-|\td|,\xin+|\td|]\subset \dom(\psi_0), |\xin|<C\log n,\leq C\rbs\]
%  has probability at least $1-\e$ for sufficiently large $n$.
% Next note that  since $\bth\to_P\theta_0$,  there exists a subsequence $(n_k)_{k\geq 1}$ such that $\bth\as 0$. 
% \textcolor{red}{Lemma~\ref{lemma: xi: tilde xi n} implies $\xin+|\td|<\tilde{\xin}=\tilde G_n^{-1}(1-\eta_n/2)$ with probability tending to one. Therefore with high probability, $\slb\inf_{x\in\mathcal A_n'}\hn(x)\srb^{-1}\leq \hn(\tilde{\xin})^{-1}$ which is $O_p(\eta_n^{-1})$ by Lemma~\ref{lemma: hn: lower bound  on hn}.
% The above, combined with the fact }
Recall that we defined $\mathcal A_n=[-\xin,\xin-\td]$ in the proof of Theorem~\ref{theorem: main: one-step: full}. Let us define
$\mathcal A'_n=[-\xin-2|\td|,\xin+2|\td|]$.
We also denote
\[\mathcal I_{1n}=\dint_{\mathcal A_n}\tilde \psi'_{n_k}(t)^2\tilde g_{n_k}(t)dt\quad\text{and}\quad \mathcal I_{2n}=\dint_{\mathcal A_n+\td}\tilde \psi'_{n_k}(t)^2\tilde g_{n_k}(t)dt.\]
First we will show that it suffices to consider almost sure convergence of $ \mathbb Y_n$ along some suitably chosen subsequence. We claim that given any subsequence of $\{n\}$, we can always obtain a further subsequence $\{n_k\}_{k\geq 1}$ so that  the set
\begin{align}\label{inlemma: def: mathcal M}
    \mathcal M=\lbs &\ \overline \theta_{n_k}\to_k\theta_0,\ \widehat I_{n_k}(\eta_{n_k})\to_k\I,\ \xi_{n_k}\to_k G_0^{-1}(1),\ \omega_{n_k}\to_k \omega_0,\  \xi_{n_k}\delta_{n_k}\to_k 0\nn\\
   &\ \frac{(\log {n_k})^2 H(\tilde g_{n_k},g_0)^2}{\inf_{x\in\mathcal A'_{n_k}}\tilde g_{n_k}(x)}\to_k 0,\ \lim_{k\to\infty}\mathcal I_{in_k }=\I\ \text{for}\ i=1,2,\\
   &\ \|\tilde{g}_{n_k}-g_0\|_\infty\to_k 0,\ \mathcal{A}_{n_k}' \subset \iint(\dom(\psi_0))\text{ for all sufficiently large }n_k \rbs\nn
\end{align}
has probability one, where $\omega_n$ and $\omega_0$ are as in Fact~\ref{fact: f grtr than F}. 
The claim follows directly by  Fact~\ref{fact: convergence in probability to convergence almost surely}  noting
\begin{compactenum}[label=(\alph*)]
    \item $\bth\to_p\th_0$.
    \item  $\hin(\eta_n)\to_p\I$ by Lemma~\ref{lemma: consistency of FI}.
    \item $\xin\to_p G_0^{-1}(1)$ by Lemma~\ref{lemma: xi: xi goes to 1}.
    \item $\omega_n\to_p\omega_0$ by Fact~\ref{fact: f grtr than F}.
    \item $\delta_n=O_p(n^{-1/2})$ and   Lemma~\ref{lemma: bound: xi n} implies that  
\[\xin\leq \frac{-5\log 2+2p\log n}{5w_n}.\]
Because $w_n\to_p w_0>0$ by Fact~\ref{fact: f grtr than F}, it follows that  $\xin=O_p(\log n)$, which implies $\xin\delta_n\to_p 0$.
    \item Suppose $\tilde \xi_n=\tilde G_n^{-1}(1-\eta_n/2)$. Then Lemma~\ref{lemma: hn: lower bound  on hn} implies $\sup_{x\in\mathcal [-\tilde\xi_n,\tilde\xi_n]}(\tilde g_{n_k}(x))^{-1}$ is $O_p(n^{2p/5})$. However, $\mathcal A_n'\subset [-\tilde\xi_n,\tilde\xi_n] $ by Lemma~\ref{lemma: xi: tilde xi n} with probability tending to one since $\td=O_p(n^{-1/2})$. Since $H(\hn,g_0)=O_p(n^{-p})$ and $p\in(0,1)$, it  follows that
    \[\frac{(\log n)^2H(\hn,g_0)^2}{\inf_{x\in\mathcal A_n'}\hn(x)}\to_p 0.\]
    \item Suppose $\tilde \xi_n=(\tilde G_n)^{-1}(1-\eta_n/2)$. 
    Lemma~\ref{lemma: xi: tilde xi n} implies that with probability tending to one, $\xin+2|\td|\leq \tilde\xi_n$, which implies
    \[\mathcal{I}_{1n}\leq \mathcal{I}_{2n}\leq \dint_{-\tilde\xi_n}^{\tilde\xi_n}\tp(z)^2\hn(z)dz\to_p\I,\]
    where the convergence in probability follows 
    from Lemma~\ref{lemma: consistency of FI: 2} noting $\hn\in\mathcal{SLC}_0$.
    On the other hand, Fatou's lemma and Condition~\ref{condition: on hn} indicates that
    \[\liminf_n\mathcal{I}_{2n}\geq \mathcal{I}_{1n}\geq \edint \liminf_{n}\slb 1_{ \mathcal A_n}(z)\tp(z)^2\hn(z)\srb dz=\I.\]
   Therefore, $\mathcal{I}_{1n}, \mathcal{I}_{2n}\to_p\I$.
    \item $\|\hn-g_0\|_\infty\to_p 0$ by Condition~\ref{condition: on hn}.
    \item  Since $\td=O_p(n^{-1/2})$, Lemma~\ref{lemma: An inclusion} yields  $P(\mathcal{A}_{n}' \subset \iint(\dom(\psi_0)))\to_n 1$.
\end{compactenum}
Suppose we can show that as $k\to\infty$,
  $\mathbb Y_{n_k}\to 1$  on $\mathcal  M$. Then it would establish that every subsequence  
of $n$ has a further subsequence $n_k$ along which $\mathbb Y_{n_k}\as  1$. Then  Fact~\ref{fact: Shorack} would yield $\mathbb Y_n\to_p 1$, as desired. 
% Therefore we have shown that it is enough to show $\mathbb Y_{n_k}\to 1$ on a probability-one subset of  $\mathcal M$. 
For the sake of simplicity, we will drop  $k$ from the subscript from  the definitions of  $\mathcal M$ and $\mathbb Y_{n_k}$.
% and show that $\mathbb Y_n\as 1$ if the convergence  in (a)-(g) above hold almost surely and
% \begin{align}\label{intheorem: second step: assertions}
%   \MoveEqLeft \overline \theta_{n}\as \theta_0,\ \widehat I_{n}(\eta_{n})\as\I,\ \xi_{n}\as G_0^{-1}(1), w_n\as w_0,\nn\\
%   &\ \frac{(\log n)^2 H^2(\hn,g_0)^2}{\hn(\tilde{\xi}_{n}+|\tilde\delta_{n}|)}\as 0, \dint_{-\xi_{n}}^{\xi_{n}-\tilde\delta_{n_k}}\tilde \psi_{n}(t)^2\hn(t)dt\as\I \text{ as }n_k\to\infty
%   \end{align}
% $\mathcal A_n'\subset \iint(\dom(\psi_0))$ with probability one  for all sufficiently large $n$.
% In the rest of the proof, we will use, often without mentioning, that the above almost sure convergences hold hold.   

Now we  derive some useful inequalities which hold on $\mathcal M$. Since $\omega_n\to \omega_0$ on  $\mathcal M$, Lemma \ref{lemma: bound: xi n} implies that    there exists $C>0$ so that $\xin\leq C\log n$ for all sufficiently large $n$ on $\M$. 
    Equation~\ref{inlemma: lemma bound psi knot prime}, on the other hand, implies that $|\psi_0'(\xin+|\td|)|$ is of the order of $\xin$. Therefore for large enough $C$,
\begin{equation}\label{intheorem: bound on psi knot prime}
    \limsup_n\sup_{t\in\mathcal A_n'}\psi_0'(t)=|\psi_0'(\xin+|\td|)|\leq C\log n\quad \text{on }\M.
\end{equation}
Here the monotonicity of $\ps_0'$ was used to obtain the last equality. 
Also note that because $\mathcal A_n'\subset \iint(\dom(\psi_0))$ for all sufficiently large $n$ on $\M$, 
we can apply Lemma~\ref{lemma: g/g(-)}
on $g_0$ to obtain
\begin{align}\label{intheorem: bound on g/g(-)}
    \limsup_n\sup_{t\in\mathcal A_n}\frac{g_0(t)}{g_0(t+\td)}\leq  \limsup_n e^{O(|\td|\xi_n)}\stackrel{(a)}{=} 1\quad \text{on }\M,
\end{align}
where (a) follows because $\td\xin\to_n 0$ on $\M$.

 Next we will establish the pointwise convergence of $b_n(t)$ on $\M$. Since $\|\hn-g_0\|_\infty\to 0$ on $\M$,  Lemma~\ref{Prop: the L1 convergence of the density estimators of one-step estimators: model: strong}(B) holds on $\M$.
 Using Lemma~\ref{Prop: the L1 convergence of the density estimators of one-step estimators: model: strong}(B) and the mean value Theorem,  we can show that on $\M$, $(\hln(t+\td)-\hln(t))/\td\to \psi'_0(t)$ for any $t\in\iint(\dom(\psi_0))$ that is a continuity point of $\psi_0'$. Because $\psi_0$ is concave, $\psi_0'$ is continuous Lebesgue almost everywhere on $\iint(\dom(\psi_0))$ \citep[see Corollary 25.5.1 and Theorem 25.5 of ][]{rockafellar}.
 Also noting $\widehat{\mathcal I}_{n}(\eta_{n})\to_n\I$ on $\M$,  we obtain that
\[g_0'(t)\dfrac{\dint_{t}^{t+\td}\tp(z)dz}{\td\hin(\eta_n)}\to_n \frac{g_0'(t)\psi_0'(t)}{\I},\quad \text{Lebesgue a.e. } t\in \iint(\dom(\psi_0))\ \text{on }\M.\]
Since  $\mathcal A_n\subset\iint(\dom(\psi_0))$ for sufficiently large $n$, and $\xin\to G^{-1}_0(1)$ on $\M$, it follows that $1_{\mathcal A_n}(t)$ converges to $1_{\iint(\dom(\psi_0))}(t)$ pointwise on $\M$ as well. Noting   $g_0'(t)=\psi_0'(t)g_0(t)$ for all  $t\in\iint(\dom(\psi_0))$, we then obtain that on $\M$,
\begin{align}\label{intheorem: T2: pointwise in dct}
    b_n(t)=1_{\mathcal A_n}(t)g_0'(t)\dfrac{\dint_{t}^{t+\td}\tp(z)dz}{\td\hin(\eta_n)}\to_n 1_{\iint(\dom(\psi_0))}(t)\frac{\psi_0'(t)^2g_0(t)}{\I}
\end{align}
for all $t\in\RR$ except a set of Lebesgue measure zero.
% Theorem~\ref{Theorem: the L1 convergence of the density estimators of one-step estimators: strong}, \eqref{inlemma: T2: DCT}, and Lemma~\ref{lemma: consistency of FI} imply that for all $t\in\iint(\dom(\psi_0))$,
% \begin{align}\label{intheorem: T2: pointwise in dct}
%     b_n(t)=g_0'(t)\dfrac{\dint_{t}^{t+\td}\tp(z)dz}{\td\hin(\eta_n)}\as \frac{g_0'(t)\psi_0'(t)}{\I}.
% \end{align}
% Here the almost sure converges follows because the $\bth$ under consideration satisfies $\bth\as\theta_0$.
The concavity of $\hln$ implies that its right derivative
$\tp$ is non-increasing. Hence, for any $t\in\iint(\dom(\psi_0))$, we have,
\begin{equation}\label{inlemma: T2: DCT}
\min\bigg\{\tp(t),\tp(t+\td)\bigg\}\leq \dfrac{\dint_{t}^{t+\td}\tp(z)dz}{\td}\leq \max\bigg\{\tp(t),\tp(t+\td)\bigg\},
\end{equation}
\begin{flalign}\label{intheorem: tp}
\text{yielding} &&  \dfrac{\bl\dint_{t}^{t+\td}\tp(z)dz\bl}{|\td|} &\leq |\tp(t)|+|\tp(t+\td)|.\quad\quad\quad\quad
\end{flalign}
Using \eqref{intheorem: tp}, we can bound $|b_n(t)|$ noting
\begin{align*}
    |b_n(t)|\leq &\ \underbrace{1[t\in\mathcal A_n]|\psi_0'(t)|g_0(t) \frac{|\tp(t)|}{\hin(\eta_n)}}_{b_{1n}(t)} + \underbrace{1[t\in\mathcal A_n]|\psi_0'(t)|g_0(t) \frac{|\tp(t+\td)|}{\hin(\eta_n)}}_{b_{2n}(t)}.
\end{align*}
Now defining
\[T_{21,n}(t)=1[t\in\mathcal A_n]|\psi_0'(t)|\sqrt{g_0(t)}\lb\sqrt{g_0(t)}-\sqrt{\hn(t)}\rb \frac{|\tp(t)|}{\hin(\eta_n)},\]
\[T_{22,n}(t)=1[t\in\mathcal A_n]|\psi_0'(t)|\sqrt{g_0(t)} \frac{|\tp(t)|\sqrt{\hn(t)}}{\hin(\eta_n)},\]
\[T_{23,n}(t)=1[t\in\mathcal A_n]|\psi_0'(t)|\frac{{g_0(t)}}{\sqrt{g_0(t+\td)}}\lb\sqrt{g_0(t+\td)}-\sqrt{\hn(t+\td)}\rb \frac{|\tp(t+\td)|}{\hin(\eta_n)},\]
\[T_{24,n}(t)=1[t\in\mathcal A_n]|\psi_0'(t)|\frac{{g_0(t)}}{\sqrt{g_0(t+\td)}} \frac{|\tp(t+\td)|\sqrt{\hn(t+\td)}}{\hin(\eta_n)},\]
we note that
\[b_{1n}(t)= T_{21,n}(t)+T_{22,n}(t)\quad\text{and}\quad b_{2n}(t)= T_{23,n}(t)+T_{24,n}(t).\]
Thus we can upper bound $|b_n(t)|$ by $c_n(t)$ where
\[c_n(t)=T_{21,n}(t)+T_{22,n}(t)+T_{23,n}(t)+T_{24,n}(t).\]

Our aim is to apply Fact~\ref{fact: pratt's lemma} (Pratt's Lemma) with $a_n=0$ to prove the current Lemma. %Therefore, we need to show that $c_n$ converges Lebesgue almost everywhere to a Lebesgue-integrable function $c$ such that $\int c_n(t)dt\to \int c(t)dt$ holds as well. 
To this end, we first show that the following assertions hold on $\M$:
\begin{itemize}
  \item[A1.]   $T_{21,n}\to_n 0$ and $T_{23,n}\to_n 0$ Lebesgue almost everywhere on $\RR$.
 \item[A2.] $\int_{\RR} T_{21,n}(t)dt\to_n 0$ and $\int_{\RR} T_{23,n}(t)dt\to_n 0$.
  \item[A3.] There are   functions $t_{22}:\RR\mapsto\RR$ and $t_{24}:\RR\mapsto\RR$ so that $T_{22}\to_n t_{22}$ and $T_{24}\to_n t_{24}(t)$ Lebesgue almost everywhere on $\RR$.
    \item[A4.] The functions $t_{22}$ and $t_{24}$ in A3 are integrable. Moreover, $\int_{\RR} T_{22,n}(t)dt\to_n \int_{\RR} t_{22}(t) dt $ and $\int_{\RR} T_{24,n}(t)dt\to_n \int_{\RR} t_{24}(t) dt $.
    \end{itemize}
    
    % For each $\e>0$, there exists $\sigma>0$ so that any set $\mathcal B\subset\RR$ with $\int_{\mathcal B} g_0(t)dt<\sigma$ satisfies
    % \[\limsup_n\dint_{\mathcal B} T_{22,n}(t)g_0(t)dt<\e\quad\text{and}\quad \limsup_n\dint_{\mathcal B} T_{24,n}(t)g_0(t)dt<\e.\]
 Let us denote $c(t)=t_{22}(t)+t_{24}(t)$. Then
A1-A4 imply that on $\M$, $c_n(t)\to_n c(t) $ Lebesgue almost everywhere, $c$ is integrable, and 
$\int_{\RR} c_n(t)dt \to_n \int_{\RR} c(t)dt $. Since $|b_n(t)|\leq c_n(t)$ and \eqref{intheorem: T2: pointwise in dct} holds,  Pratt's Lemma (see Fact~\ref{fact: pratt's lemma}) yields
\[\edint b_n(t)dt\to_n \dint_{G_0^{-1}(0)}^{G_0^{-1}(1)}\frac{g_0'(t)^2\psi_0'(t)dt}{\I}=1\quad \text{on }\M ,\]
which completes the proof of Lemma~\ref{lemma: main: key lemma for T2}.
 
\subsubsection*{\textbf{Proof of A1 and A3:}}
Since $\|\hn-g_0\|_\infty\to_n 0$ on $\mathcal M$, Lemma~\ref{Prop: the L1 convergence of the density estimators of one-step estimators: model: strong}  implies that on $\M$, the functions $\hn,\hn(\mathord{\cdot}+\td)$ converge  pointwise to $ g_0$, and $\tp,\tp(\mathord{\cdot}+\td)$ converge to $\psi_0'$ Lebesgue almost everywhere on $\iint(\dom(\psi_0))$. Continuity of $g_0$ implies $g_0(t+\td)\to_n g_0(t)$ for all $t\in\RR$.  Using the above, it can be shown that 
\[ T_{21,n}, T_{23,n}\to_n 0,\quad T_{22,n},T_{24,n}\to_n 1_{\iint(\dom(\psi_0))}{\psi_0'^2g_0}/{\I} \quad \text{a.e. Lebesgue}\ \text{on }\M.\]
 
\subsubsection*{\textbf{Proof of A2:}}
Using Cauchy-Schwarz inequality, the bound on $g_0$ from Fact~\ref{fact: Lemma 1 of theory paper}, and the bound on $\psi_0'$ from \eqref{intheorem: bound on psi knot prime}, we can show that there exists $C>0$ such that the following holds for all sufficiently large $n$ on $\M$:
\begin{align*}
\MoveEqLeft \edint |T_{21,n}(t)|dt\
 \\
 \leq &\ C \log n\left(\dint_{-\xin-\td}^{\xin}\frac{\slb\sqrt{g_0(t)}-\sqrt{\hn(t)}\srb^2}{{\hn(t)}} \right)^{1/2}\lb \dint_{\mathcal A_n}\frac{|\tp(t)|^2\hn(t) dt}{\hin(\eta_n)^2}\rb^{1/2}\\
 \leq &\ C \log n\frac{H(\hn, g_0)}{\slb \inf_{x\in\mathcal A'_n} \hn(x) \srb^{1/2}}\frac{\sqrt{\mathcal I_{1n}}}{\hin(\eta_n)},
\end{align*}
 which approaches  zero as $n\to\infty$  because 
 \begin{equation}\label{inlemma: key lemma: A2 proof}
     \frac{ C(\log n) H(g_0,\hn)}{ \slb\inf_{x\in\mathcal A_n'}\hn(x)\srb^{1/2}}\to_n 0,\quad \mathcal I_{1n},\ \hin(\eta_n)\to_n \I\quad\text{on }\M.
 \end{equation}
 by \eqref{inlemma: def: mathcal M}.
% Unimodality and the symmetry of $\hn$ about zero also gives
% \[\min\{\hn(\xin),\hn(-\xin-\td)\}\geq\slb\inf_{x\in\mathcal A_n'}\hn(x)\srb\]
%  Now by our choice of $\xin$,\todo[color=blue!20]{one place of choosing $\xi_n$ choose later}
%  \begin{align}\label{eq: choice of xin: 1}
%      \frac{H^2(\hn,g_0)}{\slb\inf_{x\in\mathcal A_n'}\hn(x)\srb}=\frac{O_p(n^{-4/5})}{\slb\inf_{x\in\mathcal A_n'}\hn(x)\srb}=o_p(1).
%  \end{align}
% % \[\min(\hn(\xin),\hn(\xin- n^{-1/3}),\hn(\xin+n^{-1/3}))>n^{-3/5}.\]
% % Since $\td=O_p(n^{-1/2})$, $\td\in[-n^{-1/3}, n^{1/3}]$ for sufficiently large $n$. Therefore, $\hn(\xin)>n^{-3/5}$ and
% % Suppose $\hn(-\xin-\td)=\hn(\xin+\td)>n^{-3/5}$.
% % Since $H(\hn,g_0)=O_p(n^{-4/5})$, we have
% % \[\frac{H(\hn, g_0)}{\min\{\hn(\xin)^{1/2},\hn(-\xin-\td)^{1/2}\}}=O_p(n^{-1/5})=o_p(1).\]
% Also,
% \begin{align}\label{inthm: FI inequality}
%     \limsup_n \dint_{-\xin-\td}^{\xin}\frac{|\tp(t)|^2\hn(t) dt}{\hin(\eta_n)^2}\leq \limsup_n \dint_{-\xin-\td}^{\xin+\td}\frac{|\tp(t)|^2\hn(t) dt}{\hin(\eta_n)^2}\leq  \limsup_n\frac{1}{\hin(\eta_n)}\leq \frac{1}{\I},
% \end{align}
% where the last inequality follows from  Fatou's Lemma.
The proof for $T_{23,n}$ is similar. An application of the Cauchy-Schwarz inequality, the bound on $\psi_0'$ by \eqref{intheorem: bound on psi knot prime} and the bound in \eqref{intheorem: bound on g/g(-)} imply that the following holds for all large $n$ on $\M$:
\begin{align*}
\edint T_{23,n}(t)dt\leq &\ 
  \frac{ C(\log n)H(g_0,\hn)}{ \slb\inf_{x\in\mathcal A_n'}\hn(x)\srb^{1/2}}\frac{\lb\dint_{-\xin+\td}^{\xin}\tp(t)^2\hn(t)dt\rb^{1/2}}{\hin(\eta_n)}\\
 =&\ \frac{ C(\log n) H(g_0,\hn)}{ \slb\inf_{x\in\mathcal A_n'}\hn(x)\srb^{1/2}}\frac{\sqrt{\mathcal I_{2n}}}{\hin(\eta_n)},
\end{align*}
which converges to zero  by \eqref{inlemma: key lemma: A2 proof} and the fact that $\mathcal{I}_{2n}\to_n\I$ on $\M$,
thus completing the proof of A2.

% Hence, for $t\in [-\xin-\td,\xin]$,
% \[\bl g_0'(t)\dfrac{\dint_{t}^{t+\td}\tp(z)dz}{\td\hin(\eta_n)}\bl\leq c'_n(t)\]
% and $\limsup_n\dint_{-\xin-\td}^{\xin}c'_n(t)dt<\infty$.
% \subsubsection{Proof of A3}
% Since $\bth\as 0$, Theorem~\ref{Theorem: the L1 convergence of the density estimators of one-step estimators: strong} implies that $\tp,\hln'(\cdot+\td)$ converge  to $\psi_0'$  and $\hn(\cdot+\td),\hn$ converges to $g_0$ pointwise on $\iint(\dom(\psi_0))$ almost surely. 
% Using \eqref{intheorem: second step: assertions}, \eqref{intheorem: main: inclusion of A1n}, and the continuity of $g_0$, we obtain that 
% \[T_{22,n}(t), T_{24,n}(t)\as 1[t\in\dom(\psi_0)]\frac{\psi_0'(t)^2 g_0(t)}{\I}\quad\text{ for all }t\in\RR.\]

\subsubsection*{\textbf{Proof of A4:}}
Let us define
\[\mathcal T_{22,n}(t)= |T_{22,n}(t)|/g_0(t)\quad\text{ and }\quad \mathcal T_{24,n}(t)= |T_{24,n}(t)|/g_0(t).\]

We will show that on $\mathcal M$,
for each $\e>0$, there exists $\sigma>0$ so that the following bounds are true for any $G_0$-measurable set $\mathcal B\subset\RR$ satisfying $\int_{\mathcal B} g_0(t)dt<\sigma$:
\begin{align}\label{intheorem: UI: step 1}
    \limsup_n\dint_{\mathcal B} \mathcal T_{22,n}(t)g_0(t)dt<\e\quad\text{and}\quad \limsup_n\dint_{\mathcal B} \mathcal T_{24,n}(t)g_0(t)dt<\e.
\end{align}
Next, we will show that
\begin{align}\label{intheorem: UI: step 2}
 \limsup_n\edint \mathcal T_{22,n}(t)g_0(t)dt<\infty\ \text{and}\  \limsup_n\edint \mathcal T_{24,n}(t)g_0(t)dt<\infty.
\end{align}
If \eqref{intheorem: UI: step 1} and \eqref{intheorem: UI: step 2} hold,  Fact~\ref{fact: condition for UI}
underscores that  the sequences $(\mathcal T_{22,n})_{n\geq 1}$ and $(\mathcal T_{24,n})_{n\geq 1}$ are uniformly integrable with respect to the measure  induced by $G_0$. Then A4 follows from A3 and Theorem 16.13 (pp. 220) of \cite{billingsley} (Vitali convergence Theorem).
 Thus it suffices to show that \eqref{intheorem: UI: step 1} and \eqref{intheorem: UI: step 2} hold.
 
 Note that since $\I<\infty$, by Fact~\ref{fact: condition for integrability}, given any $\e>0$, we can choose $\sigma>0$ so that for any $G_0$-measurable set $\mathcal B\subset\RR$ satisfying $\int_{\mathcal B}g_0(t)dt<\sigma$, the integral
 $\int_{\mathcal B}\psi_0(x)^2g_0(x)<\e^2\I=\e'$ (say).
 It will soon be clear why this choice of $\e'$ works. Using the Cauchy-Schwarz inequality in the third step, we calculate
 \begin{align*}
 \dint_{\mathcal B}\mathcal T_{22,n}(t)g_0(t)dt=  &\  \dint_{\mathcal B}T_{22,n}(t)dt= \dint_{\mathcal B\cap \mathcal A_n}|\psi_0'(t)|\sqrt{g_0(t)}\frac{|\tp(t)|\sqrt{\hn(t)}}{\hin(\eta_n)}\\
     \leq &\ \lb\dint_{\mathcal B}\psi_0'(t)^2g_0(t)dt \rb^{1/2}\lb\dint_{\mathcal A_n}\frac{\tp(t)^2\hn(t)dt}{\hin(\eta_n)^2}\rb^{1/2},
 \end{align*}
which is bounded by $\sqrt{\e'\mathcal{I}_{1n}}/\hin(\eta_n)$. Noting $\mathcal{I}_{1n}\to_n \I$ and  $\hin(\eta_n)\to_n\I$ on $\M$, we obtain
\[\limsup_n \dint_{\mathcal B}\mathcal T_{22,n}(t)g_0(t)dt\leq \sqrt{\e'/\I}=\e\quad\text{on }\M.\]
%However, since under our set up  $\mathcal{I}_{1n},\ \hin(\eta_n)\as\I$, with probability one, the bound in the last display holds for all choices of $\mathcal B$.
%  Lemma~\ref{lemma: consistency of FI} implies
%  \[\limsup_n \dint_{\mathcal B}\mathcal T_{22,n}(t)g_0(t)dt\leq \frac{C\sqrt\sigma}{\sqrt{\I}}.\]
%  Therefore, we have to choose $\sigma=\e^2\I/C^2$.
 Letting $\mathcal B=\RR$, and repeating the above steps, we can show that
  \[ \limsup_n\int_{\RR}\mathcal{T}_{22,n}(t)g_0(t)dt<1\quad\text{ on }\M.\]
  For $\mathcal T_{24,n}$, the Cauchy-Schwarz inequality yields 
\begin{align}\label{inlemma: bound: T24}
     \dint_{\mathcal B}\mathcal{T}_{24,n}(t)g_0(t)dt\leq &\ \lb\dint_{\mathcal B\cap \mathcal{A}_n} \frac{g_0(t)^2\psi_0'(t)^2}{g_0(t+\td)}dt\rb^{1/2}
      \lb \dint_{-\xin+\td}^{\xin}\frac{\tp(t)^2\hn(t) dt}{\hin(\eta_n)^2}\rb^{1/2}\nn\\
      \stackrel{(a)}{\leq} &\  
    \sqrt{\e'}\sup_{t\in\mathcal A_n}\lb\frac{g_0(t)}{g_0(t+\td)}\rb^{1/2}\frac{\mathcal{I}_{2n}^{1/2}}{\hin(\eta_n)}.
\end{align}
Here (a) follows because $\int_{\mathcal B}\psi'_0(t)^2g_0(t)dt<\e'$.
The fact that $\mathcal{I}_{1n}$, $\hin(\eta_n)\to_n\I$ on $\M$, in conjunction with the bound in \eqref{intheorem: bound on g/g(-)},  implies
 \[\limsup_n  \dint_{\mathcal B}\mathcal{T}_{24,n}(t)g_0(t)dt\leq \sqrt{\e'/\I}=\e\quad\text{on }\M.\]
%  on the set where holds, and  hold. Since the latter holds with probability one,  the bound in the last display holds almost surely uniformly across all choices of $\mathcal B$. 
 Thus \eqref{intheorem: UI: step 1}  is proved. Letting $\mathcal B=\RR$ leads to  $ \limsup_n\int_{\RR}\mathcal{T}_{24,n}(t)g_0(t)dt<1$ on $\M$, thus finishing the proof of \eqref{intheorem: UI: step 2}.

\end{proof}

 \subsection{Auxilliary lemmas for the proof of Theorem~\ref{theorem: main: one-step: full}}
 \subsubsection{\textbf{Lemmas on $\xi_n$:}}
 \label{sec: lemmas on xi n}
  Unless otherwise mentioned, for all the lemmas on $\xi_n$,  $\xi_n$ will denote $\tilde G_n^{-1}(1-\eta_n)$, where the choice of $\hn$ should be clear from the context.
  
 \begin{lemma}
 \label{lemma: xi: xi is in dop psi knot}
 Suppose $\hn\in\mathcal{S}_0$ satisfies Conditions~\ref{condition: on hn} and \ref{cond: hellinger rate}. 
 Let $\eta_n=C n^{-2p/5}$, where $C>0$ and $p$ is as in Condition~\ref{cond: hellinger rate}.
  Then for $\xi_n=\tilde G_n^{-1}(1-\eta_n)$, we have
 \[P\slb [-\xi_n,\xi_n]\subset \mathrm{int}(\dom(\psi_0))\srb\to 1.\]
 \end{lemma}
 
 \begin{proof}[Proof of Lemma~\ref{lemma: xi: xi is in dop psi knot}]
  Using Fact~\ref{fact: dTV and hellinger} in step (a) we obtain that
  \[|G_0(-\xin)-\tilde G_n(-\xin)|\leq d_{TV}(G_0,\tilde G_n)\stackrel{(a)}{\leq} \sqrt{2}H(\hn,g_0)=O_p(n^{-p})\]
  by Condition~\ref{cond: hellinger rate}.
  Therefore $G_0(-\xin)\geq \tilde G_n(-\xin)+O_p(n^{-p})\geq \eta_n+O_p(n^{-p})$ because $F(F^{-1}(q))\geq q$ for any distribution function $F$, and $q\in(0,1)$. Since $\eta_n=Cn ^{-2p/5}\gg n^{-p}$, it follows that $P(G_0(-\xin)\geq \eta_n/2)\to 1$. Thus $P(-\xin\in \iint(dom(\psi_0)))\to_n 1 $. Since $\psi_0\in \mathcal {SC}_0$,  $\iint(\dom(\psi_0))$ is an interval of the form $(-a,a)$ for some $a>0$. Noting $-\xin\in(-a,a)$ implies $[-\xin,\xin]\subset(-a,a)$, the proof follows.

 \end{proof}
 
 \begin{lemma}
 \label{lemma: xi: xi goes to 1}
 Consider the set up of Lemma~\ref{lemma: xi: xi is in dop psi knot}. Then $\xi_n\to_p G_0^{-1}(1)$ as $\eta_n\to 0$.
 \end{lemma}
 
 \begin{proof}[Proof of Lemma~\ref{lemma: xi: xi goes to 1}]
 \textcolor{red}{This has been changed.}
%   Note that
%  \[(X-\bth)_{(n)}=X_{(n)}-\bth =Z_{(n)}+\td\]
%  where $Z=X-\theta_0$ and $\td=\theta_0-\bth$. Because $Z$ has distribution $G_0$, it follows that $Z_{(n)}\to_p G_0^{-1}(1)$. Since $\td\to_p 0$, we have $(X-\bth)_{(n)}\to_p G_0^{-1}(1)$. Similarly we can show that
%  $X_{(1)}-\bth\to_p G_0^{-1}(0)$, which yields $\bth-X_{(1)}\to_p -G_0^{-1}(0)$. Since $g_0$ is symmetric about zero, $G_0^{-1}(0)=-G_0^{-1}(1)$. Thus both the minimum and maximum of $X_{(n)}-\bth$ and $\bth-X_{(1)}$ converge in probability to $G_0^{-1}(1)$ if the latter is finite and diverges to $\infty$ otherwise. Thus we have established that $\tilde G_n^{-1}(1)\to_p G^{-1}(1)$. 

% First consider the case when $G_0^{-1}(1)=\infty$. In this case it suffices to show that  $\xi_n\to_p\infty$.
%  If that does not hold, then there exists $C>0$ so that $\limsup_n P(\xi_n\leq C)>0$. However, there exists $t<1$ so that $G_0^{-1}(t)>2C$. For sufficiently large $n$,  $\eta_n<1-t$, leading to
%  \[\xi_n=\tilde G^{-1}_n(1-\eta_n)\geq \tilde G^{-1}_n(t)\stackrel{(a)}{\to_p} G_0^{-1}(t)>2C,\]
% where (a) follows by Fact~\ref{fact: consistency of the quantiles} and Condition~\ref{condition: on hn}. Therefore, a contradiction arises. Hence, $\xin\to_p\infty$ in this case. 
% Now suppose $G^{-1}_0(1)<\infty$. 

  Suppose, if possible, 
 $\xi_n\to_p G_0^{-1}(1)$ does not hold. We will consider two cases then: (a) $G_0^{-1}(1)<\infty$ and (b) $G_0^{-1}(1)=\infty$.

 \paragraph{Case (a):}
 Since  $\xi_n\to_p G_0^{-1}(1)$ does not hold, we can find an $\epsilon>0$ and a subsequence  $\{n_k\}\subset \{n\}$ so that $\liminf_k P(|\xi_{n_k}-G_0^{-1}(1)|>\epsilon)>0$. To avoid cumbersome notation, we will denote $\xi_{n_k}$ by $\xi_n$ from now on. Since Lemma~\ref{lemma: xi: xi is in dop psi knot} implies  $P(\xi_n\leq G_0^{-1}(1))\to 1$, it follows that $\liminf_n P(\xi_{n}<G_0^{-1}(1)-\epsilon)>0$. Now we show that there exists some $t\in[1/2,,1)$ such that $G_0^{-1}(t)=G_0^{-1}(1)-\epsilon$.

 Let us denote $t=G_0(G_0^{-1}(1)-\e)$. 
 Since $g_0$ is log-concave, it is positive on $J(G_0)$. Therefore, if $t,t'\in J(G_0)$ satisfies $t<t'$, then $G_0(t)<G_0(t')$.   Note that $\epsilon$ can be chosen small enough so that $G_0^{-1}(1)-\e\in(G_0^{-1}(1/2),G_0^{-1}(1))$,
  which implies $t=G_0(G_0^{-1}(1)-\e)>G_0(G_0^{-1}(1/2))\geq 1/2$.  By our choice of $\epsilon$, the number  $G_0^{-1}(1)-\e\in [G_0^{-1}(1/2),G_0^{-1}(1)) $. Because $G_0$ is strictly increasing on the latter set, it can be seen that  $t=G_0(G_0^{-1}(1)-\e)<1$. Therefore, $t\in[1/2,1)$. Finally,  
  \[G_0^{-1}(t)=G_0^{-1}\left(G_0(G_0^{-1}(1)-\e)\right)=G_0^{-1}(1)-\e,\]
  where the last step follows because $G_0^{-1}(1)-\e\in [G_0^{-1}(1/2),G_0^{-1}(1)) $ implies we can find a neighborhood of $G_0^{-1}(1)-\e$ where $G_0$ is strictly increasing. Therefore, we have
 proved that  there exists $t\in[1/2,1)$ so that  $G_0^{-1}(t)=G_0^{-1}(1)-\epsilon$, which yields 
 \[\liminf_n P(\xi_n< G_0^{-1}(t))>0.\]

 % We also have $t<z$ for some $z<1$. Moreover, since  $G_0$ is strictly increasing on $[G_0^{-1}-\e,G_0^{-1}(z)]$,  there exists $s\in\RR$ such that $t\leq G_0(s)<z$. The latter implies $G_0^{-1}(t)<G_0^{-1}(z)$.

 % 
 % is continuous on $(0,1)$ by Proposition A.7 of \cite{bobkovbig}. We can choose $\epsilon$ to be small enough so that there exists some $t\in(0,1)$ such that $G_0^{-1}(t)=G_0^{-1}(1)-\epsilon$.

 However, because $\eta_n\to 0$,  $1-\eta_n\geq (1+t)/2$ for sufficiently large $n$, which yields
 $\xi_n\geq\tilde G_n^{-1}((1+t)/2)$.  Now by Fact~\ref{fact: consistency of the quantiles} and Condition~\ref{condition: on hn}, $\tilde G_n^{-1}((1+t)/2)\to_pG_0^{-1}((1+t)/2)$. However, $G_0^{-1}((1+t)/2)>G_0^{-1}(t)$, where the strict inequality follows because  $g_0$ being log-concave, is positive on $J(G_0)$, indicating $G_0^{-1}$ is strictly increasing on $(0,1)$. Thus it follows that $P(\xi_n\geq G_0^{-1}(t))\to 1$.  Therefore the proof follows by contradiction. 

 \paragraph{Case (b)} Since  $\xi_n\to_p \infty$ does not hold,  there exists $M>0$ and a subsequence $n_k$ so that $\liminf_k P(\xi_{n_k}<M)>0$. To avoid cumbersome notation, we will denote $\xi_{n_k}$ by $\xi_n$ from now on. Note that $\xi_{n}<M$ implies $\tilde G_n(\xi_{n})\leq \tilde G_n(M)$. On one hand,  $\tilde G_n(\xi_{n})=1-\eta_n$ by Lemma A.3.5 of \cite{bobkovbig} because $\tilde G_n$ is continuous on $J(\tilde G_n)$.  On the other hand, $\tilde G_n(M)\to_p G_0(M)$ by Condition~\ref{condition: on hn}. Note that since $G_0^{-1}(1)=\infty$, $G_0(M)<1$. Specifically,  there exists $z<1$ so that $G_0(M)<z$. Therefore, we have shown that   $P(1-\eta_n \leq z)\to_n 1$.
 However, the above can not hold since $\eta_n\to 0$. Thus, there is no $M$ so that $\liminf_n  P(\xi_{n}<M)>0$ holds. We have come to a contradiction again, which concludes our proof.
 
 \end{proof}
 
  \begin{lemma}\label{lemma: bound: xi n}
 Suppose either $\hn\in\mathcal{SLC}_0$ is a density  satisfying Condition~\ref{condition: on hn} and Condition~\ref{cond: hellinger rate}, or $\hn$ satisfies Condition~\ref{cond: hnss condition}.  Let $\eta_n=Cn^{-2p/5}$ where $p$ is as in Condition~\ref{cond: hellinger rate} (or  Condition~\ref{cond: hnss condition}) and $C>0$. Then
 \[\xi_n\leq\frac{-\log 2+2p(\log n)/5}{\omega_n}\]
 where $\omega_n$ is as in Fact~\ref{fact: f grtr than F}. In fact,
 $|\xi_n|=O_p(\log n)$.
 \end{lemma}
 
 \begin{proof}[Proof of Lemma~\ref{lemma: bound: xi n}]
 Observe  that if  $\hn\in\mathcal{SLC}_0$, then $\hn(\tilde G_n^{-1}(z))>0$ for $z\in(0,1)$. If $\hn$ satisfies Condition~\ref{cond: hnss condition}, then also the above holds because by Condition~\ref{cond: hnss condition}, $\hn>0$ on $\RR$. Since $\hn$ is symmetric about zero,  $0=\tilde G_n^{-1}(1/2)$.  Noting $\xi_n=-\tilde{G}_n^{-1}(\eta_n)$,  we therefore derive that
 \begin{align*}
     \xi_n=\tilde{G}_n^{-1}(1/2)-\tilde{G}_n^{-1}(\eta_n)=\dint_{\eta_n}^{1/2}\frac{dz}{\hn(\tilde G_n^{-1}(z))}\leq \dint_{\eta_n}^{1/2}\frac{dz}{\omega_n\tilde G_n(\tilde G_n^{-1}(z))},
 \end{align*}
 where $\omega_n$ is as in Fact~\ref{fact: f grtr than F}.
 Because $\hn>0$ on $J(\tilde G_n)$, it follows that  $\tilde G_n$ is continuous on $J(\tilde G_n)$. Therefore, we have
 $\tilde G_n(\tilde G_n^{-1}(z))=z$, implying \[\xi_n\leq \frac{\log(1/2)-\log(\eta_n)}{\omega_n}=\frac{-\log 2+2p(\log n)/5}{\omega_n}.\]
 Since $\omega_n\to_p\omega_0>0$ by Fact~\ref{fact: f grtr than F}, the proof follows.
 \end{proof}
 
  \begin{lemma}\label{lemma: xi: tilde xi n}
 Consider the set up of Lemma~\ref{lemma: bound: xi n}. Let $\tilde \xi_n=\tilde G_n^{-1}(\eta_n/2)$.
 Suppose $y_n$ is a sequence of non-negative random variables so that $P(y_n< \eta_n/(2g_0(0)))\to 1$.  Then 
  \begin{align}\label{inlemma: yn set inclusion}
   P([-\xi_n-y_n,\xi_n+y_n]\subset [-\tilde{\xi}_n,\tilde{\xi}_n])\to 1.  
 \end{align}
 \end{lemma}
 
 \begin{proof}[Proof of Lemma~\ref{lemma: xi: tilde xi n}]
 Under our set up, $\hn$ is positive on the set $J(\tilde G_n)$. Therefore  the function $\tilde G_n^{-1}$ is continuous on $(0,1)$. Hence the mean value theorem implies 
 \[\tilde G_n^{-1}(\eta_n)-\tilde G_n^{-1}(\eta_n/2)= \frac{\eta_n}{2\hn(\tilde G_n^{-1}(t))}\geq \frac{\eta_n}{2\|\hn\|_\infty}\]
 for some $t\in[\eta_n/2,\eta_n]$. 
Condition~\ref{condition: on hn}  implies that $\|\hn\|_\infty\to_p\|g_0\|_\infty= g_0(0)$.
 Therefore, as $n\to\infty$,
  \[P\lb \liminf_n \frac{\tilde G^{-1}_n(\eta_n)-\tilde G^{-1}_n(\eta_n/2)}{\eta_n}\geq \frac{1}{2g_0(0)}\rb\to 1.\] 
  Hence if $y_n< \eta_n/(2g_0(0))$, then 
$ \tilde{G}^{-1}_n(\eta_n)-y_n \geq  \tilde{G}^{-1}_n(\eta_n/2)$ with probability tending to one.
Since $\hn$ is symmetric about zero, we obtain that
$ \tilde{G}^{-1}_n(1-\eta_n)+y_n < \tilde{G}^{-1}_n(1-\eta_n/2)$
with probability tending to one.
  Since $\xi_n=\tilde{G}^{-1}_n(1-\eta_n)$ and $\tilde{\xi}_n=\tilde{G}^{-1}_n(1-\eta_n/2)$, the proof  follows.
 \end{proof}
 
 \begin{lemma}\label{lemma: An inclusion}
  Consider  the set up of Lemma~\ref{lemma: bound: xi n}.  Then for $y_n=o_p(\eta_n)$, we have
   \[P\slb[-\xin-|y_n|,\xin+|y_n|]\subset \mathrm{int}(\dom(\psi_0))\srb\to 1.\]
 \end{lemma}
 
 \begin{proof}[Proof of Lemma~\ref{lemma: An inclusion}]
 Letting $\tilde{\xi}_n=-\tilde G_n^{-1}(\eta_n/2)$, and applying Lemma~\ref{lemma: xi: xi is in dop psi knot}, we obtain $P([-\tilde{\xi}_n,\tilde{\xi}_n]\subset \iint(\dom(\psi_0)))\to 1$. Then the result follows from Lemma~\ref{lemma: xi: tilde xi n}.
 \end{proof}

 \subsubsection{\textbf{lemmas on  $\hn$ and $g_0$:}}
 \begin{lemma}
 \label{lemma: hn: lower bound  on hn}
 Suppose $\xin=\tilde G^{-1}_n(1-\eta_n)$ where $\hn\in\mathcal {SLC}_0$ satisfies Condition~\ref{condition: on hn}. Then  \begin{itemize}[topsep=0pt]
     \item [A.] $\sup_{x\in[-\xi_n,\xi_n]}\hn(x)=O_p(1)$ and $\sup_{x\in[-\xi_n,\xi_n]}\hn(x)^{-1}=O_p(\eta^{-1}_n)$.
     \item[B.] $\hln=\log\hn$ satisfies $\sup_{x\in[-\xi_n,\xi_n]}\hln(x)=O_p(1)$. For $\eta_n=Cn^{-2p/5}$ with $p\in(0,1)$ and $C>0$, we have $\sup_{x\in[-\xi_n,\xi_n]}(-\hln(x))=O_p(\log n)$.
 \end{itemize}
 \end{lemma}
 
 \begin{proof}[Proof of Lemma~\ref{lemma: hn: lower bound  on hn}]
 The upper bound on $\hn$ follows from Fact~\ref{fact: Lemma 1 of theory paper} and  Condition~\ref{condition: on hn}. For the upper bound on $\hn^{-1}$, note that  Fact~\ref{fact: f grtr than F} implies that 
 \begin{align*}
    \hn(x)\geq &\  w_n \min(\tilde G_n(x), 1-\tilde G_n(x)), \quad\text{for all }x\in\RR.
 \end{align*}
 Since $\tilde G_n$ is a non-decreasing and $1-\tilde G_n$ is a non-increasing function, any $x\in[-\xin,\xin]$ satisfies
\[\hn(x)\geq {\omega_n}{\min(\tilde G_n(-\xin),1-\tilde G_n(\xin))}=\omega_n\eta_n\]
 because $\tilde G_n(-\xin)=\eta_n$.
 Since the random variable $\omega_n\to_p\omega_0>0$ by Fact~\ref{fact: f grtr than F},
   part A of the current lemma follows.
 Part B follows directly from Part A.
 \end{proof}

 \begin{lemma}\label{Prop: the L1 convergence of the density estimators of one-step estimators: model: strong}
Assume $f_0\in\mathcal{P}_0$. Suppose  $\{\hn\}_{n\geq 1}$ is a sequence of log-concave densities satisfying $\|\hn-g_0\|_\infty\to_n 0$.   Then the following hold for any $y_n\to_n 0$:
\begin{itemize}[topsep=0.5 pt]
\item[(A)] Let $\hln=\log \hn$. Then
 $\hln(\mathord{\cdot} +y_n)\to_n \ps_0$  everywhere on $\iint(\dom(\psi_0))$.
\item[(B)]
$\hln'(\mathord{\cdot} +y_n)\to_n\ps_0'$  Lebesgue almost everywhere on $\iint(\dom(\psi_0))$. In particular, if $x$ is a continuity point of $\ps_0'$, then $\hln'(x +y_n)\to_n\ps_0'$. 
\end{itemize}
\end{lemma}

\begin{proof}[Proof of Lemma~\ref{Prop: the L1 convergence of the density estimators of one-step estimators: model: strong}]
By our assumptions on $g_0$,  $\mathcal I_{g_0}<\infty$. Therefore, $g_0$ is absolutely continuous \citep[Theorem 3,][]{huber}.
 Hence, $\sup_{x\in \RR}|\hn(x+y_n)-g_0(x)|\to_n 0$. Since for each  $x\in\iint(\dom(\psi_0))$, there exists an open neighborhood around $x$ where $|\psi_0|<\infty$, $\hln(x+y_n)\to_n \psi_0(x)$ for each $x\in\iint(\dom(\psi_0))$. Therefore part (A) follows. For part (B), first note that if $x\in\iint(\dom(\psi_0))$ is a continuity point of $\psi_0'$, then $\tp(x+y_n)\to_n \tp(x)$  by Theorem 25.7 of \cite{rockafellar}. Now since $\psi_0$ is concave, $\psi_0$ is continuously differentiable at $x$ if it is differentiable at  $x$ \citep[][Corollary 25.5.1]{rockafellar}.
 However, a concave  $\psi_0$ is differentiable Lebesgue almost everywhere on $\dom(\phi_0)$ \citep[][Theorem 25.5]{rockafellar}. Therefore, the lemma follows.
 \end{proof}
 
 \subsubsection{\textbf{Lemmas on $g_0$:}}
 \begin{lemma}
 \label{lemma: g:  bound on psi}
 Suppose $g_0$ satisfies Assumption~\ref{assump: L}. Let $\kappa$ be as in Assumption~\ref{assump: L}. Then for any $x\in\dom(\psi_0)$,  $\psi_0$ satisfies
 \[|\psi_0(x)|\leq |\psi_0(0)|+\kappa x^2.\]
 In particular, if  $\eta_n=Cn^{-p}$ for $p\in(0,1)$ and $C>0$,   then under the set up of Lemma~\ref{lemma: bound: xi n}, $\xin=\tilde G_n^{-1}(1-\eta_n)$ satisfies
 \[\sup_{x\in[-\xi_n,\xi_n]}|\psi_0(x)|=O_p((\log n)^2).\]
 \end{lemma}
 
 \begin{proof}[Proof of Lemma~\ref{lemma: g:  bound on psi}]
 Because $\psi_0\in\mathcal{SC}_0$, zero is the mode of $\psi_0$. Therefore, the upper bound on $\psi_0$ follows since $\psi_0(x)<\psi_0(0)$.
For the lower bound, first note that the  concavity of $\psi_0$  indicates that if $x\geq 0$ and $x\in\dom(\psi_0)$, then
 \[\psi_0(x)\geq \psi_0(0)+\psi'_0(x-)x.\]
 By our notation, $\psi_0'(x+)=\psi_0'(x)$. 
 Noting Assumption~\ref{assump: L} implies $\psi_0'(x-)\geq \psi_0'(0-)-\kappa x$, we derive
 \[\psi_0(x)\geq \psi_0(0)+\psi_0'(0-)x-\kappa x^2.\]
Since $\psi_0'(0-)\geq 0$, the above yields $\psi_0(x)\geq \psi_0(0)-\kappa x^2$ for all $x\geq 0$.
Since $\psi_0$ is symmetric about zero, we derive that $\psi_0(x)\geq \psi_0(0)-\kappa x^2$ for all $x\in\RR$. In conjunction with the fact that $\psi_0(x)\leq \psi_0(0)$, the latter implies $|\psi_0(x)|\leq |\psi_0(0)|+\kappa x^2$ for all $x\in\RR$. Since  $P([-\xi_n,\xi_n]\subset\dom(\psi_0))\to 1$ by Lemma~\ref{lemma: xi: xi is in dop psi knot}, the rest of the proof follows noting 
 $\xi_n=O_p(\log n)$ for $\eta_n=n^{-2p/5}$ by Lemma~\ref{lemma: bound: xi n}.
 \end{proof}
 
 \begin{lemma}\label{lemma: bound: psi knot prime}
 Suppose $\psi_0\in\mathcal{SC}_0$ satisfies Assumption~\ref{assump: L}. Further suppose   $\eta_n$ is as in Lemma~\ref{lemma: g:  bound on psi} and $y_n>0$ satisfies $y_n=o_p(\eta_n)$. Then
 \[\sup_{x\in[-\xi_n-y_n,\xi_n+y_n]}|\psi_0'(x)|=O_p(-\log(\eta_n)).\]
 \end{lemma}
 
 \begin{proof}[Proof of Lemma~\ref{lemma: bound: psi knot prime}]
 Since $\psi_0\in\mathcal{SC}_0$, $\psi_0'$ attains its maxima on any interval at the endpoints.
  Lemma~\ref{lemma: An inclusion} implies $[-\xi_n-y_n,\xi_n+y_n]\subset\dom(\psi_0)$ with probability approaching one. Therefore
 Assumption~\ref{assump: L} implies
 \begin{equation}\label{inlemma: lemma bound psi knot prime}
     \psi_0'(-\xin-y_n)\leq |\psi_0'(0)|+\kappa(|\xin|+y_n).
 \end{equation}
  Rest of the proof follows from Lemma~\ref{lemma: bound: xi n} and the fact that $y_n=o_p(1)$.
%   Also, for any $x>0$,
% \begin{align}\label{rate of phiknot prime}
% |\psi_0'(-x)-\psi_0'(0)|=\psi_0'(-x)-\psi_0'(0)=\dint_{-x}^{0}\psi_0''(t)dt\leq \kappa x
% \end{align}
% where we used the fact $\psi_0'$ is non-increasing and non-negative on the negative real line.
% Because $\psi_0$ is symmetric about zero and unimodal,  we have $\psi_0'(0)=0$, which implies,
% $\psi_0'(-x)\leq Cx$. Thus,
% \[|\psi_0'(\xin+|\td|)|\leq C(\xin+|\td|)=O_p(-\log(\eta_n))\]
% by Lemma~\ref{lemma: bound: xi n}.
% Because $\psi_0$ is concave and symmetric about zero, $\psi'_0$ is an odd function. Therefore, on any interval, $\psi_0'$ attains maxima at the endpoints, implying
% \[\sup_{x\in[-\xi_n-|\td|,\xi_n+|\td|}|\psi_0'(x)|=\psi_0'(\xin+|\td|).\]
 \end{proof}

 \begin{lemma}\label{lemma: g/g(-)}
 Under the set up of Theorem~\ref{theorem: main: one-step: full}, there exists $C>0$ so that if $b>0$ satisfies $[-b-|\td|,b+|\td|]\subset\iint(\dom(\psi_0))$, then
 \[\sup_{t\in[-b,b]}\frac{g_0(t)}{g_0(t+\td)}\leq  e^{|\td|(C+\kappa b+\kappa|\td|)}.\]
 \end{lemma}
 
 \begin{proof}[Proof of Lemma~\ref{lemma: g/g(-)}]
Recalling that we use $\psi_0'$ to denote the right derivative of $\psi_0$, we obtain
\begin{align*}
   \frac{g_0(t)}{g_0(t+\td)}= &\ \exp(\psi_0(t)-\psi_0(t+\td))\\
   \stackrel{(a)}{\leq} &\ \exp(|\td|\max\{|\psi_0'(t)|,|\psi_0'(t+\td)|\})
\end{align*}
where (a) follows from \eqref{inlemma: T2: DCT}.
If $t,\ t-\td\in \iint(\dom(\psi_0))$, by
 Assumption~\ref{assump: L}, it also holds that  $|\psi_0'(t)|\leq \psi_0'(0-)+\kappa|t|$ and $|\psi_0'(t+\td)|\leq \psi_0'(0-)+\kappa|t|+\kappa|\td|$. Thus for $C=\psi'_0(0-)$, we obtain that
 \[\frac{g_0(t)}{g_0(t+\td)}\leq e^{|\td|(C+\kappa |t|+\kappa|\td|)},\]
 from which, the result follows.
 
%  Now Lemma~\ref{lemma: An inclusion} implies $[-\xin-|\td|,\xin+|\td|]\subset \dom(\psi_0)$ with probability tending to one. 
% Also, noting $\xin=O_p(\log n)$ by Lemma~\ref{lemma: bound: xi n} and $|\td|=O_p(n^{-1/2})$, we obtain that $|\xin|+|\td|$ is $O_p(\log n)$. Thus,
%  \[\sup_{t\in[-|\xin|-|\td|,|\xin|+|\td|]}\frac{g_0(t-\td)}{g_0(t)}\leq C e^{O_p((\log n)n^{-1/2})}=O_p(1).\]
 \end{proof}
 \subsubsection{\textbf{Lemmas on $\hln$:}}
 \begin{lemma}\label{lemma: Laha_Nilanjana 31}
 % \todo[inline]{Here the exponent of $\log n$ will change}
Suppose $\hn$ satisfies Condition~\ref{condition: on hn} and Condition~\ref{cond: hellinger rate}  with $p\in(0,1)$. Further suppose $a_n$ and $\hn$ satisfies
\begin{equation}\label{inlemma: statement: os: nilanjana 31}
   \sup_{x\in[-a_n, a_n]} (|\psi_0(x)|+|\hln(x)|)=O_p((\log n)^2). \end{equation}
Then
\begin{align*}
    \dint_{-a_n}^{a_n} (\hln(x)-\psi_0(x))^2 g_0(x)dx=&O_p((\log n)^4n^{-2p}),\\
     \dint_{-a_n}^{a_n} (\hln(x)-\psi_0(x))^2 \hn(x)dx=&O_p((\log n)^4n^{-2p}).
\end{align*}
In particular, if  $\hn\in\mathcal{SLC}_0$, then \eqref{inlemma: statement: os: nilanjana 31} holds with $a_n=\xin(\tilde G_n)=\tilde G_n^{-1}(1-\eta_n)$, where $\eta_n=Cn^{-2p/5}$ for some $C>0$. 
 \end{lemma}

 \begin{proof}[Proof of Lemma~\ref{lemma: Laha_Nilanjana 31}]
 We first invoke an algebraic fact. For  any $x,y>0$,
 \[(\sqrt x-\sqrt y)^2=\min(x,y)\lb \sqrt{\frac{\max(x,y)}{\min(x,y)}}-1\rb^2=\min(x,y)\slb e^{|\log x-\log y|/2}-1\srb^2.\]
 Since  for any $z>0$, $z$ and $z^2/2$ are bounded above by $e^z-1$, it follows that
 \[\slb e^{|\log x-\log y|/2}-1\srb^2\geq (\log x-\log y)^2/4, (\log x-\log y)^4/8^2.\]
 Thus
 \begin{align}\label{inlemma: Nilanjana 31}
    4\dint_{-a_n}^{a_n}(\sqrt{\hn(x)}-\sqrt{g_0(x)})^2dx\geq &\  \dint_{-a_n}^{a_n}\min(\hn(x), g_0(x))(\hln(x)-\psi_0(x))^2dx\nn\\ 
    8^2\dint_{-a_n}^{a_n}(\sqrt{\hn(x)}-\sqrt{g_0(x)})^2dx\geq &\  \dint_{-a_n}^{a_n}\min(\hn(x), g_0(x))(\hln(x)-\psi_0(x))^4dx.
 \end{align}

Therefore,
 \begin{align*}
 \MoveEqLeft \dint_{-a_n}^{a_n}(\hln(x)-\psi_0(x))^2g_0(x)dx \\
 = &\ \dint_{-a_n}^{a_n}(\hln(x)-\psi_0(x))^2\slb g_0(x)-\min(g_0(x),\hn(x))\srb dx\\
 &\ +\dint_{-a_n}^{a_n}(\hln(x)-\psi_0(x))^2\min(g_0(x),\hn(x)) dx\\
 \stackrel{(a)}{=}&\  \dint_{-a_n}^{a_n}(\hln(x)-\psi_0(x))^2( g_0(x)-\hn(x))1_{[\hn<g_0]} dx+O_p(n^{-2p})\\
=&\  \underbrace{\dint_{-a_n}^{a_n}(\hln(x)-\psi_0(x))^2( \sqrt{g_0(x)}-\sqrt{\hn(x)})^21_{[\hn<g_0]}dx}_{T_1}\\
 -2&\ \underbrace{\dint_{-a_n}^{a_n}(\hln(x)-\psi_0(x))^2 \sqrt{\hn(x)}(\sqrt{\hn(x)}-\sqrt{g_0(x)})1_{[\hn<g_0]}dx}_{T_2}+O_p(n^{-2p})
 \end{align*}
 where (a) follows from \eqref{inlemma: Nilanjana 31} and Condition~\ref{cond: hellinger rate}. 
We can upper bound $|\psi_0(x)-\hln(x)|$ noting
\[\sup_{x\in[-a_n,a_n]}|\psi_0(x)-\hln(x)|\leq \sup_{x\in[-a_n, a_n]} (|\psi_0(x)|+|\hln(x)|)=O_p((\log n)^2) \]
by \eqref{inlemma: statement: os: nilanjana 31}.
Therefore $T_1\leq O_p((\log n)^4)H(\hn,g_0)^2 $, which is $O_p((\log n)^4 n^{-2p})$. On the other hand, noting $T_2$ can be written as 
\begin{align*}
    T_2=&\ \dint_{-a_n}^{a_n}(\hln(x)-\psi_0(x))^2 \sqrt{\min(\hn(x),g_0(x))}(\sqrt{\hn(x)}-\sqrt{g_0(x)})1_{[\hn<g_0]}dx,
\end{align*}
by an application of the Cauchy-Schwarz inequality, we derive 
\[|T_2|\leq \lb \dint_{-a_n}^{a_n}(\hln(x)-\psi_0(x))^4 \min(\hn(x),g_0(x))dx\rb^{1/2}H(\hn,g_0),\]
which, by \eqref{inlemma: Nilanjana 31} and Condition~\ref{cond: hellinger rate}, is $O_p(n^{-2p})$, thus completing the proof of the first part.

It remains to show that \eqref{inlemma: statement: os: nilanjana 31} holds when $\hn\in\mathcal{SLC}_0$ and $a_n=\tilde G_n^{-1}(1-\eta_n)$. 
 Lemma~\ref{lemma: g:  bound on psi} entails that this $a_n$ satisfies
 \begin{align}\label{inlemma: os: Nilanjana 31}
     \sup_{x\in[-a_n, a_n]} |\psi_0(x)|=O_p((\log n)^2).
 \end{align}
The proof of the current lemma then follows noting  Lemma~\ref{lemma: hn: lower bound  on hn} implies
\[\sup_{x\in[-\xin,\xin]}|\hln(x)| =O_p(\log n).\]

 \end{proof}
  \subsubsection{\textbf{Lemmas on  $\hln'$:}}
  \begin{lemma}\label{FI Lemma: Lemma 1}
 Let $\rho_n=\eta_n/\log n$.  Suppose $\hn$ is a log-concave density satisfying  Condition~\ref{condition: on hn} and Condition~\ref{cond: hellinger rate}. Let $a_n$ be a positive sequence  satisfying \eqref{inlemma: statement: os: nilanjana 31} such that $a_n=O_p(\log n)$,
 \begin{equation}\label{inlemma: statement: os: suport inclusion}
 P\slb [-a_n-\rho_n,a_n+\rho_n]\subset \iint(\dom(\psi_0))\cap\iint(\dom(\hln))\srb\to 1,
 \end{equation}
 \begin{flalign}\label{inlemma: statement: os: FI: an xin}
 P(\tilde G_n(-a_n)>\eta_n/4, 1-\tilde G_n(a_n)>\eta_n/4)\to 1. 
\end{flalign}
 Then 
 \[ \dint_{-a_n}^{a_n}\slb \tp(z)-\psi_0'(z)\srb^2dz=O_p( (\log n)^6n^{-4p/5}),\]
  \[\dint_{-a_n}^{a_n} \slb \tp(z)-\psi'_0(z)\srb^2 \mu_n(z)dz=O_p( (\log n)^6n^{-4p/5})),\]
 for any density $\mu_n$ such that $\|\mu_n\|_\infty=O_p(1)$, where $p$ is as in Condition~\ref{cond: hellinger rate}.  In particular, the lemma holds if
  $\hn\in\mathcal{SLC}_0$ and $a_n=\xin(\tilde G_n)=\tilde G_n^{-1}(1-\eta_n)$, where $\eta_n=Cn^{-2p/5}$ for some $C>0$.
 \end{lemma}
 
 \begin{proof}[Proof of Lemma~\ref{FI Lemma: Lemma 1}]
%   When
%  \[\hn(x)=\hnss(x)=\frac{\hf^{sm}(\bth+x)+\hf^{sm}(\bth-x)}{2},\]
% we have
%  \[\hnssp(x)=\frac{(\hf^{sm})'(\bth+x)-(\hf^{sm})'(\bth-x)}{2}.\]
% Let us denote $\hlns=\log\hf^{sm}$ and
% \[p_n(x)=\frac{\hf^{sm}(\bth+x)}{\hf^{sm}(\bth+x)+\hf^{sm}(\bth-x)}.\]
% Then for this $\hn$, $\hln'$ equals
%  \[\frac{\hnsp(x)}{\hns(x)}=p_n(x)\hlnsp(\bth+x)-(1-p_n(x))\hlnsp(\bth-x).\]
 
%  Since $\hlns$ is concave, for $\rho_n>0$, we have for all $z\in\dom(\hlns)$,
%  \[\frac{\hlns(z+\rho_n)-\hlns(z)}{\rho_n}\leq \hlnsp(z+)\leq \hlnsp(z-)\leq \frac{\hlns(z)-\hlns(z-\rho_n)}{\rho_n}.\]
%  Thus the above inequality holds for all $z\in[-\xin,\xin]$.  
%  Therefore,
%  \begin{align*}
%  \MoveEqLeft p_n(z)   \frac{\hlns(\bth+z+\rho_n)-\hlns(\bth+z)}{\rho_n}-(1-p_n(z))\frac{\hlns(\bth+z)-\hlns(\bth+z-\rho_n)}{\rho_n}\\
% \leq &\ \tp(z+)\leq \tp(z-)\\
% \leq &\  p_n(z)\frac{\hlns(\bth+z)-\hlns(\bth+z-\rho_n)}{\rho_n}-(1-p_n(z))\frac{\hlns(\bth+z+\rho_n)-\hlns(\bth+z)}{\rho_n}
%  \end{align*}
%  which implies
%  \begin{align*}
%      \MoveEqLeft \frac{p_n(z)\hlns(\bth+z+\rho_n)+(1-p_n(z))\hlns(\bth+z-\rho_n)-\hlns(\bth+z)}{\rho_n}\leq \tp(z+)\\
%      \leq &\ \tp(z-)\leq \frac{\hlns(\bth+z)-\lb p_n(z)\hlns(\bth+z-\rho_n)+(1-p_n(z))\hlns(\bth+z+\rho_n)\rb}{\rho_n}
%  \end{align*}
%  Also, 
%  Lemma~\ref{lemma: An inclusion} implies that $[\xin-|\td|,\xin+|\td|]\subset\dom(\psi_0)$ with high probability. hence,   the above inequality holds for any $z\in[\xin-|\td|,\xin+|\td|]$ with probability tending to one. 
 Since $\hln$ is concave and $\rho_n>0$, any $z\in\dom(\hln)$ satisfies
 \begin{equation}\label{inlemma: os: FI lemma: tp inequality}
     \frac{\hln(z+\rho_n)-\hln(z)}{\rho_n}\leq \tp(z+)\leq \tp(z-)\leq \frac{\hln(z)-\hln(z-\rho_n)}{\rho_n}.
 \end{equation}
Now suppose \eqref{inlemma: statement: os: suport inclusion} holds.
Then 
the quantities
\[\Delta_n^{+}(z)=\frac{\hln(z+\rho_n)-\hln(z)}{\rho_n}-\frac{\psi_0(z+\rho_n)-\psi_0(z)}{\rho_n},\]
and
\[\Delta_n^{-}(z)=\frac{\hln(z)-\hln(z-\rho_n)}{\rho_n}-\frac{\psi_0(z)-\psi_0(z-\rho_n)}{\rho_n}\]
are well defined for all $z\in[-a_n,a_n]$.
 Recalling   $\hln'(z)=\hln'(z+)$ and $\psi_0'(z)=\psi_0'(z+)$
by our notation, we can then show that  under \eqref{inlemma: statement: os: suport inclusion},
\begin{align*}
    \hln'(z)-\psi_0'(z)\stackrel{(a)}\leq &\ \frac{\hln(z)-\hln(z-\rho_n)}{\rho_n}-\frac{\psi_0(z)-\psi_0(z-\rho_n)}{\rho_n}\\
    &\ +\lb\frac{\psi_0(z)-\psi_0(z-\rho_n)}{\rho_n} -\psi_0'(z)\rb
   \stackrel{(b)}{\leq} \Delta_n^{-}(z)+\kappa\rho_n
\end{align*}
for all $z\in[-a_n,a_n]$, where (a) follows by \eqref{inlemma: os: FI lemma: tp inequality}, and (b) follows because
\[\frac{\psi_0(z)-\psi_0(z-\rho_n)}{\rho_n} -\psi_0'(z)\leq \rho_n^{-1}\bl\dint_{z-\rho_n}^{z} \slb \psi_0'(t)-\psi_0'(z)\srb dt\bl\leq \kappa\rho_n/2\]
since Assumption \ref{assump: L} applies on the set $[z-\rho_n,z]\subset\iint(\dom(\psi_0))$. 
 Similarly, we can show that 
 \[ \Delta_n^{+}(z)-\kappa\rho_n\leq\hln'(z)-\psi_0'(z)\quad\text{for all } z\in[-a_n,a_n],\]
 provided \eqref{inlemma: statement: os: suport inclusion} holds. Thus we have established that
 \begin{equation}\label{inlemma: os : FI lemma: delta inequalities}
     |\tp(z)-\psi_0'(z)|\leq \max\{\Delta_n^{+}(z), \Delta_n^{-}(z)\}+\kappa \rho_n
 \end{equation}
whenever \eqref{inlemma: statement: os: suport inclusion} holds. 
 Now observe that  the integral $\int_{-a_n}^{a_n}\Delta^{+}_n(z)^2dz$ is well defined under \eqref{inlemma: statement: os: suport inclusion}, and equals
 \begin{align*}
  \MoveEqLeft   \dint_{-a_n}^{a_n}\lb\frac{\hln(z+\rho_n)-\hln(z)}{\rho_n}-\frac{\psi_0(z+\rho_n)-\psi_0(z)}{\rho_n} \rb^2dz\\
  \leq &\ 2\dint_{-a_n}^{a_n}\lb\frac{\hln(z+\rho_n)-\psi_0(z+\rho_n)}{\rho_n}\rb^2dz+2\dint_{-a_n}^{a_n}\lb\frac{\hln(z)-\psi_0(z)}{\rho_n}\rb^2dz\\
  \stackrel{(a)}{\leq} &\  \frac{2}{\rho_n^2\min(\hn(a_n+\rho_n),\hn(-a_n+\rho_n))}\dint_{-a_n}^{a_n}(\hln(z+\rho_n)-\psi_0(z+\rho_n))^2\hn(z+\rho_n)dz\\
  &\ +\frac{2}{\rho_n^2\min(\hn(a_n),\hn(-a_n))}\dint_{-a_n}^{a_n}(\hln(z)-\psi_0(z))^2\hn(z)dz\\
 \stackrel{(b)}{=} &\  \frac{O_p((\log n)^4n^{-2 p})}{\rho^{2}_n \min(\hn(a_n+\rho_n),\hn(-a_n-\rho_n))},
 \end{align*}
 where (a) follows because $\hn$ being log-concave, and hence unimodal, attains minimum over an interval at either of the endpoints; and (b) follows from Lemma~\ref{lemma: Laha_Nilanjana 31} and the fact that $a_n$ and $\rho_n$ are positive. 
 Let us define
 \begin{align}\label{inlemma: def: os: myen}
   \e_n(\rho_n)=\frac{(\log n)^4n^{-2 p}}{\rho^{2}_n \min(\hn(a_n+\rho_n),\hn(-a_n-\rho_n))
   }.
   \end{align}
 Since \eqref{inlemma: statement: os: suport inclusion} holds with probability tending to one by our assumption, we can write 
  $\int_{-a_n}^{a_n}\Delta^{+}_n(z)^2dz=O_p(\e_n(\rho_n))$. 
Similarly, we can show that
$\int_{-a_n}^{a_n}\Delta^{-}_n(z)dz$ is $O_p(\e_n(\rho_n))$.
The above, combined with \eqref{inlemma: os : FI lemma: delta inequalities}, leads to 
\begin{align*}
    \dint_{-a_n}^{a_n}(\tp(z)-\psi_0'(z))^2dz\leq &\ 2\dint_{-a_n}^{a_n}\Delta_n^{-}(z)^2dz+2\dint_{-a_n}^{a_n}\Delta_n^{+}(z)^2dz +4\kappa^2\rho_n^2a_n\\
    =&\ O_p(\e_n(\rho_n))+O_p(\rho_n^2a_n).
\end{align*}
Note that $\rho_n^2a_n=O(n^{-4p/5}/\log n)$ because $\rho_n=\eta_n/\log n$ and $a_n=O_p(\log n)$ by our assumption. Also, $\int_{-a_n}^{a_n}(\tp(z)-\psi_0'(z))^2\mu_n(z)dz$ can be bounded by
\begin{flalign*}
 2\|f\|_{\infty} \lb \dint_{-a_n}^{a_n}\Delta_n^{-}(z)^2dz+\dint_{-a_n}^{a_n}\Delta_n^{+}(z)^2dz\rb+2\kappa^2\rho^2_n   
 \end{flalign*}
 which is $O_p(\e_n(\rho_n))+O_p(n^{-4/5}/(\log n)^2)$ because $\|\mu_n\|_\infty$ is $O_p(1)$ and $\rho_n$ equals $\eta_n/\log n$.
%  Similarly, we can show that
%  \begin{align*}
%  \dint_{-\xin}^{\xin}(\tp(z)-\psi_0'(z))^2g_0(z-\td)dz\leq &\ 2\|g_0\|_{\infty} \lb \dint_{-\xi_n}^{\xi_n}\Delta_n^{-}(z)^2dz+\dint_{-\xi_n}^{\xi_n}\Delta_n^{+}(z)^22dz\rb+2\kappa^2\rho^2_n\\
%  = \myen+O_p(\rho^2_n).   
%  \end{align*}
%  Since $\|\hn(z)\|_{\infty}=O_p(1)$ by Lemma~\ref{lemma: hn:  bound on hln}, we have
% \begin{align*}
%   \dint_{-\xin}^{\xin}(\tp(z)-\psi_0'(z))^2\hn(z)dz\leq &\  2\|\hn\|_{\infty} \lb \dint_{-\xi_n}^{\xi_n}\Delta_n^{-}(z)^2dz+\dint_{-\xi_n}^{\xi_n}\Delta_n^{+}(z)^22dz\rb + 2\kappa^2\rho^2_n\\
%   =&\ \myen+O_p(\rho^2_n). 
% \end{align*}
 To prove the first part of the lemma, it only remains to show that
 \begin{equation}\label{inlemma: os: rate of e n}
     \e_n(\rho_n)=O_p( (\log n)^6n^{-4p/5}).
 \end{equation}
Since $a_n-\rho_n\in\iint(\dom(\hln))$ under \eqref{inlemma: statement: os: suport inclusion},
 Fact~\ref{fact: f grtr than F} implies 
 \[\hn(-a_n-\rho_n)\geq \omega_n \tilde{G}_n (-a_n-\rho_n)=\omega_n\tilde{G}_n(-a_n)-\omega_n\dint_{-a_n-\rho_n}^{-a_n}\hn(z)dz\]
 under \eqref{inlemma: statement: os: suport inclusion},
 where $\omega_n$ is as in fact~\ref{fact: f grtr than F}. 
Note that
 \[\bl\dint_{-a_n-\rho_n}^{-a_n}\hn(z)dz\bl\leq \rho_n\|\hn\|_{\infty}.\]
Also since $P(\tilde G_n(-a_n)>\eta_n/4)\to 1$ by our assumption, the following hold with probability tending to one,
 \[\hn(-a_n-\rho_n)\geq  \omega_n(\eta_n/4-\rho_n\|\hn\|_\infty)\stackrel{(a)}{\geq}\omega_n\eta_n/8 ,\]
  where (a) follows because $\rho_n=o(\eta_n)$ and  $\|\hn\|_\infty=O_p(1)$ by Condition~\ref{condition: on hn} and Fact~\ref{fact: Lemma 1 of theory paper}.
  However, since $\omega_n\to_p\omega_0$ by Fact~\ref{fact: f grtr than F},  the last display implies 
  $\hn(-a_n-\rho_n)^{-1}$ is $O_p(1/\eta_n)$.
Similarly,   we can show that $\hn(a_n+\rho_n)^{-1}$ is $O_p(1/\eta_n)$.
  Thus
 \[\epsilon_n(\rho_n)=\frac{O_p((\log n)^4n^{-2p})}{\eta_n\rho_n^2}\]
 follows.
Since $\eta_n=n^{-2p/5}$ and $\rho_n=\eta_n/\log n$,  \eqref{inlemma: os: rate of e n} follows, thus completing the proof of the first part of Lemma~\ref{FI Lemma: Lemma 1}.

Now suppose $\hn\in\mathcal{SLC}_0$ and  $a_n=\xi_n(\tilde G_n)=\tilde G_n^{-1}(1-\eta_n)$. They satisfy \eqref{inlemma: statement: os: nilanjana 31} by Lemma~\ref{lemma: Laha_Nilanjana 31}. Also $\xin=O_p(\log n)$ by Lemma~\ref{lemma: bound: xi n}. Noting $\rho_n=o(\eta_n)$, \eqref{inlemma: statement: os: suport inclusion} follows from  Lemma~\ref{lemma: xi: tilde xi n}
and Lemma~\ref{lemma: An inclusion}.
% and implies
%  \[P\slb [-a_n-\rho_n,a_n+\rho_n]\subset\iint(\dom(\hln))\srb\to 1,\]
%  where Lemma~\ref{lemma: An inclusion} implies as $n\to\infty$,
%  \begin{align}\label{inlemma: os: 31: inclusion}
%   P\slb [-a_n-\rho_n,a_n+\rho_n]\subset\iint(\dom(\psi_0))\srb\to 1, 
%  \end{align}
%  and thus  holds in this case. 
 Since \eqref{inlemma: statement: os: FI: an xin} trivially holds,  second part of Lemma~\ref{FI Lemma: Lemma 1} also follows.
 \end{proof}

%  \subsubsection{Rate of $\rho_n$}
%  The next lemma finds the optimal rate of $\rho_n$.
%  \begin{lemma}\label{lemma: bounding the myen denominator}
%  Suppose $\eta_n=C n^{-2p/5}$ for some $C>0$ and $p>0$.
%  Consider the set up of Lemma~\ref{FI Lemma: Lemma 1}.  Then taking $\rho_n =\eta_n/\sqrt{\log n}$, we have
%   \[\e_n(\rho_n)+O(\rho_n^2)=O_p( (\log n)^3n^{-4p/5})\]
%   where $\e_n(\rho_n)$ is as defined in \eqref{inlemma: def: os: myen}. 
%  \end{lemma}
%  \begin{proof}

%  \end{proof}

 %\subsubsection{Consistency of Fisher's information}
 \begin{lemma}\label{lemma: consistency of FI: 2}
 Suppose $\hn$ satisfies Condition~\ref{condition: on hn} and Condition~\ref{cond: hellinger rate}. Let $a_n$ be a sequence of positive random variables  satisfying 
 \begin{equation}\label{inlemma: statement: FI: an}
    a_n=O_p(\log n),\quad 
P\slb a_n\in\iint(\dom(\psi_0))\srb\to 1,\quad\text{and}\quad
 G_0(a_n)\to_p 1.
 \end{equation}
 Further suppose that $\hn$ and $a_n$ satisfy
 \begin{align}\label{inlemma: statement: FI}
     \dint_{-a_n}^{a_n}(\tp(z)-\psi_0'(z))^2\hn(z)dz=O_p((\log n)^6n^{-4p/5}),
 \end{align}
where $p$ is as in Condition~\ref{cond: hellinger rate}.
Then 
 $\int_{-a_n}^{a_n}\tp(z)^2\hn(z)dz\to_p \I$. In addition, if $\hn\in\mathcal{SLC}_0$, then  $a_n=\xin(\tilde G_n)=(\tilde G_n)^{-1}(1-\eta_n)$ where $\eta_n=Cn^{-2p/5}$ for some $C>0$.
\end{lemma} 
 
 \begin{proof}[Proof of Lemma~\ref{lemma: consistency of FI: 2}]
 Note that
 \begin{align*}
  \MoveEqLeft    \dint_{-a_n}^{a_n}\slb\tp(z)^2\hn(z)-\psi_0'(z)^2g_0(z)\srb dz\\
  = &\ \dint_{-a_n}^{a_n}\slb(\tp(z)-\psi_0'(z))^2\hn(z)+2\psi_0'(z)\tp(z)\hn(z)-\psi_0'(z)^2\hn(z)-\psi_0'(z)^2g_0(z)\srb dz\\
  =&\ \underbrace{ \dint_{-a_n}^{a_n}(\tp(z)-\psi_0'(z))^2\hn(z)dz}_{T_1}+ 2\underbrace{\dint_{-a_n}^{a_n}\psi_0'(z)(\tp(z)-\psi_0'(z))\hn(z)dz}_{T_2}\\
  &\ +\underbrace{\dint_{-a_n}^{a_n}\psi_0'(z)^2(\hn(z)-g_0(z))dz}_{T_3},
 \end{align*}
 It is clear that by our assumption,
 $T_1=O_p((\log n)^6n^{-4p/5})$, which is $o_p(1)$.

Because $\psi_0'$ is a non-increasing odd function, on any interval, $|\psi_0'|$ attains its maximum at both end points. Therefore,
\begin{align*}
|T_2|
   \leq &\ |\psi_0'(a_n)|\dint_{-a_n}^{a_n}|\tp(z)-\psi_0'(z)|\hn(z)dz \\
  \stackrel{(a)}{\leq} &\ |\psi_0'(a_n)|\lb\dint_{-a_n}^{a_n}(\tp(z)-\psi_0'(z))^2\hn(z)dz \rb^{1/2},
\end{align*}
where (a) follows by the Cauchy-Schwarz inequality. Thus, $|T_2|\leq |\psi_0'(a_n)|\sqrt{T_1}$.
 However,  Assumption~\ref{assump: L} implies that $|\psi_0'(a_n)|\leq |\psi_0'(0)|+\kappa a_n=O_p(\log n)$ provided $a_n\in \iint(\dom(\psi_0))$. By our assumption on $a_n$, the latter holds with probability tending to one.  Hence, 
 \[T_2=O_p(\log n) \sqrt{T_1}=O_p((\log n)^{4}n^{-2p/5})=o_p(1).\]
 
 Finally,  using the fact  $|\psi_0'(a_n)|=O_p(\log n)$ again, we bound $|T_3|$ by
 \begin{align*}
   O_p((\log n)^2) d_{TV}(\hn, g_0)
    \stackrel{(a)}{\leq}  & O_p((\log n)^2)H(\hn,g_0)\stackrel{(b)}{=} O_p((\log n)^2 n^{-p})=o_p(1),
 \end{align*}
 where (a) follows from Fact~\ref{fact: dTV and hellinger} and (b) follows noting  $H(\hn,g_0)=O_p(n^{-p})$ by Condition~\ref{cond: hellinger rate}. Thus we have shown that
 \[\dint_{-a_n}^{a_n}\slb\tp(z)^2\hn(z)-\psi_0'(z)^2g_0(z)\srb dz=o_p(1).\]
 
Since $a_n\to_p G_0^{-1}(1)$ by our assumption, noting $G_0^{-1}(0)=-G_0^{-1}(1)$, we also obtain $-a_n\to_p G_0^{-1}(0)$. Hence,
 \[\dint_{-a_n}^{a_n}\psi_0'(z)^2g_0(z)dz\to_p\dint_{G_0^{-1}(0)}^{G_0^{-1}(1)}\psi_0'(z)^2g_0(z)dz=\I,\]
 which completes the proof of the first part of Lemma~\ref{lemma: consistency of FI: 2}.
 Second part of Lemma~\ref{lemma: consistency of FI: 2} follows noting $\xin=O_p(\log n)$ by Lemma~\ref{lemma: bound: xi n}, $\xin\in\iint(\dom(\psi_0))$ with probability tending to one by Lemma~\ref{lemma: xi: xi is in dop psi knot},  $\xin\to_p G_0^{-1}(1)$ by Lemma~\ref{lemma: xi: xi goes to 1}, and $\hn\in\mathcal{SLC}_0$ satisfies \eqref{inlemma: statement: FI}  by Lemma~\ref{FI Lemma: Lemma 1}.
 \end{proof}

 \begin{lemma}\label{lemma: tilde psi prime bound}
 Consider the set up of Theorem~\ref{theorem: main: one-step: full}.
 Suppose $\eta_n=Cn^{-2p/5}$, where $C>0$ and $p$ is as in Condition~\ref{cond: hellinger rate}. Let $y_n$ be a sequence of positive random variables such that
  $P(|y_n|\leq \eta_n/(2g_0(0)))\to 1$. Then 
  \[ \sup_{x\in[-\xi_n-y_n,\xi_n+y_n]}|\tp(x)|=O_p(\eta_n^{-1/2})=O_p(n^{p/5}).\]
 \end{lemma}
 
 \begin{proof}[Proof of Lemma~\ref{lemma: tilde psi prime bound}]
Let $q\in(0,1/2)$.
 Since $\hn$, being log-concave, is positive on $\iint(J(\tilde G_n))$, using Fact~\ref{fact: bobkov big} we obtain that
\[
\dint_{\tilde{G}^{-1}_n(q/2) }^{\tilde{G}^{-1}_n(q)} \tp(x)^2\hn(x)dx=\dint_{q/2}^q\tp(\tilde{G}^{-1}_n(z))^2dz.
\]
Note that $\tp$ is non-increasing and positive on $(-\infty, -x]$, and $\hn(-x)$ is positive and non-decreasing on $(-\infty, -x]$. Thus $\tp\circ \tilde{G}^{-1}_n$ is non-increasing. Therefore
 \[ q\tp(\tilde{G}^{-1}_n(q))^2/2 \leq \dint_{\tilde{G}^{-1}_n(q/2) }^{\tilde{G}^{-1}_n(q)} \tp(x)^2\hn(x)dx \stackrel{(a)}{\leq} \dint_{\tilde{G}^{-1}_n(q/2) }^{\tilde{G}^{-1}_n(1-q/2)} \tp(x)^2\hn(x)dx \]
where (a) follows because $q<1-q/2$ for all $q\in(0,1/2)$.
Suppose $q=\eta_n/2$. Note that $\tilde{\xi}_n=-\tilde{G}^{-1}_n(\eta_n/2)$ and $|\tp|$ is symmetric about zero. Then the last display leads to
 \[\tp(\tilde{\xi}_n)^2\leq \frac{4\dint_{\tilde{G}^{-1}_n(\eta_n/4) }^{\tilde{G}^{-1}_n(1-\eta_n/4)} \tp(x)^2\hn(x)dx}{\eta_n}.\]
 From Lemma~\ref{lemma: consistency of FI: 2} it follows that the integral converges in probability to $\I$. Therefore, $|\tp(\tilde{\xi}_n)|=O_p(\eta_n^{-1/2})$ which implies
 \begin{equation}\label{inlemma: tp bound}
     \sup_{x\in[-\tilde{\xi}_n,\tilde{\xi}_n]}|\tp(x)|=O_p(\eta_n^{-1/2}).
 \end{equation}
 
 The rest of the proof follows from \eqref{inlemma: tp bound} and Lemma~\ref{lemma: xi: tilde xi n}.
  
%  Consider $y$ to be $y_n$ and let $y_n\to\infty$.
%  From the proof of Lemma~\ref{FI Lemma: Lemma 1} we obtain that 
%  \[\dint_{-y_n}^{y_n}\tp(z)^2\hn(z)dz\to \I\]
%  as long as 
%  \[y_n\frac{n^{-2/5}}{\rho\hn(y_n+|\rho|)^{1/2}}=o_p(1)\]
%  where $|\rho|$ is a small number. Suppose $|\rho|$ is smaller than $y_n$, which is possible because we let $y_n$ grows to $G_0^{-1}(1)$. Then the above relation is satisfied if 
%  \[\frac{2y_n}{\hn(2y_n)}\leq \frac{n^{2/5}}{\log n}.\]
  \end{proof}

 \begin{lemma}\label{Lemma: L2 norm of hn}
 Consider the set up of Theorem~\ref{theorem: main: one-step: full}.
Then  $\|h_n\|_{P_0, 2}^2=O_p(n^{-4p/5}(\log n)^{3})$, where $h_n$ is as defined in \eqref{inlemma: def: main: hn}.
\end{lemma}

\begin{proof}[Proof of Lemma~\ref{Lemma: L2 norm of hn}]
Note that $\|h_n\|_{P_0,2}^2$ equals
\begin{align*}
\MoveEqLeft \dint_{-\xi_n}^{\xi_n}(\hln'(z)-\psi_0'(z-\td))^2 f_0(z+\bth)dz\\
 \leq &\  2\dint_{-\xi_n}^{\xi_n}\slb\hln'(z)-\psi_0'(z)\srb^2f_0(z+\bth)dz+2\dint_{-\xi_n}^{\xi_n}\slb \psi_0'(z-\td)-\psi_0'(z)\srb^2f_0(z+\bth)dz\\
 \leq &\ 2\underbrace{\dint_{-\xi_n}^{\xi_n}\slb\hln'(z)-\psi_0'(z)\srb^2g_0(z-\td)dz}_{T_1}+2\underbrace{\dint_{-\xi_n}^{\xi_n}\slb \psi_0'(z-\td)-\psi_0'(z)\srb^2g_0(z-\td)dz}_{T_2}.
\end{align*}
$T_1$ is $O_p(n^{-4p/5}(\log n)^6)$ by Lemma~\ref{FI Lemma: Lemma 1}. Since Assumption~\ref{assump: L} implies $\psi_0'$ is Lipschitz with constant $\kappa$ on its domain, Lemma~\ref{lemma: An inclusion} entails that  $T_2$ is bounded by 
\[2\kappa^2\td^2\dint_{-\xi_n}^{\xi_n}g_0(z-\td)dz\]
which is $O_p(n^{-1})$ since $\td=O_p(n^{-1/2})$. Hence, the proof follows.
\end{proof}

\subsubsection{\textbf{Lemma on consistency of Fisher information $\hin(\eta_n)$:}}
% \todo[inline]{Ekhane $\log n$-er ordere vul ache. See the new paper. (There is no $\log n$ here).}
\begin{lemma}\label{lemma: consistency of FI} 
Under the set up of Theorem~\ref{theorem: main: one-step: full}, $\hin(\eta_n)\to_p\I$ where $\hin(\eta_n)$ is as defined in \eqref{27indef}.
\end{lemma}
\begin{proof}[of Lemma~\ref{lemma: consistency of FI} ]
Denoting $\td=\theta_0-\bth$, we observe that
\begin{align}\label{inlemma: consistency of FI}
\MoveEqLeft|\hin(\eta_n)-\Ig(\eta_n)|\nn\\
\leq &\ \underbrace{\bl\dint_{\bth-\xi_n}^{\bth+\xi_n}\hln'(x-\bth)^2d(\Fm-F_0)(x)\bl}_{T_1}
 +\underbrace{\bl\dint_{-\xi_n}^{\xi_n} \hln'(x)^2 \slb g_0(x-\td)-g_0(x)\srb dx\bl }_{T_2}\nn\\
 &\ + \underbrace{\bl\dint_{-\xi_n}^{\xi_n} \hln'(x)^2 \slb g_0(x)-\hn(x)\srb dx\bl}_{T_3}+\underbrace{\bl\dint_{-\xi_n}^{\xi_n} \hln'(x)^2 \hn(x) dx-\I\bl}_{T_4}
\end{align}
Let us consider the term $T_1$ first. Denoting  $M_n=C n^{p/5}$ as in the 
   the proof of the first step of Theorem~\ref{theorem: main: one-step: full}, we recall the class of functions $\mathcal U_n(M_n)$  defined in \eqref{intheorem: def: U n}. 
%  Recall also for some $C>0$.
%  Denote by  the set $[\bth-\xin,\bth+\xin]$.
%  We define the monotone function $\widehat u_n:\RR\to\RR$ so that $\widehat u_n=\tp(\mathord{\cdot}-\bth)$ on $\mathcal C_n$ but $\widehat u_n=\tp(\bth-\xi_n)$ for $x<\bth-\xi_n$ and $\widehat u_n=\tp(\bth+\xi_n)$ for $x>\bth+\xi_n$. Therefore
%  \[T_1=\edint \widehat u_n(x)^2 1_{\mathcal C_n}(x)d(\mathbb F_n-F_0)(x).\]
%  Lemma~\ref{lemma: bound: xi n} and Lemma~\ref{lemma: tilde psi prime bound} imply that $\|\widehat u_n\|_{\infty}=O_p(n^{p/5})$.
%  Therefore, given $\e>0$,  $C$ can be chosen so that $P(\widehat u_n\in \mathcal{U}_n)>1-\e$  for all sufficiently large $n$. 
 
%  Lemma~\ref{lemma: An inclusion} indicates that as $n\to\infty$,  $[-\xi_n-\td,\xi_n-\td]\subset \iint(\dom(\psi_0))$ with probability tending to one. Since $\dom(\phi_0)=\dom(\psi_0)+\theta_0$, it follows that
%  \[[\bth-\xi_n,\bth+\xi_n]=\theta_0+[-\td-\xi_n,\xi_n-\td]\subset \iint(\dom(\phi_0))\]
%  with high probability tending to one.
%  Lemma~\ref{lemma: bound: xi n} implies that $\xi_n=O_p(\log n)$. Since $\bth-\th_0=O_p(n^{-1/2})$ as well, we see that for sufficiently large $C$,
% \begin{equation}\label{intheorem: donsker: function class: indicator }
% \lim_{n\to\infty}  P\slb  [\bth-\xi_n,\bth+\xi_n]\subset [\th_0-C\log n,\th_0+C\log n]\cap \iint(\dom(\phi_0))\srb= 1.
% \end{equation}
In the same way we showed that $h_n\in\mathcal{H}_n(C)$ with high probability in the proof of the first step of Theorem~\ref{theorem: main: one-step: full},
we can show that the function defined by 
\[\tilde {h}_n(x)=\tp(x-\bth)^21_{[\bth-\xin,\bth+\xin]}(x)\]
 is a member of the class
\begin{align*}
 \mathcal{V}_n(C)=\bigg\{h:\RR\mapsto & \RR\ \bl \ h(x)=u(x)^2 1_{[r_1,r_2]}(x),\ u\in\mathcal{U}_{n}(M_n),\\
 &\
 [r_1,r_2]\subset [\th_0-C\log n,\th_0+C\log n]\cap\iint(\dom(\phi_0)) \bigg \}
 \end{align*}
 with high probability for all  large $n$ provided $C>0$ is sufficiently large. Using \eqref{inlemma: finite entropy increasing}, \eqref{inlemma: finite entropy indicator functions} and following some standard calculations,
  we can show that
\[\sup_Q\log N_{[\ ]}(\e,\mathcal{V}_n(C),L_2(Q)\lesssim M_n^2\e^{-1},\]
 where the supremum is over all probability measures on $\RR$.
 % We want to find the covering number of $\mathcal{U}_{s,n}=\{u^2:u\in\mathcal U_n\}$, however. To that end, we define two more classes,
%  \[\mathcal{U}_{n}^+=\lbs u:\RR\mapsto [0,2M_n]\ :\ u\text{ is non-increasing}\rbs\] and 
%  $\mathcal{U}_{s,n}^{+}=\{u^2:u\in\mathcal U_n^+\}$. 
%  Because 
%  \[\|f^2-g^2\|_{Q,2}\leq (\|f\|_\infty+\|g\|_\infty)\|f-g\|_{Q,2}\]
%  for  any probability measure $Q$,
%  an $4M_n\e$ bracket of $\mathcal U^+_{s,n}$ may be obtained by just squaring any bracket of $\mathcal U^+_{n}$. 
%  Noting $u\in\mathcal U_n$ satisfies
%  $u^2=(u+M_n)^2-2M_nu+M_n^2$, we obtain that 
%  \[ N_{[\ ]}(2\e,\mathcal{U}_{s,n},L_2(Q))\leq  N_{[\ ]}(\e,\mathcal{U}^{+}_{s,n},L_2(Q)) N_{[\ ]}(\e/(2M_n),\mathcal{U}_{n},L_2(Q))\]
%   Since $\log  N_{[\ ]}(\e,\mathcal{U}_{s,n},L_2(Q))\leq C'M_n/\e$ for some $C'>0$, we have
%  \[\log  N_{[\ ]}(\e,\mathcal{U}_{s,n},L_2(Q))\leq \e^{-1}C'M_n^2.\]
 Because bracketing number is larger than covering number, it also follows that
 \[\sup_Q\log N(\e,\mathcal{V}_{n}(C),L_2(Q))\lesssim M_n^2\e^{-1}.\] 
 The definition of $\mathcal U_n(M_n)$ in \eqref{intheorem: def: U n} implies that the functions in $\mathcal{V}_n(C)$ are uniformly bounded by $M_n^2$.
Since for any fixed $\e>0$,
 \[M_n^2\sup_Q\log N(\e,\mathcal{V}_{n}(C),L_2(Q))\lesssim M_n^4\e^{-1}=O(n^{-4p/5}),\] 
Fact~\ref{fact: GC} leads to $E\|\mathbb F_n-F_0\|_{\mathcal{V}_n(C)}=o(1)$. Thus  Markov's inequality yields that $\|\mathbb F_n-F_0\|_{\mathcal{V}_n(C)}=o_p(1)$. Since for large $C$, $P(\tilde{h}_n\in\mathcal{V}_n(C))$ with high probability,  it can be shown that $\int \tilde {h}_n d(\mathbb F_n-F_0)\to_p 0$, which establishes $T_1=o_p(1)$.
 
 Since the supremum of $|\tp|$ over  $[-\xin,\xin]$ is $O_p(n^{p/5})$  by Lemma~\ref{lemma: tilde psi prime bound}, 
 we obtain that
  \[T_2\leq O_p(n^{2p/5})d_{TV}(g_0(\cdot-\td),g_0)\stackrel{(a)}{\leq }O_p(n^{2p/5}) H(g_0(\cdot-\td),g_0)\stackrel{(b)}{=}O_p(\td n^{2p/5}),\]
  which is $o_p(1)$ because $\td=O_p(n^{-1/2})$ and $p\in(0,1)$. Here (a) and (b) follow from Fact~\ref{fact: dTV and hellinger} and Fact~\ref{fact: helli fknot}, respectively.
%   where by Fubini's Theorem,  
%   \begin{align*}
%   \dint_{-\xin}^{\xin}\dint_{x-|\td|}^{x+|\td|}|g_0'(z)|dzdx  =2|\td|\dint_{-\xin-|\td|}^{\xin+|\td|}|g_0'(z)|dz\leq 2|\td| g_0(0)=O_p(n^{-1/2}).
%   \end{align*}
%   Because $p<1$, $T_2=o_p(1)$ follows.
 In a similar way we can show that
  \[T_3\leq O_p(n^{2p/5})d_{TV}(\hn,g_0)\leq O_p(n^{2p/5})H(\hn,g_0),\]
  which is $O_p(n^{-3p/5})$ by Condition~\ref{cond: hellinger rate}.
%  Fact~\ref{fact: Lemma 1 of theory paper} and Condition~\ref{condition: on hn} imply that  $\|\hn\|_\infty$ and $g_0$ are bounded. 
% Noting $\|u\|_{1}\leq \|u\|_2$ for any function $u$, we thus deduce that  $T_3=O_p(n^{2p/5}) H(\hn, g_0)$.
%   Since $H(\hn, g_0)$ is $o_p(n^{-p})$, we obtain that $T_3=o_p(1)$. 
   Finally, noting $T_4$ is also $o_p(1)$ by Lemma~\ref{lemma: consistency of FI: 2}, the proof follows from \eqref{inlemma: consistency of FI}.
\end{proof}

% \subsubsection*{Proof of Lemma \ref{Prop: the L1 convergence of the density estimators of one-step estimators: model: strong}}

 %%%%%%%%%%%%%%%%%%%%%%%%%%%%%%%%%%%%%%%%%%%%%
 %%%%%%%%%%%%%%%%%%%%%%%%%%%%%%%%%%%%%%%%%%%%%
 %%%%%%%%%%%%%%%%%%%%%%%%%%%%%%%%%%%%%%%%%%%%%

 %%%%%%%%%%%%%%%%%%%%%%%%%%%%%%%%%%%%%%%%%%%%%%%%%%%%%%%%%%%%%%%%%%%%%%
 
\section{Proof of proposition~\ref{prop: partial and geo suffices}}
\label{app: proof of proposition 1}

We will first show that $\widehat{g}_{\bth}$ and $\hn^{geo, sym}$ satisfy Condition~\ref{condition: on hn}. Then using this result, we will show in  Lemma~\ref{lemma: hellinger: charles} and Lemma~\ref{lemma: hellinger: geometric} that $\widehat{g}_{\bth}$ and $\hn^{geo, sym}$ satisfy Condition~\ref{cond: hellinger rate}, respectively. To show that Condition~\ref{condition: on hn} holds for these two densities, we prove a general Proposition which states that Condition~\ref{condition: on hn} holds for all the density estimators of $g_0$ we have discussed so far. 

\begin{proposition}\label{prop: the L1 convergence of the density estimators of one-step estimators}
Suppose $f_0\in\mathcal{P}_1$ and $\bth$ is a  consistent estimator of $\th_0$. Then   $\hn=$ $\widehat{g}_{\bth}$, $\hf(\bth\pm\mathord{\cdot})$, $\hn^{geo,sym}$,  $\hts(\bth\pm\mathord{\cdot})$,  and $\hnss$
 satisfy Condition~\ref{condition: on hn}.
\end{proposition}

The key step in proving Proposition~\ref{prop: the L1 convergence of the density estimators of one-step estimators} is showing that the $L_1$ consistency  in Condition~\ref{condition: on hn}(A) holds, which is established by Lemma~\ref{lemma: L1 convergence: one step density estimators}. The proof of 
Lemma~\ref{lemma: L1 convergence: one step density estimators} can be found in  Appendix~\ref{sec: adlemma: os: conditions}.

 \begin{lemma}\label{lemma: L1 convergence: one step density estimators}
Suppose $\bth\to_p \th_0$, and $\hn$ is one among 
$\hts(\bth\pm\mathord{\cdot})$, $\hf(\bth\pm\mathord{\cdot})$, $\hnss$, $\hn^{geo,sym}$, and $\widehat{g}_{\bth}$. Then $\|\hn-g_0\|_1\to_p 0$.  
\end{lemma}

Now we are ready to prove  Proposition~\ref{prop: the L1 convergence of the density estimators of one-step estimators}.

\begin{proof}[Proof of Proposition~\ref{prop: the L1 convergence of the density estimators of one-step estimators}]
As in the proof of Lemma~\ref{lemma: L1 convergence: one step density estimators}, one can show that it suffices to prove   Proposition~\ref{prop: the L1 convergence of the density estimators of one-step estimators} when $\bth\as\th_0$, and $y_n\as 0$.  Hence, in what follows,  we assume that $\bth\as\th_0$, and $y_n\as 0$.
First we will verify Condition~\ref{condition: on hn} when $\hn\in\mathcal{LC}$. Note that, this covers the case of $\widehat{g}_{\bth}$, $\hf(\bth\pm\mathord{\cdot})$, $\hn^{geo,sym}$, and $\hts(\bth\pm\mathord{\cdot})$. We will consider the case of $\hnss$ separately because the latter is not log-concave.

% Therefore, we will consider this case first. So, let us assume, $\hn$ is one of the above mentioned log-concave estimator of $g_0$. 
Assuming $\hn\in\mathcal{LC}$, to verify part A of Condition~\ref{condition: on hn}, we first note that
\begin{align*}
 \|\hn(\cdot+y_n)-\q\|_1\leq  \|\hn(\cdot +y_n)-\q(\cdot +y_n)\|_1+
\|\q(\cdot+y_n)-\q\|_1,
\end{align*}
whose first term converges to zero almost surely by 
 Lemma~\ref{lemma: L1 convergence: one step density estimators}.
Also, since $\q$ is continuous, $\q(x+y_n)$ converges to $\q(x)$ for each $x\in\RR$.
Therefore the second term also converges to zero almost surely by Glick's Theorem \citep[Theorem 2.6,][]{devroye1987}.
   Thus we obtain that $\|\hn(\cdot+y_n)-\q\|_1\as 0$.
   Since $\hn(\mathord{\cdot}+y_n)$ is log-concave, the above, combined with  Proposition 2(c) of \cite{theory}, yields that
 $\|\hn(\cdot+y_n)-\q\|_{\infty}\as 0$ which completes the verification of part A of Condition~\ref{condition: on hn}.
  As a consequence,
  \[\hln(x+y_n)=\log(\hn(x+y_n))\as\psp(x),\quad \text{ for each }x\in\iint(\dom(\psp)).\]
  %The proof of part A and (B) is quite straightforward when $\hn$ is log-concave, i.e. $\hn$ is either $\widehat{g}_{\bth}$ or $\hn^{geo,sym}$. 
%Combined with Proposition $2$(b) of \cite{theory}, this  weak convergence leads to pointwise convergence of $\hn$ to $\q$ almost everywhere on $\iint(\dom(\psp))$.   
 Since  $\hln$ is concave for $\hn\in\mathcal{LC}$, Theorem~10.8 of \cite{rockafellar} entails that the above 
 pointwise convergence translates to uniform convergence on all compact sets inside $\iint(\dom(\psp))$, which leads to
 $$\sup_{x\in K}|\hln(x+y_n)-\psp(x)|\as 0,$$
 proving part B of Condition~\ref{condition: on hn}.
Since $\hln$ is concave, Part C follows directly  from Part B  by Theorem 25.7 of \cite{rockafellar}. 
Thus we  have established Condition~\ref{condition: on hn} for $\hn=\widehat{g}_{\bth}$, $\hf(\bth\pm\mathord{\cdot})$, $\hn^{geo,sym}$, 
and $\hts(\bth\pm\mathord{\cdot})$.
%  uniform convergence of $\hn(\mathord{\cdot}+y_n)$ over $K$ also follows. 

%  Next, note that  $\hf^{sm}(\mathord{\cdot}+\bth)$ converges weakly to  $\qsm$ with probability one by Lemma \ref{lemma: L1 convergence: one step density estimators}. Therefore following the same arguments as above, one can show that uniform  convergence over $K'$ also holds for  $\hts$ and its logarithm.  Therefore we have to show part A and (B) for the remaining  two mixture densities, $(\hf)^{sm}$ and $\htsm$.

%  From the definition of $\hn^{sym}$ in \eqref{25eq3} we obtain that 
% \begin{align*}
% \MoveEqLeft 2\sup_{x\in K}|\hn^{sym}(x+y_n)-\q(x)|\\
% \leq  &\  \sup_{x\in K}|\hf(\bth+x+y_n)-\q(x)| +\sup_{x\in K}|\hf(\bth-x+y_n)-\q(x)|.
% \end{align*}
% Since we have already verified part A of Condition~\ref{condition: on hn} for $\hf(\bth\pm\mathord{\cdot})$, it is not hard to see that both terms on the right hand side of the above display converge to zero almost surely, which establishes part A for $\hn^{sym}$ as well. 
Now we verify Condition~\ref{condition: on hn} for $\hnss$.
Part A of Condition~\ref{condition: on hn} can be verified  noting  \eqref{representation of htsm}  implies
  \begin{align*}
\MoveEqLeft 2\sup_{x\in \RR}|\htsm(x+y_n)-\qsm(x)|\\
\leq  &\  \sup_{x\in \RR}|\hf^{sm}(\bth+x+y_n)-\qsm(x)| +\sup_{x\in \RR}|\hf^{sm}(\bth-x+y_n)-\qsm(x)|,
\end{align*}
which converges to zero almost surely because, as we have already shown, the log-concave density $\hts$ satisfies Condition~\ref{condition: on hn}(A).

  To prove part B,
%   recalling $\hlnsm=\log(\htsm)$, we write
%  \begin{align*}
% \sup\limits_{x\in K}(\hlnsm(x+y_n)-\psp(x))= & \sup\limits_{x\in K}\log \lb\dfrac{\htsm(x+y_n)-\q(x)}{\q(x)}+1\rb\\
% \leq &\ \sup\limits_{x\in K}\dfrac{\htsm(x+y_n)-\q(x)}{\q(x)},
% \end{align*}
% because $\log(1+x)\leq x$ for any $x>-1$.
% Similarly we can show that
% \[\sup\limits_{x\in K}(\psp(x)-\hlnsm(x+y_n))\leq \sup\limits_{x\in K}\dfrac{\q(x)-\htsm(x+y_n)}{\htsm(x+y_n)},\]
% leading to \todo{This algebra can be shortened}
we observe that
 \begin{align*}
 \sup_{x\in K}\bl\hlnsm(x+y_n)-\psp(x)\bl\leq \dfrac{\sup_{x\in K}|\htsm(x+y_n)-\q(x)|}{\min\lb\inf\limits_{x\in K}\htsm(x+y_n),\inf\limits_{x\in K}\q(x)\rb},
 \end{align*}
whose numerator converges to zero almost surely by part A of Condition~\ref{condition: on hn}.  Thus, to verify part B of Condition~\ref{condition: on hn} for $\htsm$, we only need to show that the denominator of the term on the right hand side of last display is bounded away from zero. To this end, notice that 
\[\inf_{x\in K}\htsm(x+y_n)\geq \inf_{x\in K}\min(\hts(\bth+x+y_n),\hts(\bth-x-y_n))\stackrel{(a)}{\as} \inf_{x\in K}g_0(x),\]
where (a) follows because we just showed that  $\hts(\bth\pm\mathord{\cdot})$ satisfy  Condition~\ref{condition: on hn}. Now $ \inf_{x\in K}g_0(x)>0$ because  $K$ is a  subset of  $\iint(\dom(\psp))$. Thus we have verified  part B of Condition~\ref{condition: on hn}  for $\htsm$.

Next note that  $\hts$ is a smooth function, and it is also positive on  $\RR$. Therefore $\hlnsm$ and  $\hlns$ are differentiable on $\RR$. Therefore, for any $x\in\RR$,
 \begin{equation}\label{definition of tpz by tpc}
(\hln^{sym,sm})'(x)=\varrho_n(x)\lb\hlnsp(\bth+x)\rb-(1-\varrho_n(x))\lb\hlnsp(\bth-x)\rb,
\end{equation}
where  $\varrho_n(x)=\hts(\bth+x)/2\hnss(x)<1$.
Thus 
\begin{align*}
\MoveEqLeft|\hlnspm(x+y_n)-\psp'(x)| \leq  \varrho_n(x+y_n)|\hlnsp(\bth+x+y_n)-\psp'(x)|\nonumber\\
&\ +\ (1-\varrho_n(x+y_n))|\hlnsp(\bth-x-y_n)-\psp'(-x)|.
\end{align*}
 Since $ \varrho_n$ is uniformly bounded by one,  Condition~\ref{condition: on hn}(C)  applied on the concave function $\hlns(\bth\pm\mathord{\cdot})$ completes the verification of part C for $\hlnsm$.
\end{proof}

\subsection{Auxiliary lemmas for the proof of proposition~\ref{prop: partial and geo suffices}}
\label{sec: adlemma: os: conditions}
% \subsubsection{Proof of Lemma~\ref{lemma: L1 convergence: one step density estimators}}
\begin{proof}[Proof of Lemma~\ref{lemma: L1 convergence: one step density estimators}]
First we show that  it suffices to prove the current lemma when $\bth\as \th_0$.  Since $\bth$ is  consistent, Fact~\ref{fact: convergence in probability to convergence almost surely} implies
given any subsequence of $\{\bth\}_{n\geq 1}$, there exists a further subsequence $\{\overline\theta_{n_k}\}_{k\geq 1}$ such that $\overline\theta_{n_k}\as \th_0$ as $k\to\infty$.  If we can show Therefore, along this subsequence $\{n_k\}_{k\geq 1}$, the $L_1$ distance between $\hn$ and $g_0$  approaches zero almost surely.  In that case,  Fact~\ref{fact: Shorack} implies  that $\|\hn-g_0\|_1$ converges in probability to zero. Therefore, in what follows, we assume that $\bth\as \th_0$. 

We begin with the case of $\hts$. 
Theorem~1 of  \cite{smoothed} implies that when $f_0$ has finite second central moment, we have
\begin{equation}\label{convergence: L1: hts}
\edint|\hts(x)-\psm(x)|dx\as 0.
\end{equation}
That $f_0$ has second central moment is immediate by Fact~\ref{fact: Lemma 1 of theory paper}.
Note that
\[\|\hts(\bth+\cdot)-g_0\|_1\leq \|\hts-f_0\|_1+\|f_0(\bth+\cdot)-g_0\|_1,\]
whose first term converges to zero almost surely by \eqref{convergence: L1: hts}, and the second term 
\[\|f_0(\bth+\cdot)-g_0\|_1=\|g_0(-\td+\cdot)-g_0\|_1\stackrel{(a)}{\leq} \sqrt 2 H(g_0(-\td+\cdot),g_0)\stackrel{(b)}{\lesssim}\td\as 0\]
where (a) and (b) follow from Fact~\ref{fact: dTV and hellinger} and Fact~\ref{fact: helli fknot}, respectively. Thus we have established that $\|\hts(\bth+\cdot)-g_0\|_1\as 0$. Since $g_0$ is symmetric about zero,   $\|\hts(\bth-\cdot)-g_0\|_1\as 0$ follows.
Because $\|\hf-f_0\|_1\as 0$ by Theorem~$4$ of \cite{theory}, the proof of 
 $\|\hf(\bth\pm\mathord{\cdot})-g_0\|_1\as 0$ follows in the same way.
 
The  $L_1$ consistency of $\htsm$  also follows noting  \eqref{representation of htsm} implies  
   \begin{align}\label{inlemma: L1 convergence: smoothed}
2 \|\htsm-\qsm\|_1\leq \|\hts(\bth+\cdot)-\qsm\|_1+
\|\hts(\bth-\cdot)-\qsm\|_1\as 0.
\end{align}

Next, we consider the  geometric mean estimator $\hg^{geo,sym}$. 
 We have already established
  \begin{equation}\label{L1 convergence: MLE: centered}
   \edint|\hf(\bth\pm x)-\q( x)|dx\as 0,
  \end{equation}
     which entails that  the distribution functions of $\hf(\bth\pm\mathord{\cdot})$ converge weakly to  $G_0$. The above, combined with 
     Proposition $2$(b) of \cite{theory} shows that \eqref{L1 convergence: MLE: centered} leads to almost sure convergence of
 $\hf(\bth\pm x)$ to  $\q( x)$ almost everywhere on $\RR$ with respect to the Lebesgue measure. As a consequence, it follows that
 \[\hg^{geo,sym}(x)C^{geo}_{n}\as\sqrt{\q(x)\q(-x)}=\q(x)\quad a.e.\ \ x.\]
Recall from \eqref{est 5} that
 \[C^{geo}_{n}=\edint\sqrt{\hf(\bth+x)\hf(\bth-x)}dx.\]
   From Scheff\'{e}'s Lemma it follows that $C^{geo}_{n}\as \int d\tilde{G}_0=1$.
 We have thus established that $\hg^{geo,sym}$ converges almost everywhere to $\q$ almost surely. Therefore, $\Gs$ converges weakly to $G_0$ almost surely. The desired strong $L_1$ consistency then follows from Proposition 2(c) of \cite{theory}.

To establish the $L_1$ consistency of the partial MLE estimator $\widehat{g}_{\bth}$, we appeal to  the  projection  theory developed in  \cite{xuhigh}. According to this theory, $\widehat{G}_{\bth}$ can be interpreted  as the  the projection (w.r.t. Kullback-Leibler divergence)  of \Fn,  the empirical distribution function of the $X_i-\bth$'s, onto the space of the distribution functions with density in $\mathcal{SLC}_0$. This projection operator has some continuity properties. In particular, if we can show that
\begin{equation}\label{convergence of empirical: wasserstein metric : partial mle estimator}
 d_W(\Fn, G_0)\as 0,
 \end{equation} 
the desired $L_1$ consistency $\|\widehat{g}_{\bth}-\q\|_1\as 0$ follows from Proposition $6$ of \cite{xuhigh} provided $G_0$ is non-degenerate and it has first finite moment. The non-degeneracy is trivial and the existence of first moment follows from Fact~\ref{fact: Lemma 1 of theory paper}. Hence, it is enough to prove \eqref{convergence of empirical: wasserstein metric : partial mle estimator} holds,  for which, by Theorem~$6.9$ of \cite{villani2009}, 
 it suffices to show 
 \begin{equation}\label{condition: dw: first condition}
  \edint|x|d\Fn(x)\as \edint|x|dG_0(x),
  \end{equation} 
  and that \Fn\ converges to $G_0$ weakly with probability one.  
Since $\bth$ is strongly consistent for $\th_0$, and
\[
 \edint|x|d\Fn(x)=\edint |x-\bth|d\Fm(x), 
\]
 for any $d>0$,  an application of Glivenko-Cantelli Theorem  \citep[for example, see Theorem~$2.4.1$ of][]{wc} yields
\begin{align*}
\sup\limits_{\bth\in[\th_0-d,\th+d]}\bl\edint |x-\bth|d(\Fm-F_0)(x)\bl\as 0.
\end{align*}
 On the other hand,  strong consistency of $\bth$ implies  $|x-\bth|\leq|x-\th_0|+1$  with probability one for all sufficiently large $n$,  where the latter is integrable with respect to $f_0$. Therefore, the dominated convergence theorem  leads to
\[\edint |x-\bth|dF_0(x)\as \edint |x-\th_0|dF_{0}(x)=\edint |x|dG_0(x),\]
which proves \eqref{condition: dw: first condition}.

Our next step is to prove the weak convergence of \Fn to $G_0$. To this end, we note that
\[
\Fn(x)=\Fm(x+\bth)=\edint 1_{(-\infty,x+\bth]}(z)d\Fm(z),
\]
which converges almost surely to
\begin{align*}
\edint 1_{(-\infty,x+\th_0]}(z)dF_0(z)=G_0(x)
\end{align*}
 by an application of basic Glivenko-Cantelli Theorem \citep[see Theorem~$2.4.1$ of][]{wc}, and the fact that $F_0(x+\bth)\as F_0(x+\th_0)$ for all $x\in\RR$. This establishes \eqref{convergence of empirical: wasserstein metric : partial mle estimator}, which proves the strong $L_1$ consistency of $\widehat{g}_{\bth}$, thus finishing the proof of the current lemma.
 
 %%%%%%% Geometric mean estimator %%%%%%%%%%%%%%%%%%%%%%%%%%%

 \end{proof}
 \subsubsection{\textbf{Lemmas on Hellinger error of $\widehat g_{\bth}$ and $\tilde{g}_n^{geo, sym}$:}}
 
 \begin{lemma}\label{lemma: hellinger: charles}
 Suppose $f_0=g_0(\mathord{\cdot}-\theta_0)$where $g_0\in\mathcal{SLC}_0$ and $\bth-\theta_0=O_p(n^{-1/2})$. Then
 $H(\widehat g_{\bth},g_0)=O_p(n^{-1/4})$.
 \end{lemma}
 
 \begin{proof}[Proof of Lemma~\ref{lemma: hellinger: charles}]
 From Theorem 4.1 of \cite{dosssymmetric} it follows that $H(\widehat g_{\th_0},g_0)=O_p(n^{-2/5})$. The result will therefore follow by triangle inequality if we can show that 
 $H(\widehat g_{\th_0},\widehat g_{\bth})=O_p(n^{-1/4})$. 
 To that end, for any function $\ph:\RR\mapsto\RR$, and distribution function $G$, we define the functional $\om:(\ph,G)\mapsto\RR$ by
 \begin{equation}\label{definition: new criterion function }
     \om(\ph,G)=\edint \ph(x)dG(x)-\edint e^{\ph(x)}dx. 
     \end{equation}
     Recall that we defined $\widehat \psi_{\th}$ to be $\log \widehat g_{\th}$ for any $\th>0$. 
Denoting $\mathbb F_{n,Y}$ to be the empirical distribution function of random variables $Y_1,\ldots,Y_n$,  we observe that for any $\th>0$, $\widehat\psi_{\th}$ writes as \citep[see (2.4) of ][]{dosssymmetric}  
\[ \widehat \psi_{\th}= \argmax_{\phi\in\mathcal{SC}_0} \Phi(\phi, \mathbb F_{n,X-\th})
     = \argmax_{\phi\in\mathcal{SC}_0} \Phi(\phi, \mathbb F_{n,Z+\th_0-\th}),\]
     where $Z=X-\theta_0$. Let us denote $\dd=\th_0-\th$.
Using Lemma~\ref{Lemma: projection theory for a distribution function F} we obtain that
 \[\argmax_{\phi\in\mathcal{SC}_0} \Phi(\phi, \mathbb F_{n,Z+\dd})=\argmax_{\phi\in\mathcal{C}} \Phi(\phi, \mathbb F^{sym}_{n, Z+\dd}),\quad\text{ where }\]
 \[\mathbb F^{sym}_{n, Z+\dd}(x)=\frac{\mathbb F_{n,Z+\dd}(x)+1-\mathbb F_{n,Z+\dd}(-x)}{2}\stackrel{(a)}{=}\frac{\mathbb F_{n,Z}(x-\dd)+1-\mathbb F_{n,Z}(-x-\dd)}{2}\]
 is the symmetrized version of $\mathbb F_{n, Z+\dd}$. Here (a) follows because $\mathbb F_{n,Z+\dd}(x)$ equals $\mathbb F_{n, Z}(x-\dd)$. 
In particular, the choice  $\th=\th_0$ yields $\dd=0$, which leads to\\
 $\widehat\psi_{\th_0}=\argmax_{\phi\in\mathcal{C}} \Phi(\phi, \mathbb F^{sym}_{n, Z})$, 
 where
 $\mathbb F^{sym}_{n, Z}(x)=(\mathbb F_{n, Z}(x)+1-\mathbb F_{n, Z}(-x))/{2}.$
When $\th=\bth$, on the other hand, $\dd=\td$, which yields
%   However, Lemma~\ref{Lemma: projection theory for a distribution function F} implies that
%  \[\argmax_{\phi\in\mathcal{SC}_0} \Phi(\phi, \mathbb F_{n,Z+\td})=\argmax_{\phi\in\mathcal{C}} \Phi(\phi, \mathbb F^{sym}_{n, Z+\td}),\]
%  where
 \[ \widehat \psi_{\bth}
     = \argmax_{\phi\in\mathcal{C}} \Phi(\phi, \mathbb F^{sym}_{n,Z+\td}),\ \mathbb F^{sym}_{n, Z+\td}(x)=\frac{\mathbb F_{n, Z}(x-\td)+1-\mathbb F_{n, Z}(-x-\td)}{2}.\]
%  is the symmetrized version of $\mathbb F^{sym}_{n, Z+\td}$. Here (a) follows because $\mathbb F_{n,Z+\td}(x)$ equals $\mathbb F_{n, Z}(x-\td)$. Thus we have established that $\widehat \psi_{\bth}=\argmax_{\phi\in\mathcal{C}} \Phi(\phi, \mathbb F^{sym}_{n, Z+\td})$. 
%  Similarly we can show that

If we can show that $\mathbb F^{sym}_{n, Z}(x)$ and $\mathbb F^{sym}_{n, Z+\td}$ are non-degenerate with finite first moment, then  Theorem 2 of \cite{barber2020} would imply that 
\begin{equation}\label{intheorem: Rina barbara statement}
   H(\widehat g_{\bth},\widehat{g}_{\th_0})\leq C \lb\frac{d_W(\mathbb F^{sym}_{n, Z},\mathbb F^{sym}_{n, Z+\td})}{\epsilon_{\mathbb F^{sym}_{n, Z}}}\rb^{1/2} ,
\end{equation}
 where $C>0$ is an absolute constant and for any distribution function $F$, $\e_F$ is defined by
 \[\epsilon_{F}=E_{ F}[|Y-E_{F}[Y]|].\]
 Here $E_F$ is the expectation with respect to $F$. 
  Since $G_0$ is non-degenerate, $\mathbb F_{n, Z}$ is non-degenerate with probability one. Therefore both $\mathbb F^{sym}_{n, Z}(x)$ and $\mathbb F^{sym}_{n, Z+\td}$ are non-degenerate with probabilty one. Also, because $\mathbb F_{n, Z}$ has finite first moment for all $n\geq 1$, both $\mathbb F^{sym}_{n, Z}(x)$ and $\mathbb F^{sym}_{n, Z+\td}$ have finite first moment for all $n\geq 1$. Therefore \eqref{intheorem: Rina barbara statement} holds.
 
 Next, we show that $\epsilon_{\mathbb F^{sym}_{n, Z}}$ is bounded away from zero almost surely. To that end, we first prove the side result that $ d_W(\mathbb F_{n,Z}^{sym},G_0)\as 0$.
 By \eqref{def: Wasserstein distance},
 \begin{align*}
   d_W(\mathbb F_{n,Z}^{sym},G_0)=&\ \edint |\mathbb F_{n,Z}^{sym}(x)-G_0(x)|dx  \\
   = &\ \edint \left | \frac{\mathbb F_{n, Z}(x)+1-\mathbb F_{n, Z}(-x)}{2}-G_0(x)\right |,
 \end{align*}
which, due to the symmetry of $g_0$  about the origin,  equals
\begin{align*}
    \MoveEqLeft \edint \left | \frac{\mathbb F_{n, Z}(x)+1-\mathbb F_{n, Z}(-x)}{2}-\frac{G_0(x)+1-G_0(-x)}{2}\right |\\
    \leq &\ \frac{1}{2}\lb\edint\left | \mathbb F_{n, Z}(x)-G_0(x)\right |dx+\edint\left | \mathbb F_{n, Z}(-x)-G_0(-x)\right |dx\rb\\
   \leq &\ \edint\left | \mathbb F_{n, Z}(x)-G_0(x)\right |dx,
\end{align*}
  which equals $d_W(\mathbb F_{n, Z}, G_0)$. The latter converges to zero almost surely by Varadarajan's Theorem \citep[][Theorem 11.4.1]{dudley} and the strong law of large numbers. Therefore,  $ d_W(\mathbb F_{n,Z}^{sym},G_0)\as 0$ follows.

 Proposition 1 of \cite{barber2020} implies that if $F$ and $F'$ are  distribution functions  with finite first moment, then $\epsilon_F>0$, and  $|\epsilon_{F}-\epsilon_{F'}|$ is bounded by $ 2d_W(F,F')$. Now $G_0$ being log-concave, has finite first moment. Therefore we have $\epsilon_{G_0}>0$. Also since $G_0$ and $\mathbb F_{n,Z}^{sym}$ are non-degenerate,  it follows that 
 \[\abs{\epsilon_{\mathbb F_{n,Z}^{sym}}-\epsilon_{G_0}}\leq d_W({\mathbb F_{n,Z}^{sym}},G_0),\]
 which implies 
 $\epsilon_{\mathbb F_{n,Z}^{sym}}>\epsilon_{G_0}-2d_W({\mathbb F_{n,Z}^{sym}},G_0) $. We have just shown $d_W(\mathbb F_{n,Z}^{sym},G_0)$ converges to zero almost surely. Therefore $\epsilon_{\mathbb F_{n,Z}^{sym}}>\epsilon_{G_0}/2$ for sufficiently large $n$ almost surely.
 
  If we can show that $d_W(\mathbb F_{n,Z}^{sym},\mathbb F_{n,Z+\td}^{sym})=O_p(\td)$, the proof of lemma~\ref{lemma: hellinger: charles} follows from \eqref{intheorem: Rina barbara statement} because $\td=O_p(n^{-1/2})$.  To that end, we use an alternative representation of $d_W$ which is due to the Kantorovich-Rubinstein duality theorem \citep[cf. Theorem 2.5]{bobkovbig}. For distribution functions $F_1$ and $F_2$ with finite first moment, it holds that
 \[
    d_W(F_1,F_2)=\sup_{h\in\text{Lip}_1}\abs{\edint h(x)d(F_1-F_2)} 
\]
 where $\text{Lip}_1$ is the set of all real-valued functions $h: \RR\mapsto\RR$ with Lipschitz constant one.
Therefore,
\begin{align*}
   d_W(\mathbb F_{n,Z}^{sym},\mathbb F_{n,Z+\td}^{sym})= &\ \sup_{h\in \text{Lip}_1}\frac{1}{2}\bl\edint h(x) d(\mathbb F_{n, Z}(x-\td)+\mathbb F_{n, Z}(-x-\td))\\
   &\ -\edint h(x) d(\mathbb F_{n, Z}(x)+\mathbb{G}_n(-x)\bl\\
   =&\ \sup_{h\in {\text{Lip}}_1}\frac{1}{2}\bl\edint \slb h(x+\td)-h(x)\srb  d(\mathbb F_{n, Z}(x)+\mathbb{G}_n(-x))\bl\\
   \leq &\ \sup_{h\in {\text{Lip}}_1}\frac{1}{2}\edint | h(x+\td)-h(x)| d(\mathbb F_{n, Z}(x)+\mathbb F_{n, Z}(-x))\\
   \stackrel{(a)}{\leq} &\ |\td| \edint d(\mathbb F_{n, Z}(x)+\mathbb F_{n, Z}(-x))dx/2,
\end{align*}
which equals $|\td|$. Here (a) uses the fact that $h$ is Lipschitz with Lipschitz constant one. Therefore, the proof follows.
 \end{proof}
 
 \begin{lemma}\label{lemma: hellinger: geometric}
 Suppose $f_0\in\mP_0$ and $\bth-\theta_0=O_p(n^{-1/2})$. Then the geometric mean estimator $\hn(x)=\sqrt{\hf(\bth+x)\hf(\bth-x)}/C_n^{geo}$ satisfies $H(\hn, g_0)=O_p(n^{-2/5})$.
 \end{lemma}
 
 \begin{proof}[Proof of Lemma~\ref{lemma: hellinger: geometric}]
We first decompose $H(\hn,g_0)^2$ as follows:
 \begin{align}\label{inlemma: helli: geo: decomposition 1}
     H(\hn,g_0)^2\leq &\ \edint \left(\left(\sqrt{\hf(\bth+x)\hf(\bth-x)}/C_n^{geo}\right)^{1/2}-\sqrt{g_0(x)}\right)^2dx\nn\\
     \leq &\ 2(C_n^{geo} )^{-1}\underbrace{\edint \left(\left(\hf(\bth+x)\hf(\bth-x)\right)^{1/4}-\sqrt{g_0(x)}\right)^2dx}_{T_1}\nn\\
     &\ +2(C_n^{geo})^{-1}\underbrace{(\sqrt{C_n^{geo}}-1)^2}_{T_2}.
 \end{align}
 We focus on $T_1$ first. Note that
 \begin{align}\label{inlemma: helli: geo: T1 first decomp}
   T_1=&\ \edint \left(\left(\hf(\bth+x)\hf(\bth-x)\right)^{1/4}-\sqrt{g_0(x)}\right)^2dx\nn\\
  =&\ \edint \left(\left(\frac{\hf(\bth+x)}{\hf(\bth-x)}\right)^{1/4}\sqrt{\hf(\bth-x)}-\sqrt{g_0(x)}\right)^2\\
  \lesssim &\ \underbrace{H(\hf,f_0(\mathord{\cdot}+\td))^2}_{T_{11}}+\underbrace{\edint\lb \left(\frac{\hf(\bth+x)}{\hf(\bth-x)}\right)^{1/4}-1\rb^2\hf(\bth-x)dx}_{T_{12}}.\nn
  \end{align}
  $T_{11}$ is bounded by $2H(\hf,f_0)^2+2H(f_0,f_0(\mathord{\cdot}+\td))^2$. Thus $T_{11}=O_p(n^{-4/5})$ follows noting (a) $H(\hf, f_0)=O_p(n^{-2/5})$ by Theorem 3.2 of \cite{dossglobal},  and (b) $H(f_0,f_0(\mathord{\cdot}+\td))=O_p(\td)$ by Fact~\ref{fact: helli fknot}.
 Since   $(x-1)^2\leq (x^2-1)^2$ for  $x>0$, the term $T_{12}$ can be bounded by 
  \begin{equation*}
      \edint\lb \sqrt{\hf(\bth+x)}-\sqrt{\hf(\bth-x)}\rb^2dx=  \edint\lb \sqrt{\hf(x)}-\sqrt{\hf(2\bth-x)}\rb^2dx.
  \end{equation*}
Since $f_0(2\theta_0-x)=f_0(x)$, we can further bound $T_{12}$  by a constant multiple of
 \begin{align*}
 \MoveEqLeft 
  \edint\slb \sqrt{\hf(x)}-\sqrt{f_0(x)}\srb^2dx+\edint\slb\sqrt{\hf(2\bth-x)}-\sqrt{f_0(2\bth-x)}\srb^2dx\\
  &\ +\edint \slb \sqrt{f_0(2\bth-x)}-\sqrt{f_0(2\th_0-x)}\srb^2dx\\
  \leq &\ 4H(\hf, f_0)^2+\edint \slb \sqrt{f_0(x-2\td)}-\sqrt{f_0(x)}\srb^2dx\\
  =&\  8H(\hf, f_0)^2+4H(f_0(\mathord{\cdot}+2\td),f_0)^2,
 \end{align*}
 whose first term is $O_p(n^{-4/5})$, and the second term, by Fact~\ref{fact: helli fknot}, is of order $O_p(\td^2)$. Thus similar to $T_{11}$,  $T_{12}$ is $O_p(n^{-4/5})$ as well. Therefore from \eqref{inlemma: helli: geo: T1 first decomp} it follows that  $T_1=O_p(n^{-4/5})$.
 
  Using the fact that $(x-1)^2\leq (x^2-1)^2$ for non-negative $x$, we obtain that
  $T_2\leq (C_n^{geo}-1)^2$, which equals
  \begin{align*}
 \MoveEqLeft    \lb \edint \sqrt{\hf(\bth+x)\hf(\bth-x)}dx-1\rb^2\\
      =&\ \lb\edint \sqrt{\hf(\bth+x)}\slb \sqrt{\hf(\bth-x)}-\sqrt{\hf(\bth+x)}\srb dx\rb^2.
  \end{align*}
 The Cauchy-Schwarz inequality implies that the term on the right hand side of the above display is bounded by 
 \begin{align*}
 &\edint \slb \sqrt{\hf(\bth-x)}-\sqrt{\hf(\bth+x)}\srb^2 dx
     \lesssim \edint \slb \sqrt{\hf(\bth-x)}-\sqrt{f_0(\bth-x)}\srb^2 dx\\
     &+\edint \slb \sqrt{\hf(\bth+x)}-\sqrt{f_0(\bth+x)}\srb^2 dx
     +\edint \slb \sqrt{f_0(\bth-x)}-\sqrt{f_0(\bth+x)}\srb^2 dx.
     \end{align*}
   Clearly, the first two terms equal $4H(\hf, f_0)^2$, which is $O_p(n^{-4/5})$. Since $f_0$ is symmetric about $\th_0$, we can show that $f_0(\bth-x)=f_0(2\td+\bth+x)$, which implies the third term equals $H(f_0, f_0(\mathord{\cdot}+2\td))^2$, which, 
 by Fact~\ref{fact: helli fknot}, is of order $O_p(\td^2)$. Thus we have established  that $T_2$ is $O_p(n^{-4/5})$ as well, which also implies that $C_n^{geo}\to_p 1$. Therefore, by Slutskey's Theorem and \eqref{inlemma: helli: geo: decomposition 1}, the proof follows.
\end{proof}

% \subsubsection{ Lemma~\ref{Lemma: projection theory for a distribution function F}}

  \begin{lemma} \label{Lemma: projection theory for a distribution function F}
 Suppose $F$ is non-degenerate and $F$ has finite first moment.
 Define
\begin{equation*}
F^{sym}_{\th}(x)=2^{-1}\lb F(x)+1-F(2\th-x)\rb.
\end{equation*}
Then it follows that
$\argmax_{\ph\in\mathcal{SC}_{\th}}\om(\ph,F)=\argmax_{\ph\in\mathcal{C}}\om(\ph,F^{sym}_{\th})$
where $\om$ is as defined in \eqref{definition: new criterion function }.
 \end{lemma} 
 
 \begin{proof}[Proof of Lemma~\ref{Lemma: projection theory for a distribution function F}]
 First we will show that
 \begin{align}\label{inlemma: symmetry: 2}
\argmax_{\ph\in\mathcal{SC}_{\th}}\om(\ph,F)=\argmax_{\ph\in\mathcal{SC}_{\th}}\om(\ph,F^{sym}_{\th}).
\end{align}
 Recall the definition of $\Psi$ from \eqref{criterion function: xu samworth}. 
For any distribution function $F$ and $\psi\in\mathcal{SC}_0$, the following holds:
\begin{align*}
\Psi(0,\ps,F)=&\ \dint_{-\infty}^{0}\ps(x)dF(x)+\dint_{0}^{\infty}\ps(x)dF(x)-\edint e^{\ps(x)}dx\\
=&\ -\dint_{0}^{\infty}\ps(-x)dF(-x)+\dint_{0}^{\infty}\ps(x)dF(x)-\edint e^{\ps(x)}dx\\
=&\ \dint_{0}^{\infty}\ps(x)d(F(x)-F(-x))-\edint e^{\ps(x)}dx.
\end{align*}
where the last step uses $\psi(x)=\psi(-x)$. By symmetry, it also follows that
\begin{align*}
\dint_{0}^{\infty}\ps(x)d(F(x)-F(-x))=\dint_{-\infty}^{0}\ps(x)d(F(x)-F(-x)).
\end{align*}
\[\text{Therefore,}\quad\Psi(0,\ps,F)= 2^{-1}\edint\ps(x)d(F(x)-F(-x))-\edint e^{\ps(x)}dx.\]
Equation \ref{criterion function: xu samworth} implies
$\Psi(\theta,\psi,F)=\Psi(\theta, \psi, F(\cdot+\theta))$. Therefore,
\begin{align}\label{inlemma: symmetry}
\Psi(\th,\ps,F)=&\ 2^{-1}\edint\ps(x)d(F(\th+x)-F(\th-x))-\edint e^{\ps(x)}dx\nn\\
\stackrel{(a)}{=}&\  2^{-1}\edint\ps(z-\th)d(F(z)-F(2\th-z))-\edint e^{\ps(z)}dz\nn\\
=&\ \edint\ps(z-\th)dF^{sym}_{\th}(z)-\edint e^{\ps(z)}dz
= \Psi(\th,\ps,F^{sym}_{\th})
\end{align}
where (a) follows substituting $\theta+x$ by $z$. 
Suppose $\psi\in\mathcal{SC}_0$ and $\phi=\psi(\cdot-\th)$. Equation \ref{definition: new criterion function }  implies that for any $\phi\in\mathcal{SC}_\theta$, 
$\om(\ph,F)=\Psi(\th,\ps,F)$, where $\ps=\phi(\cdot+\theta)$.
This, in conjunction with \eqref{inlemma: symmetry}, yields  that
$\om(\ph,F)=\om(\ph,F^{sym}_{\th})$ for any $\ph\in\mathcal{SC}_{\th}$.
Therefore, \eqref{inlemma: symmetry: 2} follows.

Proposition 4(iii) of \cite{xuhigh} entails that  $\argmax_{\ph\in\mathcal{SC}_{\th}}\om(\ph,F)$  exists and  is  unique for a degenerate $F$ with finite first moment. Under similar conditions on $F$, $\argmax_{\ph\in\mathcal{C}}\om(\ph,F^{sym}_{\th})$ also exists and it is unique by Theorem 2.7 of \cite{dumbreg}. Therefore, it suffices to prove $\argmax_{\ph\in\mathcal{C}}\om(\ph,F^{sym}_{\th})$ is in $\mathcal{SC}_{\theta}$ because the latter implies \[\argmax_{\ph\in\mathcal{SC}_{\th}}\om(\ph,F^{sym}_{\th})=\argmax_{\ph\in\mathcal{C}}\om(\ph,F^{sym}_{\th}),\]
which, in conjuction with \eqref{inlemma: symmetry: 2}, completes the proof of the current lemma.

Without loss of generality, we will assume $\theta=0$.
In that case, $F^{sym}_0(x)=(F(x)+1-F(-x))/2$, which implies
\begin{equation}\label{inlemma: symmetry: dF sym}
    dF^{sym}_0(x)=(dF(x)-dF(-x))/2=-dF^{sym}_0(-x).
\end{equation}
   For any concave function $\phi\in\mathcal C$, note that  $\om(\ph,F^{sym}_{0})$ can be written as
 \begin{align*}
\MoveEqLeft \dint_{-\infty}^{0}\ph(x)dF^{sym}_{0}(x)+\dint_{0}^{\infty}\ph(x)dF^{sym}_{0}(x)-\edint e^{\ph(x)}dx\\
 =&\ -\dint_{0}^{\infty}\ph(-x)dF^{sym}_{0}(-x)+\dint_{0}^{\infty}\ph(x)dF^{sym}_{0}(x)-\edint e^{\ph(x)}dx\\
% =&\ -\dint_{\infty}^{0}\ph(-x)dF(x-)+\dint_{0}^{\infty}\ph(x)dF(x)-\edint e^{\ph(x)}dx\\
% =&\  \dint_{0}^{\infty}\ph(-x)dF(x-)+\dint_{0}^{\infty}\ph(x)dF(x)-\edint e^{\ph(x)}dx\\
 \stackrel{(a)}{=}&\  \dint_{0}^{\infty}\ph(-x)dF^{sym}_{0}(x)+\dint_{0}^{\infty}\ph(x)dF^{sym}_{0}(x)-\edint e^{\ph(x)}dx\\
\stackrel{(b)}{=}&\ \edint \dfrac{\ph(-x)+\ph(x)}{2}dF^{sym}_{0}(x)-\edint e^{\ph(x)}dx,
 \end{align*}
 where (a) uses \eqref{inlemma: symmetry: dF sym}, and (b) follows since 
 \[\dint_{0}^{\infty} \dfrac{\ph(-x)+\ph(x)}{2}dF^{sym}_{0}(x)=\dint_{-\infty}^0 \dfrac{\ph(-x)+\ph(x)}{2}dF^{sym}_{0}(x)\]
 by symmetry.
Moreover, since exponential function is convex, we obtain
 \[\Phi(\phi, F_0^{sym})\leq \edint\dfrac{\ph(x)+\ph(-x)}{2}dF^{sym}_{0}(x)-\edint e^{(\ph(x)+\ph(-x))/2}dx,\]
which proves that a $\phi\in\mathcal{SC}_0$ maximizes $\om(\phi,F^{sym}_{0})$ over $\mathcal{C}$, as speculated. 
 Therefore, the proof follows.

% but the distribution function
% $F^{sym}_{\th}$
% is non-degenerate, has finite first moment, and satisfies \eqref{condition: symmetry}. Hence, applying Lemma~\ref{lemma: symmetry of projection operator}, we obtain that
% \[\argmax_{\ph\in\mathcal{SC}_{\th}}\om(\ph,F^{sym}_{\th})=\argmax_{\ph\in\mathcal{C}}\om(\ph,F^{sym}_{\th}),\]
%  thus completing the proof.
 \end{proof}

 \section{Proof of Theorem~\ref{theorem: main: one-step: hnss}}
 \label{app: proof of Theorem 2}
Similar to Theorem~\ref{theorem: main: one-step: full}, we can argue that it suffices to prove Theorem~\ref{theorem: main: one-step: hnss} for the case when $\eta_n$ is $C n^{-2p/5}$.  For the rest of the proof, we will denote $\xi_n=(\Gs)^{-1}(1-\eta_n)$.
 First of all note that $\hts(\bth\pm\mathord{\cdot})$ and $\hnss$ satisfy Condition~\ref{condition: on hn} by Proposition~\ref{prop: the L1 convergence of the density estimators of one-step estimators}. Lemma~\ref{lemma: os: helli: hnss} in Appendix~\ref{adlemma: os: hnss} implies that these densities also satisfy Condition \ref{cond: hellinger rate} with $p=1/5$. Since the proof of Theorem~\ref{theorem: main: one-step: hnss} closely follows the proof of Theorem~\ref{theorem: main: one-step: full}, we will only  highlight the differences. Following the arguments in Theorem~\ref{theorem: main: one-step: full}, we can represent $-(\hth-\bth)$ as the sum of the three terms $T_{1n}$, $T_{2n}$, and $T_{5n}$, where
 \[T_{1n}=\dint_{-\xin}^{\xin}\dfrac{\hlnspm(z)-\psp'(z-\td)}{\hin(\eta_n)}d(\Fm(z+\bth)-F_0(z+\bth)),\]
 \[T_{2n}=\dint_{-\xin}^{\xin}\dfrac{\hlnspm(z)}{\hin(\eta_n)}\lb f_0(z+\bth)-g_0(z)\rb dz
 ,\]
 and $T_{5n}$ is as in \eqref{theorem: main: one-step: first split}. The treatment of $T_{5n}$ in this case will be identical to that in Theorem~\ref{theorem: main: one-step: full}. Hence it suffices to  redo step one and step two of Theorem~\ref{theorem: main: one-step: full} only in the context of $\hnss$.
 
 \subsubsection*{\textbf{Step one: showing $T_{1n}=o_p(1)$:}}
 The main difference in the analysis of $T_{1n}$ between Theorem~\ref{theorem: main: one-step: full} and here stems from the fact that $\hlnspm$ is no longer guaranteed to be monotone since $\hnss$ is not log-concave. So  one needs to be more careful before applying the Donsker theorem to control the $T_{1n}$ term here. By construction, $\hts$ and $\hnss$ are positive on the entire real line, and differentiable everywhere. Using 
  \eqref{definition of tpz by tpc}, we obtain the formula
  \begin{equation*}
\hlnspm(x)=\varrho_n(x)\lb\hlnsp(\bth+x)\rb-(1-\varrho_n(x))\lb\hlnsp(\bth-x)\rb,
\end{equation*}
where $\varrho_n(x)=\hts(\bth+x)/2\hnss(x)$.
 Note that $\hlnsp(\bth\pm\mathord{\cdot})$ is  non-increasing because $\hts$ is log-concave.  On the other hand, because $\hts$ is smooth, and $\hts>0$ on $\RR$, $\varrho_n$ is differentiable with derivative
 \[\varrho'_n(x)=\frac{(\hts)'(\bth-x)\hts(\bth+x)+\hts(\bth-x)(\hts)'(\bth+x)}{(\hts(\bth-x)+\hts(\bth+x))^2},\]
 which is less than $|\hlnsp(\bth-x)|+|\hlnsp(\bth+x)|$ in absolute value. However, Lemma \ref{lemma: hnss: tilde psi prime bound} implies that
\begin{align}\label{intheorem: Theorem 2: bound on phi-sm}
 \sup_{x\in[-\xin,\xin]}\slb|\hlnsp(\bth-x)|+|\hlnsp(\bth+x)|\srb =O_p(n^{p/5}).
\end{align}
 Therefore, on $[-\xin,\xin]$, the derivative of  $\varrho_n$ is uniformly bounded by an $O_p(n^{p/5})$ term. The same bound can be proved for $1-\varrho_n$ as well. 
 Noting $\varrho_n$ is a fraction, we also deduce that $\|\varrho_n\|_\infty$ and $\|1-\varrho_n\|_\infty$ are bounded by one.
 For a convex set $\mathcal X\subset \RR$ and a number $M>0$, define the class of functions $\mathcal D_{n,M}(\mathcal X)$ by
 \[\mathcal{D}_{n,M}(\mathcal X)=\lbs h:\mathcal X\mapsto\RR \ \bl\  h\text{ is differentiable  on }\mathcal X,\ \sup_{x\in\mathcal X}|h(x)|+\sup_{x\in\mathcal X}|h'(x)|\leq M\rbs\]
 As in the proof of Theorem~\ref{theorem: main: one-step: full}, we let $M_n=Cn^{p/5}$ where $C>0$ is a  constant. Our earlier discussion on $\varrho_n$ indicates that
for sufficiently large $C>0$,  $\varrho_n$ and $1-\varrho_n$ restricted to $[-\xin,\xin]$ belongs to $\mathcal D_{n,M_n}([-\xin,\xin])$ 
 with high probability as $n\to\infty$. Note also that \eqref{intheorem: Theorem 2: bound on phi-sm} implies $\hlnsp(\bth\pm\mathord{\cdot})\in\mathcal U_n(M_n)$ with high probability for sufficiently large $C>0$, where $\mathcal U_n(M_n)$ is as defined in \eqref{intheorem: def: U n}. 
%  Hence
% \[\sup_{x\in[-\xin,\xin]}|\hlnspm(x)|\stackrel{(a)}{\leq} \sup_{x\in[-\xin,\xin]}|\hlnsp(\bth+x)|+\sup_{x\in[-\xin,\xin]}|\hlnsp(\bth-x)|\]
%  is $O_p(n^{p/5})$ where (a) follows from \eqref{definition of tpz by tpc}.
 Therefore it is not hard to see that for sufficiently large $C>0$, $\hlnspm 1_{[-\xin,\xin]}\in \mathcal{U}^{sym}_n(M_n, -\xin, \xin)$ with high probability as $n\to\infty$, where for $-\infty\leq r_1< r_2\leq \infty$ and $C>0$, the class  $\mathcal{U}^{sym}_n(M_n, r_1, r_2)$ is defined by
 \begin{align}\label{intheorem: def: Un C sym}
   & \mathcal{U}^{sym}_n (M_n,  r_1, r_2)=\lbs  h:\RR\mapsto[-M_n,M_n]\ \bl\   h(x)=q_1(x)f_1(x)+q_2(x)f_2(x)\text{ for }\nn\\
   &\ x\in[-r,r],\text{ and }0\text{ o.w. where }\   q_1,q_2\in  \mathcal D_{n,M_n}([r_1,r_2]),\ f_1,f_2\in \mathcal U_n(M_n)\rbs
 \end{align} 
  It must be noted that in case of Theorem~\ref{theorem: main: one-step: full}, we had $\tp\in \mathcal U_n(M_n)$. Thus in Theorem~\ref{theorem: main: one-step: hnss}, $\mathcal U_n(M_n)$ is replaced by  $\mathcal{U}^{sym}_n(M_n,-\xin,\xin)$.

Corollary 2.7.2 of \cite{wc} implies  
\[
   \sup_{Q} \log N_{[\  ]}(\e, \mathcal D_{n,M_n}([r_1,r_2]),L_2(Q))\lesssim \frac{(r_2-r_1)M_n}{\e},
\]
where the supremum is over all probability measure $Q$ on real line. 
On the other hand, \eqref{inlemma: finite entropy increasing} implies   $\sup_Q\log N_{[\  ]}(\e, \mathcal U_n(M_n), L_2(Q))\lesssim M_n/\e $. Furthermore, \eqref{inlemma: finite entropy indicator functions} entails that the bracketing entropy of the function-class $\mathcal F_{I}$, consisting of indicator functions of the form $1_{[r_1,r_2]}$, is of the order $\e^{-1}$. Therefore we can  show that
\begin{equation}\label{intheorem: bracketing entropy: U sym n}
   \sup_{Q} \log N_{[\  ]}(\e, \mathcal{U}^{sym}_n(M_n,r_1,r_2),L_2(Q)) \lesssim\frac{(r_2-r_1 )M_n}{\e}.
\end{equation}

 Next, we replace the class $\mathcal{H}_n(C)$ in the proof of Theorem~\ref{theorem: main: one-step: full}  by the class
 \begin{align*}
 \mathcal{H}^{sym}_n(C)=\bigg\{h:\RR\mapsto  \RR\ \bl & \ h(x)=(u(x)-{\phi}_0'(x))1_{[r_1,r_2]}(x),\ u\in\mathcal{U}^{sym}_n(M_n, r_1, r_2),\\
 &\ \|h\|_{P_0,2}\leq Cn^{-2p/5}(\log n)^{3},\quad \|h\|_\infty\leq M_n,\\
 &\
 [r_1,r_2]\subset [\th_0-C\log n,\th_0+C\log n]\cap\iint(\dom(\phi_0)) \bigg \},
 \end{align*}
 where we substituted the class $\mathcal U_n(C)$ in $\mathcal H_n(C)$  by the class $\mathcal U_n^{sym}(C, r_1, r_2)$. Although the dependence of $M_n$ on $C$ is suppressed by its notation, the former is a function of $C$ and $n$. This  validates that the set $\mathcal{H}^{sym}_n(C)$ depends only on $C$  and $n$, as indicated by the notation.
 Note  that $[-\xin,\xin]\subset\iint(\dom(\psi_0))$ by Lemma~\ref{lemma: xi: xi is in dop psi knot},  and $\xin$ is $O_p(\log n)$ by Lemma~\ref{lemma: bound: xi n}.
 Therefore proceeding  as in Theorem~\ref{theorem: main: one-step: full}, but replacing Lemma~\ref{Lemma: L2 norm of hn} by Lemma~\ref{Lemma: L2 norm of hnss}, we can also show that the function
  \begin{equation}\label{inlemma: def: main: hnss}
     h_n(x)=(\hlnspm(x-\bth)-\phi_0'(x))1_{[\bth-\xi_n,\bth+\xi_n]}(x),\quad x\in\RR,
 \end{equation}
 is a member of $\mathcal H^{sym}_n(C)$ with high probability for sufficiently large $n$.
 Using \eqref{intheorem: bracketing entropy: U sym n} in conjuction with \eqref{inlemma: finite entropy indicator functions} we can show that
 \begin{equation}
   \sup_{Q} \log N_{[\  ]}(\e, \mathcal{H}^{sym}_n(C),L_2(Q)) \lesssim\frac{C(\log n )M_n}{\e}.
\end{equation}
Since the bracketing entropy of $\mathcal{H}^{sym}_n(C)$ differs from that of $\mathcal H_n(C)$ only by a poly-log term, so does the entropy integral.  Also, noting $\hnss$  yields a consistent $\hin(\eta_n)$ (see Lemma~\ref{lemma: hnss: FI}) analogous to the log-concave $\hn$'s,  rest of the proof of $T_1=o_p(1)$ follows in a similar fashion as that of  Theorem~\ref{theorem: main: one-step: full}.

\subsubsection*{\textbf{Step two: showing $T_{2n}\to_p -1$:}}
Recall the  function  $b_n$ defined in \eqref{intheorem:t2n:representation}. Because $T_{2n}=-\int_{\RR}b_n(t)dt$, it suffices to show that $\mathbb Y_n=\int_{\RR}b_n(t)dt\as 1$. The proof is not much different from the proof of Lemma~\ref{lemma: main: key lemma for T2}. We will only point out where the current proof differs  from the proof of Lemma~\ref{lemma: main: key lemma for T2}.
Suppose $\mathcal A_n$ and $\mathcal A_n'$ are as defined in the proof of Lemma~\ref{lemma: main: key lemma for T2}.
Let us also introduce the integrals
\[\mathcal I^+_{1n}=\dint_{\mathcal A_n}\tilde \phi^{sm}_{n}(\bth+t)^2\widehat h^{sm}_{n}(\bth+t)dt,\quad \mathcal I^+_{2n}=\dint\limits_{\mathcal A_n+\td}\tilde \phi^{sm}_{n}(\bth+t)^2\widehat h^{sm}_{n}(\bth+t)dt,\]
 \[\mathcal I^-_{1n}=\dint_{\mathcal A_n}\tilde \phi^{sm}_{n}(\bth-t)^2\widehat h^{sm}_{n}(\bth-t)dt,\quad\ \mathcal I^-_{2n}=\dint\limits_{\mathcal A_n+\td}\tilde \phi^{sm}_{n}(\bth-t)^2\widehat h^{sm}_{n}(\bth-t)dt.\]
 The above integrals replace the integrals $\mathcal I_{1n}$ and $\mathcal I_{2n}$ in the proof of Lemma~\ref{lemma: main: key lemma for T2}.
 We also define
 \[\mathcal{J}^{+}_n=\frac{(\log n)^2H(\hts(\bth+\mathord{\cdot}), g_0)^2}{\inf_{x\in\mathcal A_n'}\hts(\bth+x)},\quad \mathcal{J}^{-}_n=\frac{(\log n)^2H(\hts(\bth-\mathord{\cdot}), g_0)^2}{\inf_{x\in\mathcal A_n'}\hts(\bth-x)}\]
 Similar to  Lemma~\ref{lemma: main: key lemma for T2}, it can be shown that it suffices to show that every subsequence has a further subsequence $n_k$, along which, $\mathbb Y_{n_k}\as 1$. We claim that given any sequence, there exists a subsequence $n_k$ such that the set $\mathcal M^{sym}$ has probability one, where we define $\mathcal M^{sym}$ to be the set on which the following hold:\\
 (a) $\overline{\theta}_{n_k}\to_k\th_0$, (b) $\widehat{\mathcal{I}}_{n_k}(\eta_{n_k})\to_k\I$, (c) $\xi_{n_k}\to_k G_0^{-1}(1)$, (d) $\omega_{n_k}\to \omega_0$, (e)$\mathcal J_{n_k}^+$, $\mathcal J_{n_k}^-\to_k 0$,
 (f) $\mathcal{I}^+_{in_k}$, $\mathcal{I}^-_{in_k}\to_k\I$ for $i=1,2$, (g)$\|\widehat h_{n_k}^{sm}(\overline{\theta}_{n_k}\pm \mathord{\cdot})-g_0\|_\infty\to 0$,  (h) $\|\tilde{g}_{n_k}^{sym,sm}-g_0\|_\infty\to 0$, (i) $\mathcal A_{n_k}'\subset \iint(\dom(\psi_0))$ for all sufficiently large $k$.\\
 Note that $\mathcal M^{sym}$ is similar to the good set $\mathcal M$ in the proof of Lemma~\ref{lemma: main: key lemma for T2}.
 The claim that there exists a sequence $n_k$ so that $P(\mathcal M^{sym})=1$  can be verified using Fact~\ref{fact: convergence in probability to convergence almost surely} in the same  way  we verified  a similar claim for $\mathcal M$. The only difference is that here we require Lemma~\ref{lemma: hnss: FI} for (b), Lemma~\ref{lemma: hnss: tilde psi prime bound} instead of Lemma~\ref{lemma: hn: lower bound  on hn} for (e), and Lemma~\ref{lemma: hnss: FI} instead of Lemma~\ref{lemma: consistency of FI: 2}  for (f). As in Lemma~\ref{lemma: main: key lemma for T2}, we will show that  $\mathbb Y_{n_k}\to _k 1$ on $\mathcal M^{sym}$. For the sake of simplicity, we drop $k$ from the subscripts. 
 
 The pointwise converges of $b_n$ can be proved along the lines of \eqref{intheorem: T2: pointwise in dct}. However, Lemma~\ref{Prop: the L1 convergence of the density estimators of one-step estimators: model: strong} can not be directly applied this time because $\hnss$ is not log-concave. On the other hand, Lemma~\ref{Prop: the L1 convergence of the density estimators of one-step estimators: model: strong} does apply to $\hts(\bth\pm\mathord{\cdot})$, because the latter is  log-concave. 
 Exploiting the connection between $\hlnspm$ and $\hlnsp$ as given by  \eqref{definition of tpz by tpc}, and arguing as in the proof of Proposition~\ref{prop: the L1 convergence of the density estimators of one-step estimators}, we can show that the assertions of Lemma~\ref{Prop: the L1 convergence of the density estimators of one-step estimators: model: strong} still hold for $\hnss$  on $\mathcal M^{sym}$. Thus  \eqref{intheorem: T2: pointwise in dct} holds for $b_n$ in case  of $\hnss$.
 
 However, we can not bound $b_n$ using \eqref{inlemma: T2: DCT} because $\hlnspm$ is not monotone. However,
 using \eqref{definition of tpz by tpc}, we can still bound 
 \begin{align*}
   \bl\dint_{t}^{t+\td}\hlnspm(z)dz\bl\leq &\  \dint_{t}^{t+\td}\slb |\hlnsp(\bth+ z)|+|\hlnsp(\bth- z)|\srb dz  \\
   \leq &\ \td\slb \max\{|\hlnsp(\bth+ t+\td)|, |\hlnsp(\bth+ t)|\}\\
   &\ + \max\{|\hlnsp(\bth- t-\td)|, |\hlnsp(\bth- t)|\} \srb.
 \end{align*}
Using the above, it can be shown that $|b_n(t)|\leq |b_n^{+}(t)|+|b_n^{-}(t)|$, where
\begin{align*}
    b_n^{+}(t)=&\  1_{ A_n}(t)|\psi_0'(t)|g_0(t)\slb |\hlnsp(\bth+ t)| + |\hlnsp(\bth+ t+\td)|\srb/\hin(\eta_n),\\
     b_n^{-}(t)=&\  1_{ A_n}(t)|\psi_0'(t)|g_0(t)\slb |\hlnsp(\bth- t)| + |\hlnsp(\bth- t-\td)|\srb/\hin(\eta_n).
\end{align*}
The proof will be complete by Pratt's Lemma (Fact~\ref{fact: pratt's lemma}) if we can show that there exists integrable functions $c_n^+$, $c_n^-$, $c^+$ and $c^-$ so that $|b_{n}^{+}|\leq c_n^+$, $|b_n^{-}|\leq c_n^{-}$, $\int_{\RR}c_n^+(t)dt\to_n \int_{\RR}c^+(t)dt$, $\int_{\RR}c_n^-(t)dt\to_n \int_{\RR}c^-(t)dt$,  and $c_n^+\to_n c^+$ and $c_n^-\to_n c^-$ almost everywhere Lebesgue on $\mathcal{M}^{sym}$. The functions $c_n^+$ and $c_n^-$ can be constructed in the same way we constructed $c_n$ for bounding $b_n$ in the proof of Lemma~\ref{lemma: main: key lemma for T2}. Since the proof follows in a similar manner by replacing $\mathcal M$ by  $\mathcal M^{sym}$, and  $\hn$ by $\hts(\bth\pm\mathord{\cdot})$, it is skipped.  \hfill $\Box$
%  provided $\|h_n\|_{P_0,2}=O_p(n^{2p/5})(\log n)^{3/2}$. 
%  To that end, note that using \eqref{definition of tpz by tpc}, we can show that
%  \begin{align*}
%  \MoveEqLeft    \dint_{\bth-\xin}^{\bth+\xin}\slb\hlnspm(z-\bth)-\psi_0'(z-\th_0)\srb^2f_0(z)dz\\
%  =&\ \dint_{-\xin}^{\xin}\slb\hlnspm(z)-\psi_0'(z-\td)\srb^2f_0(z+\bth)dz\\
%  \leq &\ 2 \dint_{-\xin}^{\xin}\lbs \slb\hlnsp(\bth+z)-\psi_0'(z-\td)\srb^2+\slb\hlnsp(\bth-z)-\psi_0'(z-\td)\srb^2\rbs f_0(z+\bth)dz
%  \end{align*}
%  which is $O_p(n^{p/5})(\log n)^{3/2}$ by Lemma~\ref{Lemma: L2 norm of hn} \textcolor{red}{(Check!)}. 

%  \begin{lemma}\label{lemma: hnss: FI}
%   Suppose $\hn=\htsm$. Then under the set up of Theorem~\ref{theorem: main: one-step: hnss}, we have $\hin(\eta_n)\to_p\I$.
%  \end{lemma}

%  To that end, by Fatou's lemma,
%  \begin{align*}
%   \MoveEqLeft \liminf_n\dint_{-\xin}^{\xin} \hlnspm(x)^2\htsm(x)dx\\
%   \geq &\ \edint\liminf_n\lbs\hlnspm(x)^21_{[-\xin,\xin]}(x)\htsm(x)\rbs dx,  
%  \end{align*}
%  \[\]
%  which is $\I$ by Condition~\ref{condition: on hn} and Lemma~\ref{lemma: xin goes to infity}. On the other hand, \eqref{definition of tpz by tpc} implies that

 \subsection{Auxiliary lemma for Theorem~\ref{theorem: main: one-step: hnss}}
 \label{adlemma: os: hnss}
 In this subsection, $\xi_n$ will generally refer to $\xin(\Gs)\equiv (\Gs)^{-1}(1-\eta_n)$. Although $\hn$ can be either $\hts(\bth\pm\mathord{\cdot})$ or $\hnss$, its definition should be clear from the context.
 
 \subsubsection{\textbf{Lemmas on  Hellinger error of $\hnss$:}}
 \begin{lemma}\label{lemma: os: helli: hnss}
   Under the conditions of Theorem~\ref{theorem: main: one-step: hnss}, $H(\hts(\bth\pm\mathord{\cdot}),g_0)=O_p(n^{-1/5})$ and $H(\hnss, g_0)=O_p(n^{-1/5})$.
 \end{lemma}
 
 \begin{proof}[Proof of Lemma~\ref{lemma: os: helli: hnss}]
 First of all note that $ 2H^2(\hts, f_0)^2$ is bounded by $ \|\hts-f_0\|_1$ which is not larger than
 \begin{align*}
  \MoveEqLeft \widehat{\lambda}_n^{-1}\edint \abs{\edint \slb\hf(x-t)-f_0(x-t)\srb\varphi(t/\smbn)dt} dx\\
  &\ + \widehat{\lambda}_n^{-1}\edint \abs{ \edint \slb f_0(x)-f_0(x-t)\srb \varphi(t/\smbn)dt} dx\\
  \leq &\ \|\hf-f_0\|_1+\widehat{\lambda}_n^{-1}\edint \abs{\edint \varphi(t/\smbn)\dint_{x-t}^x f_0'(z)dzdt }dx,
 \end{align*}
whose first term can be bounded using Fact~\ref{fact: dTV and hellinger}, which yields
  \[ \|\hf-f_0\|_1\leq \sqrt{2} H(\hf, f_0)=O_p(n^{-2/5})\]
  by Theorem 3.2 of \cite{dossglobal}. The  second term
 \begin{align*}
 \MoveEqLeft  \widehat{\lambda}_n^{-1}\edint \abs{\edint \varphi(t/\smbn)\dint_{x-t}^x f_0'(z)dzdt }dx\\
   \leq &\ \widehat{\lambda}_n^{-1}\edint\lb \dint_{-\infty}^0 \dint_{x}^{x-t}  \varphi(t/\smbn)\abs{f_0'(z)}dzdt +\dint_{0}^\infty \dint_{x-t}^{x}  \varphi(t/\smbn)\abs{f_0'(z)}dzdt\rb dx\\
   = &\ \widehat{\lambda}_n^{-1}\edint\lb \dint_{0}^\infty \dint_{x}^{x+t}  \varphi(-t/\smbn)\abs{f_0'(z)}dzdt +\dint_{0}^\infty \dint_{x-t}^{x}  \varphi(t/\smbn)\abs{f_0'(z)}dzdt\rb dx\\
   \leq &\ \widehat{\lambda}_n^{-1}\edint \dint_{-\infty}^\infty \dint_{x-|t|}^{x+|t|}  \varphi(t/\smbn)\abs{f_0'(z)}dzdtdx\\
   =&\ \widehat{\lambda}_n^{-1}\edint \varphi(t/\smbn)\dint_{-\infty}^\infty \dint_{x-|t|}^{x+|t|}  \abs{f_0'(z)}dzdxdt\\
   =&\ 2\edint |t| \widehat{\lambda}_n^{-1}\varphi(t/\smbn)\edint |f_0'(z)|dz dt\\
   =&\ 4\widehat{\lambda}_n E[|\mathbb{Z}|] f_0(\th_0)
 \end{align*}
 where $\mathbb Z\sim N(0,1)$. In the last step we used the fact $\int_{\RR} |f_0'(z)|dz=2f_0(\th_0)$ which follows because $f_0\in\mP_0$.
 Thus
 \begin{equation}\label{inlemma: os: sym: helli: hts}
   H^2(\hts,f_0)=O_p(n^{-2/5})+O_p(\widehat{\lambda}_n).  
 \end{equation}
 Our next step is finding the rate of $\smbn$. To that end, note that because $\int_R x\hf(x)dx$ is the sample average \citep[Corollary 2.3 of][]{2009rufi}, \eqref{definition bn} implies that
  $\widehat{\lambda}^2_n =\int_{\RR} z^2 d(\Fm -\hnn)$ where $\hnn$ is the distribution function of $\hf$.
Therefore, 
\[\widehat{b}^2_n\leq \abs{\edint z^2d(\mathbb{F}_n-F_0)}+\abs{\edint z^2d(\hnn-F_0)},\]
whose first term is $O_p(n^{-1/2})$ by the central limit theorem because $F_0$ has finite second central moment. On the other hand, since the second term equals
\[\abs{\edint z^2(\sqrt{\hf(z)}-\sqrt{f_0(z)})(\sqrt{\hf(z)}+\sqrt{f_0(z)})dz}, \]
The Cauchy-Schwarz inequality indicates that 
its square is bounded by
\[4H(\hf, f_0)^2\edint z^4(\hf(z)+f_0(z))dz.  \]
The fourth moment of $f_0$ is finite by Fact~\ref{fact: Lemma 1 of theory paper}. On the other hand,  Theorem 4 of \cite{theory} implies that there exists $a>0$ so that
\[\edint e^{a|z|}|\hf(z)-f_0(z)|dz\as 0.\]
Therefore it  follows that $\int_{\RR}z^4\hf(z)dz=O_p(1)$. Thus, we conclude $\smbn$ is\\ $O_p(1)H(\hf,f_0)$, which is $O_p(n^{-1/5})$.
 Therefore, \eqref{inlemma: os: sym: helli: hts} yields that $H(\hts, f_0)$ is $O_p(n^{-1/5})$. Since the  Hellinger distance is translation invariant, 
 \[H(\hts(\bth+\mathord{\cdot}),g_0)=H(\hts, g_0(-\bth+\mathord{\cdot}))\leq H(\hts,f_0)+H(g_0(-\bth+\mathord{\cdot}),f_0),\]
 whose first term is $O_p(n^{-1/5})$, and second term is $O_p(|\bth-\th_0|)$ by Fact~\ref{fact: helli fknot}. Because $\bth-\th_0=O_p(n^{-1/2})$, $H(\hts(\bth+\mathord{\cdot}),g_0)=O_p(n^{-1/5})$ follows. Since $g_0$ is symmetric about zero, we can show that  $H(\hts(\bth-\mathord{\cdot}),g_0)=O_p(n^{-1/5})$ as well. Since $2g_0(x)=g_0(x)+g_0(-x)$, and
 \[(\sqrt{a+b}-\sqrt{c+d})^2\lesssim (\sqrt a-\sqrt c)^2+(\sqrt b-\sqrt d)^2\quad \text{ for }a,b,c,d>0,\]
it follows that 
 \[H(\htsm,g_0)\lesssim H(\hf(\bth+ \mathord{\cdot}), g_0)+H(\hf(\bth- \mathord{\cdot}), g_0)=O_p(n^{-1/5}).\]
 \end{proof}
 
 \subsubsection{\textbf{Lemmas on distance between $\Gs$ and $\Gsm$}:}
 \begin{lemma}
   \label{lemma: os: distribution function differnece}
   Under the set up of Theorem~\ref{theorem: main: one-step: hnss}, \begin{align*}
       \text{(A) }&\ \ \  \|\Gs-\Gsm\|_{\infty}=O_p(n^{-p})\\
     \text{(B)  }&\ \ \  \sup_{x\in[-\xin,\xin]}\hts(\bth\pm x)^{-1}=O_p(\eta_n^{-1}),  
   \end{align*}
   where $p=1/5$ and $\xin=(\Gs)^{-1}(1-\eta_n)$. 
 \end{lemma}
 
 \begin{proof}[Proof of Lemma~\ref{lemma: os: distribution function differnece}]
 From the definition of total variation distance, it follows that\\ $\|\Gs-\Gsm\|_{\infty}\leq d_{TV}(\Gs,\Gsm)$, which equals
 \begin{align*}
    \MoveEqLeft 2\|\htsm-\gts\|_1
    \leq 2\sqrt{2} H(\htsm, \gts),
 \end{align*}
 where the last step follows by Fact~\ref{fact: dTV and hellinger}. The proof of part  (A)  then follows noting
 \[H(\htsm,\gts)\leq H(\htsm,g_0)+H(\gts,g_0)=O_p(n^{-p}) \]
 by Lemma~\ref{lemma: os: helli: hnss}.
 
  For the proof of part (B), note that since
  $\hts$ is log-concave, it attains its minimum on any interval at one of the endpoints.
  Therefore
  \[\inf_{x\in[-\xin,\xin]}\hts(\bth\pm x)= \min(\hts(\bth+\xin) ,\hts(\bth-\xin)).\]
   Using  Fact~\ref{fact: f grtr than F} in step (a), and part A of the current lemma in step (b), we can show that
   \[\hts(\bth-\xin)\stackrel{(a)}{\geq} \omega_n\Hsm(\bth-\xin)\stackrel{(b)}{\geq} \omega_n(\Gs(-\xin)-O_p(n^{-p})),\]
  which, by definition of $\xin$, equals $\omega_n\eta_n-O_p(n^{-p})$. Since $\omega_n\to_p\omega_0>0$ by Fact~\ref{fact: f grtr than F}, and $\eta_n=Cn^{-2p/5}$, it follows that $\hts(\bth-\xin)^{-1}=O_p(\eta_n^{-1})$. In a similar way, it can be shown that $\hts(\bth+\xin)^{-1}=O_p(\eta_n^{-1})$. Therefore, the proof follows.
 \end{proof}
 
  \subsubsection{\textbf{Lemmas on $\hlnsp$ and $\hlnspm$:}}
 \begin{lemma}\label{lemma: FI lemma: lemma 1: hnss}
   The conditions of
  Lemma~\ref{lemma: Laha_Nilanjana 31} and  Lemma~\ref{FI Lemma: Lemma 1}  hold for \\ $a_n=\xin\equiv\xin(\Gs)=(\Gs)^{-1}(1-\eta_n)$ and $\hn=\hts(\bth\pm\mathord{\cdot})$.
 \end{lemma}
 
 \begin{proof}[Proof of Lemma~\ref{lemma: FI lemma: lemma 1: hnss}]
 Note that since  $\hnss$ satisfies Condition~\ref{cond: hnss condition}, Lemma~\ref{lemma: g:  bound on psi} entails that $a_n=\xin$ satisfies \eqref{inlemma: os: Nilanjana 31}. Moreover, Lemma~\ref{lemma: os: distribution function differnece} (B) indicates that the supremum of  $-\hlns(\bth\pm\mathord{\cdot})$ over  $[-a_n,a_n]$ is  $O_p(\log n)$  for the above choice of $a_n$. Also because $\hts$ satisfies Condition~\ref{condition: on hn}, $\hlns$ is bounded above. Therefore we obtain that the supremum of  $|\hlns(\bth\pm\mathord{\cdot})|$  on $[-a_n,a_n]$ is  $O_p(\log n)$. Thus, we conclude that \eqref{inlemma: statement: os: nilanjana 31} holds for our choice of $a_n$ and $\hn$. As a result, this $(a_n,\hn)$  pair satisfies the conditions of Lemma~\ref{lemma: Laha_Nilanjana 31}.

 For Lemma~\ref{FI Lemma: Lemma 1}, first note that $a_n=O_p(\log n)$ by Lemma~\ref{lemma: bound: xi n}. Noting  $\dom(\hlns)=\supp(\hts)=\RR$, we also obtain that \eqref{inlemma: statement: os: suport inclusion} holds with probability tending to one  because 
 \[P\slb [-\xin-\eta_n/\log n,\xin+\eta_n/\log_n]\subset\iint(\dom(\psi_0))\srb\to 1\] 
 by Lemma~\ref{lemma: An inclusion}.
%  Since we already showed that \eqref{inlemma: os: Nilanjana 31} holds, and \eqref{inlemma: statement: os: FI: an xin}  holds  by Lemma~\ref{lemma: os: distribution function differnece}, 
 Thus the conditions of Lemma~\ref{FI Lemma: Lemma 1} are satisfied if 
 \eqref{inlemma: statement: os: FI: an xin}  holds for $(\xin,\hts(\bth\pm\mathord{\cdot}))$. 
%  Now
%   Lemma~\ref{lemma: xi: tilde xi n} implies
%  \[[-a_n-\rho_n,a_n+\rho_n]\subset [-\tilde a_n,\tilde a_n]\subset\iint(\dom(\hnss)),\]
%  where $\tilde a_n=(\Gs)^{-1}(1-\eta_n/2)$. 
 Now
 by Lemma~\ref{lemma: os: distribution function differnece}, with probability tending to one,
 \[ \Hsm(\bth-a_n)\geq \Gs(-a_n)-o_p(n^{-p})= \Gs((\Gs)^{-1}(\eta_n))-o_p(n^{-p}),\]
 which is $\eta_n-o_p(n^{-p})$. Since $\eta_n$ is $O(n^{-2p/5})$, it follows that
 \[P\slb \Hsm(\bth-a_n)\geq \eta_n/4\srb \to 1.\]
Similarly we can show that 
\[P\slb 1-\Hsm(\bth+a_n)\geq \eta_n/4\srb\to 1,\]
which implies  \eqref{inlemma: statement: os: FI: an xin}
 holds for  $a_n=\xin$  when $\hn=\hts(\bth+\mathord{\cdot})$. The proof for $\hn=\hts(\bth-\mathord{\cdot})$ follows in a similar way, which completes the proof of the current lemma.
 \end{proof}
 
%  \subsubsection{Nilanjana 31 type lemma for hnss}
 \begin{lemma}\label{lemma: Laha_Nilanjana 31: hnss}
 Suppose $\mu_n$ is a density so that $\|\mu_n\|_\infty=O_p(1)$. 
 Let \\
 $\xi_n=(\Gs)^{-1}(1-\eta_n)$. Then $p=1/5$ satisfies
\begin{align*}
    \dint_{-\xi_n}^{\xi_n} (\hlnspm(x)-\psi'_0(x))^2 \mu_n(x)dx=  O_p((\log n)^6n^{-4p/5}).
\end{align*}
 \end{lemma}
 
 \begin{proof}[Proof of lemma~\ref{lemma: Laha_Nilanjana 31: hnss}]
  Using the representation of $\hlnspm$  given by \eqref{definition of tpz by tpc}, we obtain that
  \begin{align*}
  \MoveEqLeft \dint_{-\xi_n}^{\xi_n} \slb \hlnspm(x)-\psi'_0(x)\srb^2 \mu_n(x)dx\\
  =&\ \dint_{-\xi_n}^{\xi_n} \slb \hlnspm(x)-\varrho_n(x)\psi'_0(x)+(1-\varrho_n(x))\psi_0'(-x)\srb^2 \mu_n(x)dx\\
  \leq &\  2\dint_{-\xi_n}^{\xi_n} \varrho_n(x)^2\slb\hlnsp(\bth+x)-\psi'_0(x)\srb^2 \mu_n(x)dx\\
  &\ + 2\dint_{-\xi_n}^{\xi_n}(1-\varrho_n(x))^2 \slb\hlnsp(\bth-x)-\psi'_0(-x)\srb^2 \mu_n(x)dx\\
  \stackrel{(a)}{\leq} &\ 2\dint_{-\xi_n}^{\xi_n} \lbs \slb\hlnsp(\bth+x)-\psi'_0(x)\srb^2 +\slb\hlnsp(\bth-x)-\psi'_0(-x)\srb^2 \rbs \mu_n(x)dx
  \end{align*}
  which is $O_p((\log n)^6n^{-4p/5})$ by  Lemma~\ref{lemma: FI lemma: lemma 1: hnss} and Lemma~\ref{FI Lemma: Lemma 1}. Here
 (a) follows  because $\varrho_n$ is a fraction.
 \end{proof}

% \subsubsection{Convergence of FI}
 
 \begin{lemma}\label{lemma: hnss: FI}
 Suppose $\eta_n=Cn^{-2p/5}$, where $p=1/5$ and $C>0$.
 Let\\ $\xi_n=(\Gs)^{-1}(1-\eta_n)$. Then under the set up of Theorem~\ref{theorem: main: one-step: hnss}, 
 \begin{equation}\label{statement: lemma: hlnsp}
     \dint_{-\xin}^{\xin}\hlnsp(\bth\pm x)^2\hts(\bth\pm x)\to_p\I,
 \end{equation}
 and
 \begin{equation}\label{statement: lemma: hlnspm}
     \dint_{-\xin}^{\xin}\hlnspm( x)^2\hnss(\bth\pm x)\to_p\I.
 \end{equation}
 \end{lemma}
 
 \begin{proof}[Proof of Lemma~\ref{lemma: hnss: FI}]
 It suffices to show that the pairs $(\xin,\hts(\bth\pm\mathord{\cdot}))$ and $(\xin,\hnss)$ satisfy the conditions of Lemma~\ref{lemma: consistency of FI: 2}.
   By Lemma~\ref{lemma: FI lemma: lemma 1: hnss}, $a_n=\xin$ satisfies the conditions of  Lemma~\ref{FI Lemma: Lemma 1}, which entails that (a) $a_n$ is $O_p(\log n)$ and (b) \eqref{inlemma: statement: os: suport inclusion} holds for $a_n$ with probability tending to one, where \eqref{inlemma: statement: os: suport inclusion} implies $[-a_n,a_n]\subset \iint(\dom(\psi_0))$. 
  Next, the condition $a_n=\xin\to_p G_0^{-1}(1)$ holds by Lemma~\ref{lemma: xi: xi goes to 1}. Finally, \eqref{inlemma: statement: FI} holds for $\hts(\bth\pm\mathord{\cdot})$ and $\hnss$ by Lemma \ref{lemma: FI lemma: lemma 1: hnss} and Lemma~\ref{lemma: Laha_Nilanjana 31: hnss}, respectively.
   Therefore the proof follows from 
    Lemma~\ref{lemma: consistency of FI: 2}.
 \end{proof}
 
% \subsubsection{psi prime lemma}
  \begin{lemma}
   \label{lemma: hnss: tilde psi prime bound}
 Let $\xi_n=(\Gs)^{-1}(1-\eta_n)$ where $\eta_n=Cn^{-2p/5}$ for some $C>0$ and  $p=1/5$.
Suppose $y_n$ is a sequence of random variables such that
  $P(|y_n|\leq \eta_n/(2g_0(0)))\to 1$.
  Then under the conditions of Theorem~\ref{theorem: main: one-step: hnss}, we have
  \[ \sup_{x\in[-\xi_n-y_n,\xi_n+y_n]}\lbs|\hlnspm(x)|+|\hlnsp(\bth\pm x)|\rbs =O_p(n^{p/5}).\]
 \end{lemma}
 
 \begin{proof}[Proof of Lemma~\ref{lemma: hnss: tilde psi prime bound}]
  Note that \eqref{definition of tpz by tpc} implies
  \[|\hlnspm(x)|\leq \max\lbs|\hlnsp(\bth+ x)|,|\hlnsp(\bth-x)|\rbs .\]
  Therefore it suffices to  bound $|\hlnsp(\bth\pm )|$ only. Since the proof  of $\hlnsp(\bth+\mathord{\cdot})$ and $\hlnsp(\bth-\mathord{\cdot})$ are similar, we  only show the proof for $\hlnsp(\bth+\mathord{\cdot})$. 
  Denoting $\hn=\hts(\bth+\mathord{\cdot})$, $\hln=\hlnsp(\bth+\mathord{\cdot})$, and $\tilde G_n=\Hsm(\bth+\mathord{\cdot})$, note that
the following holds for any $q\in(0,1/2)$ by Fact~\ref{fact: bobkov big} because $\hn$ is positive on $J(\tilde G_n)$:
  \[\dint_{(\Gs)^{-1}(q/2)}^{(\Gs)^{-1}(q)}\tp(x)^2\hn(x)dx=\dint_{\tilde G_n((\Gs)^{-1}(q/2))}^{\tilde G_n((\Gs)^{-1}(q))}\tp(\tilde G_n^{-1}(z))^2 dz.\]
  Because $(\tp)^2$ is non-increasing on $(-\infty,0]$, the above yields
  \[\tp((\Gs)^{-1}(q))^2\leq \frac{\dint_{(\Gs)^{-1}(q/2)}^{(\Gs)^{-1}(q)}\tp(x)^2\hn(x)dx}{\tilde G_n((\Gs)^{-1}(q))-\tilde G_n((\Gs)^{-1}(q/2))}.\]
  Letting $q=\eta_n/2$, and denoting $\tilde {\xi}_n=(\Gs)^{-1}(1-\eta_n/2)$, we obtain that
  \[\tp(\tilde{\xi}_n)^2\leq \dint_{(\Gs)^{-1}(\eta_n/4)}^{(\Gs)^{-1}(\eta_n/2)}\tp(x)^2\hn(x)dx \Big /\slb \eta_n/4 -2\|\tilde G_n-\Gs\|_{\infty}\srb.\]
  Now Lemma~\ref{lemma: os: distribution function differnece} implies 
  \[\eta_n/4 -2\|\tilde G_n-\Gs\|_{\infty}=\eta_n/4-O_p(n^{-p}),\]
  whose dominating term is $\eta_n/4$. Also  
  \[\dint_{(\Gs)^{-1}(\eta_n/4)}^{(\Gs)^{-1}(\eta_n/2)}\tp(x)^2\hn(x)dx\leq \dint_{(\Gs)^{-1}(\eta_n/4)}^{(\Gs)^{-1}(1-\eta_n/4)}\tp(x)^2\hn(x)dx\]
  which converges in probability to $\I$ by Lemma~\ref{lemma: hnss: FI}.
  Therefore $\tp({\tilde \xi_n})$ is $O_p(\eta_n^{-1/2})$. The rest of the proof follows similar to the proof of  Lemma~\ref{lemma: tilde psi prime bound}.
 \end{proof}
 
%  \subsubsection{FI lemma : 2}
%  \begin{lemma}\label{lemma: hnss: consistency of FI: with delta}
%  Suppose  $(y_n)_{n\geq 1}$, $(y_n')_{n\geq 1}$ are sequences of random variables satisfying $y_n,y_n'=o_p(1)$. Let $\xi_n=(\Gs)^{-1}(1-\eta_n)$. Then under the set up of Lemma~\ref{lemma: consistency of FI: 2}, we have
%  \[\dint_{-\xi_n-y_n}^{\xi_n+y_n'}\hlnsp(\bth\pm z)^2\hts(\bth\pm z)dz\to_p \I.\]
%  \end{lemma}
 
%  \begin{proof}[Proof of Lemma~\ref{lemma: hnss: consistency of FI: with delta}]
% The proof follows from Lemma~\ref{lemma: consistency of FI: with delta} letting $a_n=\xi_n$.
%   \end{proof}
  
% \subsubsection{The $L_2$ bound lemma}
 \begin{lemma}\label{Lemma: L2 norm of hnss}
 Suppose $\hn$ satisfies Condition~\ref{cond: hellinger rate}.
  Then under the set up of Theorem~\ref{theorem: main: one-step: hnss},
\[\dint_{-\xin}^{\xin}(\hlnspm(x)-\psi_0'(z-\td))^2f_0(z+\bth)dz=O_p(n^{-4p/5}(\log n)^{6}).\]
\end{lemma}

\begin{proof}[Proof of Lemma~\ref{Lemma: L2 norm of hnss}]
 The proof   is similar to the proof of Lemma~\ref{Lemma: L2 norm of hn}. The only difference is that  one needs to use  use Lemma~\ref{lemma: Laha_Nilanjana 31: hnss} instead of Lemma~\ref{FI Lemma: Lemma 1} to bound $T_1$
\end{proof}

\subsubsection{\textbf{Lemmas on consistency of Fisher information:}}

\begin{lemma}\label{lemma: hnss: consistency of FI} 
Under the set up of Theorem~\ref{theorem: main: one-step: full}, $\hin(\eta_n)\to_p\I$.
\end{lemma}

\begin{proof}[Proof of Lemma~\ref{lemma: hnss: consistency of FI}]
 The proof follows in the same way as Lemma~\ref{lemma: consistency of FI}  by replacing Lemma~\ref{lemma: consistency of FI: 2} and Lemma~\ref{lemma: tilde psi prime bound} by Lemma~\ref{lemma: hnss: FI} Lemma~\ref{lemma: hnss: tilde psi prime bound}, respectively, and replacing the class of monotone functions $\mathcal U_n(M_n)$ by the class $\mathcal{U}^{sym}_n(M_n,r_1,r_2)$ defined in \eqref{intheorem: def: Un C sym}.
\end{proof}
 %%%%%%%%%%%%%%%%%%%%%%%%%%%%%%%%%%%%%%%%%%%%%%%%%%%%%%%%%%%%%%%%%%%%%%%%%%%%%%%%%%%%%%%%%%%%%%%%%%%%%%%%%%%%%%%%%%%%%%%%

\section{Proof of Lemma~\ref{lemma: mle does not exist degenerate}}
\label{app: mle lemma}

\begin{proof}[Proof of Lemma~\ref{lemma: mle does not exist degenerate}]
 For $k\geq 1$, we denote $A_k$ to be the set $[-1/(2k),1/(2k)]$, and consider
 the sequence of  functions $\{\ps_k\}_{k\geq 1}\in\mathcal{SC}_0$ defined by
 \[\ps_k(x)=\begin{cases}
 \begin{matrix}
 \log k, & \quad x\in A_k\\
 -\infty, &\quad  \text{o.w.}
 \end{matrix}
 \end{cases}\]
   Observe that
 \[\Psi_n(x_0,\ps_k)= \log k-1\to\infty,\quad\text{ as }k\to\infty.\]
 Therefore, $x_0$ indeed is  a candidate for the MLE of $\th_0$. However,  the MLE of $\ps_0$, i.e. $\hpm$, does not exist in this case. To verify, observe that if \hpm\ does exist for some $\hthm\in\RR$, we also have
 \[\hpm(x_0-\hthm)-\edint e^{\hpm(x)}dx= \Psi_n(x_0,\hpm)\geq \lim_{k\to\infty}\Psi_n(x_0,\ps_k)=\infty,\]
 leading to $\hpm(x_0-\hthm)=\infty$,
 which contradicts the fact that $\hpm$ is a proper concave function. 
 Hence, we conclude that the MLE of $(\th_0,\ps_0)$ does not exist when  $\Fm$ is degenerate.
\end{proof}

\section{Proof of Theorem~\ref{theorem: existence}}
\label{app: existence theorem}
To prove Theorem~\ref{theorem: existence}, it will be beneficial to prove a general result first. We begin by stating a condition.
\begin{condition}(Existence of log-concave projection.)
 \label{cond: A}
 $F$ is a non-degenerate distribution function with finite first moment.
 \end{condition}
 Any $F$ satisfying Condition~\ref{cond: A} has a well-defined log-concave projection, i.e. its projection (with respect to the KL divergence) onto the space of all distributions with density in $\mathcal{LC}$ is a unique distribution function\citep[Theorem 2.2,][]{dumbreg}.
 Note that $\Fm$ satisfies Condition~\ref{cond: A} with probability one.
 We will show that for any distribution function $F$ satisfying  Condition~\ref{cond: A},
 \begin{equation}\label{MLE: maximizer of criterion}
(\th^*(F),\ps^*(F))=\argmax_{\th\in\RR,\ps\in\mathcal{SC}_0}\Psi(\th,\ps,F)
\end{equation}
exists where $\Psi$ is the criterion function  defined in \eqref{criterion function: xu samworth}.
  \begin{proposition}\label{Prop: existence of maximizers}
 If $F$ satisfies condition~\ref{cond: A}, then   $\th^*(F)$ and $\ps^*(F)$  exist.
 \end{proposition}  
 
 Observe that   Proposition \ref{Prop: existence of maximizers}
implies the first part of Theorem~\ref{theorem: existence} because
if $F=\Fm$, $(\th^*(F),\ps^*(F))$ corresponds to the MLE $(\hthm,\hgm)$.
The second part of Theorem~\ref{theorem: existence} follows from Lemma \ref{lem: support of projection map}, which is proved in Appendix~\ref{subsec: addlemma: MLE: existence}. 
\begin{lemma}\label{lem: support of projection map}
Suppose $F$ is such that $J(F)=(a,b)$ where $a,b\in\RR$, and $J(F)=\{0<F<1\}$. Then
Under condition~\ref{cond: A}, there exists at least one $\th^*(F)$ so that $\th^*(F)\in[a,b]$. 
\end{lemma}

Thus it remains to prove Proposition~\ref{Prop: existence of maximizers}. To that end, we will need a continuity result on the partially maximized criterion function
\begin{equation}\label{def: L(th,F)}
L(\th;F)=\sup_{\psi\in\mathcal{SC}_0}\Psi(\th,\ps,F).
\end{equation}

\begin{lemma}\label{lemma: continuity: general}
Suppose the distribution function $F$ satisfies condition~\ref{cond: A}. Then the map $\th\mapsto L(\th;F)$
is continuous on $\RR$, where  $L(\th;F)$ is as  defined in \eqref{def: L(th,F)}.
\end{lemma}
The proof of Lemma~\ref{lemma: continuity: general} can be found in Appendix~\ref{subsec: addlemma: MLE: existence}.
Now we are ready to prove Proposition~\ref{Prop: existence of maximizers}.

\begin{proof}[Proof of Proposition~\ref{Prop: existence of maximizers}]
Let us define 
\begin{equation}\label{maximum of criterion function}
L(F)=\sup_{\th\in\RR,\ps\in\mathcal{SC}_0}\Psi(\th,\ps,F).
\end{equation}
Our first step is to show that  $L(F)$  is finite. 
%Proposition $4$ of \cite{xuhigh} and condition~\ref{cond: A} indicate that $g(\th)<\infty$ for each $\th\in\RR$. \
From the definition of $\Psi$ in \eqref{criterion function: xu samworth}, it is not hard to see that
\[L(F)\leq \sup_{\ps\in\mathcal{C}}\lb\edint\ps(x) dF(x)-\edint e^{\ps(x)}dx\rb,\]
where $\mathcal{C}$ denotes the set of all real-valued concave functions. Theorem 2.2 of \cite{dumbreg} entails that under condition~\ref{cond: A}, the term on the right hand side of the above display  is finite. Therefore, $L(F)<\infty$ follows. To show that $L(F)>-\infty$, we note that the map $x\mapsto-|x|\in\mathcal{SC}_{0}$. Therefore, \eqref{criterion function: xu samworth} and \eqref{maximum of criterion function} lead to 
\[L(F)\geq -\edint |x|dF(x)-\edint e^{-|x|}dx>-\infty,\]
which follows from condition~\ref{cond: A}. Hence, we  conclude that $L(F)\in\RR$.

Now we have to show that there exist $\th^*(F)\in\RR$ and $\ps^*(F)\in\mathcal{SLC}_0$ such that
\[\Psi(\th^*(F),\ps^*(F),F)=\sup_{\th\in\RR,\ps\in\mathcal{SLC}_0}\Psi(\th,\ps,F)=\sup_{\th\in\RR}L(\th;F)=L(F).\]
 Now there exists a sequence $\{\th_k\}_{k\geq 1}$ such that $L(\th_k;F)\uparrow L(F)$ as $k\to\infty$. Suppose 
the sequence $\{\th_k\}_{k\geq 1}$ is bounded. Then we can find a  subsequence $\{\th_{k_r}\}_{r\geq 1}$  converging to some $\th'\in\RR$. Since the map $L(\th;F)$ is continuous in $\th$ by Lemma~\ref{lemma: continuity: general},  we  also have 
\[L(\th';F)=\lim_{r\to\infty}L(\th_{k_r};F)=L(F),\]
which implies that $\th'$ is a maximizer of $L(\th;F)$. Now we invoke Proposition 4(iii) of \cite{xuhigh}, which
% Doss and Wellner (2018)
   states that for each $\th\in\RR$,  there exists a unique log-density ${\ps}_{\th}$, which maximizes  $\Psi(\th,\ps,F)$ in $\ps\in\mathcal{SC}_{0}$ provided $F$ satisfies condition~\ref{cond: A}. It is not hard to see that $(\th',\ps_{\th'})$ will be a candidate for $(\th^*(F),\ps^*(F))$. Thus, to complete the proof, it remains to show that  $\{\th_k\}_{k\geq 1}$ is bounded. We will show that
$\th_k\to_k\pm\infty$ leads to  $L(\th_k;F)\to_k-\infty$, which contradicts the fact that $L(\th_k;F)\to_k L(F)\in\RR$, thus completing the proof. 
 
 Consider $\th_k\to_k\pm\infty$. By Proposition 4(iii) of \cite{xuhigh}, for each $\th_k$, there exists a log-density $\ps_{\th_k}\in\mathcal{SC}_0$ such that $L(\th_k;F)=\Psi(\th,\ps_{\th_k},F)$.
Now note that if $e^{\psi}\in\mathcal{SLC}_0$, then $\psi$ satisfies
\[2xe^{\psi(x)}\leq \dint_{-x}^xe^{\psi(z)}dz\leq 1\quad\text{ for any }x\geq 0,\]
which implies $|\psi(x)|\leq -\log|2x|$. Noting $\psi_{\th_k}\in\mathcal{SC}_0$ for each $k\geq 1$, we obtain that
  \begin{align*}
  \Psi(\th,\ps_{\th_k},F)=\edint \ps_{\th_k}(x-\th)dF(x)-1\leq -\edint\log \lb 2|x-\th_k|\rb dF(x)-1.
  \end{align*}
 Now if $\th_k\to_k\pm\infty$, using Fatou's Lemma, we derive that
 \begin{align*}
 \limsup_{k\to\infty} L(\th_k;F)\leq -\edint\liminf_{k\to\infty}\lb\log |x-\th_k|\rb dF(x)-(\log 2+1),
 \end{align*}
 which is $-\infty$.
 This leads to the desired contradiction, which completes the proof. 
\end{proof}
 
\subsection{Auxilliary Lemmas for Theorem~\ref{theorem: existence}}
\label{subsec: addlemma: MLE: existence}

% \subsubsection{Proof of Lemma~\ref{lemma: continuity: general}}
\begin{proof}[Proof of Lemma~\ref{lemma: continuity: general}]
Observe that  \eqref{criterion function: xu samworth} implies $\Psi(\th,\ps,F)$ can also be written as
$\Psi(\th,\ps,F)=\Psi(0,\ps,F(\mathord{\cdot}+\th))$.
 Hence, to prove Lemma~\ref{lemma: continuity: general}, it suffices to show that as  $\th_k\to_k\th\in\RR$,
\begin{equation*}
L(\th_k;F)=\sup_{\psi\in\mathcal{SC}_0}\Psi(0,\ps,F(\mathord{\cdot}+\th_k))\to_k\sup_{\psi\in\mathcal{SC}_0}\Psi(0,\ps,F(\mathord{\cdot}+\th))=L(\th,F).
\end{equation*}
Proposition $6$ of \cite{xuhigh} implies  that under condition~\ref{cond: A}, the convergence in the above display holds if the Wasserstein distance
\begin{equation}\label{inlemma: continuity: convergence of }
d_W(F(\mathord{\cdot}+\th_k),F(\mathord{\cdot}+\th))\to_k 0.
\end{equation}
Now by Theorem $6.9$ of \cite{villani2009} \citep[see also Theorem $7.12$ of][\  ]{villani2003}, \eqref{inlemma: continuity: convergence of } follows  if  (a) $F(\mathord{\cdot}+\th_k)$ converges weakly to $F(\mathord{\cdot}+\th)$ as $k\to\infty$, and 
\begin{flalign*}
(b)\quad\quad \edint|x|dF(x+\th_k)\to_k \edint|x|dF(x+\th).
\end{flalign*}
Now (a) follows noting that for any bounded continuous function $h$,
\[\edint h(x-\th_k)dF(x)\to_k\edint h(x-\th)dF(x)\]
by the dominated convergence theorem since $\th_k\to_k\th$.
For proving (b), first notice that $F$ has finite first moment  by condition~\ref{cond: A}. Therefore, another application of the dominated convergence  yields that as $\th_k\to_k\th$,
\[\edint|x|dF(x+\th_k)=\edint|x-\th_k|dF(x)\to_k \edint|x|dF(x+\th),\]
which proves (b), and thus completes the proof.
\end{proof}
% \subsubsection*{Proof of Lemma \ref{lem: support of projection map}}

\begin{proof}[Proof of Lemma \ref{lem: support of projection map}]
  We will show that if $J(F)=(a,b)$, the functional $\th\mapsto L(\th;F)$ defined in \eqref{def: L(th,F)} is non-decreasing in $\th$ on $(-\infty,a]$, and non-increasing in $\th$ on $[b,\infty)$. Suppose the above claim holds. Then clearly either $L(\theta,F)$ attains its maximum  in $[a,b]$ or  $L(\theta,F)=L(F)$ over an interval with nonempty overlap with $[a,b]$. Here $L(F)$ is as defined in \eqref{maximum of criterion function}.
  In either cases, one can find a $\theta^*(F)\in[a,b]$, which  completes the proof of Lemma \ref{lem: support of projection map}

 To show that $L(\th;F)$  is non-decreasing in $\th$ on $(-\infty,a]$, we first note that for   $\th<\th'\leq a$, and $\ps\in\mathcal{SC}_0$,
  \begin{align*}
 \dint_{a}^{b} \psi(x-\th)dF(x) \leq  \dint_{a}^{b} \psi(x-\th')dF(x),
  \end{align*}
 since $\ps$ is non-increasing on $[0,\infty)$, and 
$0\leq x-\th'<x-\th$
 for $x\geq a$.
Therefore, from  \eqref{def: L(th,F)}, it is not hard to see that $L(\th;F)\leq L(\th';F)$. Similarly we can show that for $\th>\th'\geq b$,
 \begin{align*}
 \dint_{a}^{b} \psi(x-\th)dF(x)\leq \dint_{a}^{b} \psi(x-\th')dF(x),
  \end{align*}
since $\psi$ is non-decreasing on $(-\infty,0]$, and $x-\th<x-\th'\leq 0$ for $x\leq b$. Therefore, $L(\th;F)\leq L(\th';F)$, which completes the proof.
\end{proof}

% Lemma about the concave bound for each symmetric concave function

\section{Proof of Theorem \ref{MLE: Rate results}}
\label{app: rate results}
Before going into the proof, we will introduce some new notations and state some lemmas that will be required later in the proof.
We let $\hFm$ and $\hgF$ denote the  distribution functions corresponding to $\hgm$ and $\hgf$, respectively.  Also, we let $\hpfm$ denote the  log-density $\log\hgf$. Also, we let $\delta_n=\th_0-\hthm$.

Now we state a lemma which basically says that $\log\hgf$ is uniformly bounded above for sufficiently large $n$ with probability one. This lemma is proved in Appendix~\ref{subsec: aux: mle: rate}.
 \begin{lemma}\label{lemma: Pal's paper Lemma}
  Under the hypotheses of Theorem \ref{MLE: Rate results},
 \[P\lb\sup_n\sup_{x\in\RR}\log \hgf(x)<\infty\rb =1.\]
 \end{lemma}

We first show that $\hthm\as\th_0$.
In their proof of Theorem $3.1$, \cite{exist} show that if a sequence of log-concave functions $\{f_n\}_{n\geq 1}$ (which can be stochastic) satisfies
 \begin{equation}\label{mle: Hellinger consistency: likelihood is larger}
  \si\log f_n(X_i)\geq \si\log f_0(X_i) 
 \end{equation}
with probability one, we have $H(f_n,f_0)\as 0$,  provided 
 \begin{equation*}
 P\lb\sup_x\log f_n(x)=o\lb\dfrac{\sqn}{\log n}\rb\rb =1.
 \end{equation*}
 If we take $f_n=\hgf$, we have 
 $\sup_x\log \hgf(x)=\hpfm(\hthm)=\hpm(0)$.
 Lemma~\ref{lemma: Pal's paper Lemma} entails that 
 $P(\limsup_n\hpm(0)<\infty)=1$.
 Also, note that being the MLE of $f_0$,  $\hgf$ automatically satisfies \eqref{mle: Hellinger consistency: likelihood is larger},
 which implies
 \begin{equation}\label{inlemma: hellinger consistency of MLE of fnot}
  H(\hgf,f_0)\as 0.
  \end{equation} 
  Denote by $\widehat F_n$ the distribution of $\hgf$.
  Because $d_{TV}(\widehat F_n, F_0)\leq \sqrt 2H(\hgf,f_0)$ by Fact~\ref{fact: helli fknot}, \eqref{inlemma: hellinger consistency of MLE of fnot} implies $d_{TV}(\widehat F_n, F_0)\as 0$, which indicates  $F_n\to_d F_0$ almost surely. In that case,
  Proposition 2 of \cite{theory} implies that there exists $\alpha>0$ so that
 \[\edint e^{\alpha |x|}\abs{\hgf(x)-f_0(x)}dx\as 0.\]
 Therefore the moments of $\hgf$ converges almost surely to that of $f_0$. Notably, the first moment of $f_0$ is $\theta_0$, and because $\hgm$ is symmetric about zero, we also have
 \[\edint x\hgf(x)dx=\hthm+\edint (x-\hthm)\hgm(x-\hthm)dx=\hthm.\]
 Thus $\hthm\as\theta_0$ follows.
  Since   $\I<\infty$, the density  $f_0$ is absolutely continuous  \citep[ ][Theorem~3]{huber}. Because  $f_0$ is continuous, Proposition 2 of \cite{theory} yields another useful result which will be required later:
  \begin{equation}\label{intheorem: MLE: ks distance}
      \sup_{x\in\RR}\abs{\hgf(x)-f_0(x)}\as 0.
  \end{equation}

Next we  show that $H(\hgm,g_0)\as 0$, which completes the proof of part A. To that end, note that
 \begin{align*}
 2H^2(\hgm,g_0) =&\ \edint\lb\sqrt{\hgm(z-\hthm)}-\sqrt{g_0(z-\hthm)}\rb^2dz\\
 \leq &\ 4H(\hgf,f_0)^2 +2\edint\lb\sqrt{g_0(z-\hthm)}-\sqrt{g_0(z-\th_0\vphantom{\hthm})}\rb^2dz,
 \end{align*}
 where the first term on the right hand side  of the last display approaches zero almost surely by \eqref{inlemma: hellinger consistency of MLE of fnot}. The integrand in the second term  is also bounded above by a constant multiple of $g_0(z-\hthm)+g_0(z-\th_0)$, which converges to $2f_0(z)$, and is integrable. Therefore, using Pratt's lemma (Fact~\ref{fact: pratt's lemma}),  we deduce that the second  term also converge to zero almost surely.  Hence $H^2(\hgm,g_0)\as 0$ follows. 
 
 Now we turn to the proof of part B, where
we first establish  that  $H(\hgf, f_0)=O_p(n^{-4/5})$. To that end, we first introduce the class of functions 
 \begin{align*}
 \mathcal{P}_{M,0}=\bigg\{f\in\mathcal{LC}\ \bl\ &\ \sup\limits_{x\in\RR}f(x)<M,\ \inf_{|x|>1}f(x)>1/M, \ \supp(f)\subset\supp(f_0)\bigg\}.
 \end{align*}
 We will show  that without loss of generality, one can assume that $f_0\in\mP_{M,0}$ for some  $M>0.$ To this end, we translate and rescale the data letting $\tilde{X}_i=\aalpha X_i+\abeta$, where $\alpha>0$ and $\abeta\in\RR$. Observe that the rescaled data has density $\tilde{f}_0(x)=\aalpha^{-1}f_0((x-\abeta)/\aalpha)$. Denote by $\tilde{f}_{0,n}$ the MLE of $f_0$ based on the rescaled data. Note that the  MLE is affine-equivalent, which  entails that $\tilde{f}_{0,n}(x)=\aalpha^{-1}\hgf((x-\abeta)/\aalpha).$ Noting Hellinger distance  is invariant under affine transformations, we observe that $H(\hgf,f_0)=H(\tilde{f}_{0,n},\tilde{f}_0).$  
   Therefore, it suffices to show that $H(\tilde{f}_{0,n},\tilde{f}_0)\as 0$. Note that since $f_0$ is log-concave,  $\iint(\dom(f_0))$ contains an interval. We can choose $\aalpha$ and $\abeta$ in a way such that  $( x-\abeta)/\aalpha$ lie inside that interval for $x=\pm 1$. Then it is possible to find $M>0$ large enough such that
 \[f_0((x-\abeta)/\aalpha)>\aalpha/M,\ \text{ for } x=\pm 1,\quad{ \text{yielding} }\quad \min (\tilde{f}_0(-1),\tilde{f}_0(1))>1/M.\]
The above implies
 $\inf_{x\in[-1,1]}\tilde{f}_0(x)>1/M$,
 since $f_0,$ or equivalently $\tilde{f}_0$ is unimodal. Hence, without loss of generality, we can assume that there exists $M>0$ such that $f_0(x)>1/M$ for $x\in[-1,1].$ 
 We can choose $M$ large enough such that additionally, $\sup\limits_{x\in\RR} f_0(x)<M$.   On the other hand,  \eqref{intheorem: MLE: ks distance} implies 
  \[\limsup_n\sup\limits_{x\in\RR} \hgf(x)<M,\quad\text{and}\quad\lim_n\hgf(\pm 1)>1/M.\]
    Therefore, $f_0\in\mathcal{P}_{M,0}$, and with probability one,  $\hgf\in\mathcal{P}_{M,0}$ as well for all sufficiently large $n$. \cite{dossglobal} obtained the bracketing entropy of the class $\mathcal{P}_{M,0}$. They showed that for any $\e>0$,
 \[\log N_{[\ ]}(\e,\mathcal{P}_{M,0},H)\lesssim \e^{-1/2}.\] 
The rest of the proof for $H(\hgf,f_0)=O_p(n^{-2/5})$ now follows from an application of Theorem $3.4.1$ and $3.4.4$ of \cite{wc}.
% \todo[inline]{Remove the following.}
% To this end, fixing $\e>0,$ for $f\in \mathcal{P}_{M,0}$  we define the function
%  \[m_f(x)=\log \lb\dfrac{f(x)+f_0(x)}{2 f_0(x)}\rb,\]
%  and we let $\mathcal{M}_{\e}$ denote the class
%   \[\mathcal{M}_{\e}=\{m_f-m_{f_0}\ :\ H(f,f_0)\leq \e\}.\]
%  Also, we denote
%   \[\|\mathbb{G}_n\|_{\mathcal{M}_{\e}}=\sup_{h\in\mathcal{M}_{\e}}\int h d\mathbb{G}_n.\]
% Then  Theorem $3.4.4$ of \cite{wc} yields 
% \[P (m_f-m_{f_0})\lesssim -H^2(f,{f_0}),\]
%  and
% \[E_{P}\|\mathbb{G}_n\|_{\mathcal{M}_{\e}}\lesssim \tilde{J}_{[\ ]}(\e,\mathcal{P}_{M,0},H)\lb 1+\dfrac{\tilde{J}_{[\ ]}(\e,\mathcal{P}_{M,0},H)}{\e^2 \sqn}\rb,\]
% where
% \begin{align*}
% \tilde{J}_{[\ ]}(\e,\mathcal{P}_{M,0},H) =&\ \dint_{0}^{\e}\sqrt{\log N_{[\ ]}(t,\mathcal{P}_{M,0},H)}dt
% \lesssim  \e^{3/4}.
% \end{align*} 
% Therefore, it is easy to see that
% \[E_{P}\|\mathbb{G}_n\|_{\mathcal{M}_{\e}}\lesssim \e^{3/4}+\e^{-1/2}n^{-1/2}=g(\e),\] 
% where
%   $$g(t)=t^{3/4}+t^{-1/2}n^{-1/2},\quad t>0.$$
%   Notice that $g(t)/t^{3/4}$ is decreasing in $t.$  Also, for any constant $c>0,$ we have
%   \[n^{4/5}g(cn^{-2/5})= (c^{3/4}+c^{-1/2})n^{1/2}.\] 
%  Therefore, an application of Theorem $3.4.1$ of \cite{wc} yields
%   \[H(\hgf,{f_0})=O_p(n^{-2/5}),\]
%   which completes the proof of part A of Theorem \ref{MLE: Rate results}.
  
   Now we turn to establishing the rate of convergences of $\hthm$ and $\hgm$. If $x-\th_0$ is a continuity point of $g_0'$, using the fact that $\hthm\as\th_0$, we obtain that 
  \[\dfrac{\sqrt{g_0(x-\hthm)}-\sqrt{\vphantom{\hthm}g_0(x-\th_0)}}{(\hthm-\th_0)}\as \dfrac{g_0'(x-\th_0)}{2\sqrt{g_0(x-\th_0)}}.\]
 Noting $g_0'$ is continuous almost everywhere with respect to Lebesgue measure, and using  Fatou's lemma and part A of the current theorem, we obtain that
  \begin{align}\label{inlemma: inequality: hellinger distance and distance between theta}
  &\liminf_n \dfrac{\edint \lb\sqrt{g_0(x-\hthm)}-\sqrt{\vphantom{\hthm} g_0(x-\th_0)}\rb^2dx}{(\hthm-\th_0)^2}
  \geq  \lb\dfrac{g_0'(x-\th_0)}{2\sqrt{g_0(x-\th_0)}}\rb^2dx=\dfrac{\I}{4}
  \end{align} 
 with probability one. Now observe that
  \begin{align*}
  2H(\hgf,f_0)^2
 = &\ \edint\lb\sqrt{\hgm(x-\hthm)}-\sqrt{\vphantom{\hgm}g_0(x-\th_0)}\rb^2dx\\
 =&\ 2H(\hgm,g_0)^2 +
 \edint\lb\sqrt{g_0(x-\hthm)}-\sqrt{\vphantom{\hthm}g_0(x-\th_0)}\rb^2dx+T_c,
   \end{align*}
   where
   \[T_c=2\edint\lb\sqrt{\hgm(x-\hthm)}-\sqrt{g_0(x-\hthm)}\rb\lb\sqrt{g_0(x-\hthm)}-\sqrt{\vphantom{\hgm}g_0(x-\th_0)}\rb dx.\]
      The inequality in \eqref{inlemma: inequality: hellinger distance and distance between theta} entails that  for all sufficiently large $n$,
   \begin{equation}\label{inlemma: lower bound on hellinger dist between f and hgf}
  2H(\hgf,f_0)^2\nonumber\geq  2H(\hgm,g_0)^2+\dfrac{(\hthm-\th_0)^2\I}{4}-|T_c|\quad a.s.
   \end{equation}
 We aim to show that the cross-term $|T_c|$ is small. In fact, we show that 
 \begin{equation}\label{claim: cross term small for rough rate calculations}
 \dfrac{|T_c|}{|\hthm-\th_0|^2+H(\hgm,g_0)^2}=o_p(1).
 \end{equation}
 Suppose \eqref{claim: cross term small for rough rate calculations} holds. Then  it follows that 
 \begin{align*}
 2H(\hgf,f_0)^2 \geq &\ 2H(\hgm,g_0)^2\\
 &\ +\dfrac{(\hthm-\th_0)^2\I}{4}-o_p(1)H(\hgm,g_0)^2-o_p(1)(\hthm-\th_0)^2,
 \end{align*}
  which completes the proof because $\I>0$. 
  
 Hence, it remains to prove \eqref{claim: cross term small for rough rate calculations}. To this end, notice that $T_c$ can be written as
  \begin{align*}
  T_c= 2\edint\lb\sqrt{\hgm(x)}-\sqrt{\vphantom{\hgm}g_0(x)}\rb\lb\sqrt{\vphantom{\hgm}g_0(x)}-\sqrt{\vphantom{\hgm}g_0(x+\hthm-\th_0)}\rb dx.
  \end{align*}
  Recalling $\d=\th_0-\hthm$, and noting $g_0$ is absolutely continuous because  $f_0\in\mP_0$, we can write
  \begin{align*}
 |T_c|= \MoveEqLeft \bl 2\edint\lb\sqrt{\hgm(x)}-\sqrt{\vphantom{\hgm}g_0(x)}\rb\lb\dint_{-\d}^{0}\dfrac{g_0'(x+t)}{2\sqrt{g_0(x+t)}}dt\rb dx\bl.
  \end{align*}
  Since $g_0\in\mathcal{S}_0$, we have
  \begin{align*}
 |T_c|=&\ 2\bl \dint_{0}^{\infty} \lb\sqrt{\hgm(x)}-\sqrt{\vphantom{\hgm}g_0(x)}\rb\lb\dint_{-\d}^{0}\dfrac{g_0'(x+t)}{2\sqrt{g_0(x+t)}}dt\rb dx\\
&\ -\ \dint_{-\infty}^{0} \lb\sqrt{\hgm(-x)}-\sqrt{\vphantom{\hgm}g_0(-x)}\rb\lb\dint_{-\d}^{0}\dfrac{g_0'(-x-t)}{2\sqrt{g_0(-x-t)}}dt\rb dx\bl\\
=&\ 2\bl \dint_{0}^{\infty} \lb\sqrt{\hgm(x)}-\sqrt{\vphantom{\hgm}g_0(x)}\rb\lb\dint_{-\d}^{0}\dfrac{g_0'(x+t)}{2\sqrt{g_0(x+t)}}dt\rb dx\\
&\ -\dint_{0}^{\infty} \lb\sqrt{\hgm(x)}-\sqrt{\vphantom{\hgm}g_0(x)}\rb\lb\dint_{-\d}^{0}\dfrac{g_0'(x-t)}{2\sqrt{g_0(x-t)}}dt\rb dx\bl,
  \end{align*}
  yielding
  \begin{equation*}
  |T_c|=2\bl \dint_{0}^{\infty} \lb\sqrt{\hgm(x)}-\sqrt{\vphantom{\hgm}g_0(x)}\rb\lb \dint_{-\d}^{0}\lb\dfrac{g_0'(x+t)}{2\sqrt{g_0(x+t)}}-\dfrac{g_0'(x-t)}{2\sqrt{g_0(x-t)}}\rb dt\rb dx\bl.
  \end{equation*}
  Using the Cauchy-Schwarz inequality, we obtain that
  \begin{align*}
\dfrac{|T_c|}{2|\d|}
\leq &\ \lb\dint_{0}^{\infty} \lb\sqrt{\hgm(x)}-\sqrt{\vphantom{\hgm}g_0(x)}\rb^2dx\rb^{1/2}\\
&\ \lb\dint_{0}^{\infty}\lb\dint_{-\d}^{0}\dfrac{1}{|\d|}\lb\dfrac{g_0'(x+t)}{2\sqrt{g_0(x+t)}}-\dfrac{g_0'(x-t)}{2\sqrt{g_0(x-t)}}\rb dt\rb^2 dx\rb^{1/2}.
  \end{align*}
  Since $\sqrt{\hgm(x)}-\sqrt{\vphantom{\hgm}g_0(x)}$ is an even function, the first term on the right hand side of the last inequality is $\sqrt{2}H(\hgm,g_0).$ Hence,
  \begin{align*}
  \dfrac{T_c^2}{8H(\hgm,g_0)^2\d^2}\leq &\ \dint_{0}^{\infty}\lb\dint_{-\d}^{0}\dfrac{1}{|\d|}\lb\dfrac{g_0'(x+t)}{2\sqrt{g_0(x+t)}}-\dfrac{g_0'(x-t)}{2\sqrt{g_0(x-t)}}\rb dt\rb^2 dx,
  \end{align*}
  which, noting 
  \[t\mapsto \dfrac{g_0'(x+t)}{2\sqrt{g_0(x+t)}}-\dfrac{g_0'(x-t)}{2\sqrt{g_0(x-t)}}\]
  is an even function for each $x>0$, can be bounded above by  
 \[  \dint_{0}^{\infty}|\d|\lb\dint_{0}^{|\d|}\dfrac{1}{(\d)^2}\lb\dfrac{g_0'(x+t)}{2\sqrt{g_0(x+t)}}-\dfrac{g_0'(x-t)}{2\sqrt{g_0(x-t)}}\rb ^2 dt\rb dx\]
  using the Cauchy-Schwarz inequality. Therefore, we obtain 
  \begin{align}\label{inlemma: rough rate: intermediate limsup}
   \dfrac{T_c^2}{2H(\hgm,g_0)^2\d^2}\leq&\ \dfrac{1}{|\d|}\dint_{0}^{|\d|}\bigg[ \dint_{0}^{\infty}\lb\dfrac{g_0'(x+t)}{\sqrt{g_0(x+t)}}\rb^2dx\nonumber+\dint_{0}^{\infty}\lb\dfrac{g_0'(x-t)}{\sqrt{g_0(x-t)}}\rb^2dx\\
      &\ -2\dint_{0}^{\infty}\dfrac{g_0'(x-t)}{\sqrt{g_0(x-t)}}\dfrac{g_0'(x+t)}{\sqrt{g_0(x+t)}} dx\bigg ]dt.
  \end{align}
  For $t\geq 0,$ 
   \begin{align*}
  \dint_{0}^{\infty}\lb\dfrac{g_0'(x+t)}{\sqrt{g_0(x+t)}}\rb^2dx
 =  \dint_{t}^{\infty}\lb\dfrac{g_0'(x)}{\sqrt{g_0(x)}}\rb^2dx\leq \dfrac{\I}{2}.
  \end{align*}
  Now observe that for  $z\in(-|\d|,0),$
  \begin{equation}\label{intheorem: bound of ps'}
  |g_0'(z)/\sqrt{g_0(z)}|=|\ps_0'(z)|\sqrt{g_0(z)}\leq |\ps_0'(\d)|\sqrt{g_0(0)}=O_p(1),
  \end{equation}
  since $\ps_0\in\mathcal{SC}_0$, and $\d\as 0$. 
    Hence, for $t\in(0,|\d|)$,
  \begin{align*}
  \dint_{0}^{\infty}\lb\dfrac{g_0'(x-t)}{\sqrt{g_0(x-t)}}\rb^2dx \ =&\ \dint_{-t}^{\infty}\lb\dfrac{g_0'(z)}{\sqrt{g_0(z)}}\rb^2dz \\
 =&\ \dint_{-t}^{0}\lb\dfrac{g_0'(x)}{\sqrt{g_0(z)}}\rb^2dz+\dint_{0}^{\infty}\lb\dfrac{g_0'(z)}{\sqrt{g_0(z)}}\rb^2dz\\
 \leq &\ |\d| \ps_0'(\d)^2\sqrt{g_0(0)}+\I/2\\
 =&\ |\d| O_p(1)+\I/2,
  \end{align*}
  where the last step follows from \eqref{intheorem: bound of ps'}.
  Hence, for any $t\in(0,|\d|),$
  \begin{equation}\label{inlemma: rough rate: sum of sums}
   \dint_{0}^{\infty}\lb\dfrac{g_0'(x+t)}{\sqrt{g_0(x+t)}}\rb^2dx+\dint_{0}^{\infty}\lb\dfrac{g_0'(x-t)}{\sqrt{g_0(x-t)}}\rb^2dx
  = |\d| O_p(1)+\I.
  \end{equation}
 
  Our objective is to apply Fatou's lemma on the  third term on the right hand side of \eqref{inlemma: rough rate: intermediate limsup}. Therefore,  we want to ensure that the integrand is non-negative.
   Note that when $x\geq|\d|$ and $t\in(0, |\d|)$, we have $x>t$, which leads to
  \begin{equation}\label{intheorem: positivity of cross term}
  g_0'(x-t)g_0'(x+t)\geq 0.
  \end{equation}
   Keeping that in mind, we partition the term 
  \begin{align*}
  \MoveEqLeft - \dint_{0}^{\infty}\dfrac{g_0'(x-t)}{\sqrt{g_0(x-t)}}\dfrac{g_0'(x+t)}{\sqrt{g_0(x+t)}} dx\\
  =&\ - \dint_{|\d|}^{\infty}\dfrac{g_0'(x-t)}{\sqrt{g_0(x-t)}}\dfrac{g_0'(x+t)}{\sqrt{g_0(x+t)}} dx -\dint_{0}^{|\d|}\dfrac{g_0'(x-t)}{\sqrt{g_0(x-t)}}\dfrac{g_0'(x+t)}{\sqrt{g_0(x+t)}} dx\\
  \leq &\ - \dint_{|\d|}^{\infty}\dfrac{g_0'(x-t)}{\sqrt{g_0(x-t)}}\dfrac{g_0'(x+t)}{\sqrt{g_0(x+t)}} dx+|\d| O_p(1),
  \end{align*}
  where the last step follows from \eqref{intheorem: bound of ps'}.
The above combined with \eqref{inlemma: rough rate: intermediate limsup} and \eqref{inlemma: rough rate: sum of sums} leads to
  \begin{align}\label{intheorem: the bound on the ratio}
 \MoveEqLeft \limsup_n  \dfrac{T_c^2}{2H(\hgm,g_0)^2\d^2}\nonumber\\
  \leq &\ \limsup_n\dfrac{1}{|\d|}\dint_{0}^{|\d|}\bigg[ |\d| O_p(1)+\I- 2\dint_{|\d|}^{\infty}\dfrac{g_0'(x-t)}{\sqrt{g_0(x-t)}}\dfrac{g_0'(x+t)}{\sqrt{g_0(x+t)}} dx\bigg]dt\nonumber\\
 =  &\ O_p(1)\limsup_n|\d|+\I\nonumber\\
 &\ -2\liminf_n\dfrac{1}{|\d|}\dint_{0}^{|\d|}\dint_{|\d|}^{\infty}\dfrac{g_0'(x-t)}{\sqrt{g_0(x-t)}}\dfrac{g_0'(x+t)}{\sqrt{g_0(x+t)}} dxdt\nonumber\\
 = &\ 0+\I -2\liminf_n \dint_{|\d|}^{\infty}\dfrac{\dint_{0}^{|\d|}\dfrac{g_0'(x+t)}{\sqrt{g_0(x+t)}}\dfrac{g_0'(x-t)}{\sqrt{g_0(x-t)}}dt}{|\d|}dx.
   \end{align}
  Therefore, an application of Fatou's Lemma and \eqref{intheorem: positivity of cross term} yield
  \[\liminf_n \dint_{|\d|}^{\infty}\dfrac{\dint_{0}^{|\d|}\dfrac{g_0'(x+t)}{\sqrt{g_0(x+t)}}\dfrac{g_0'(x-t)}{\sqrt{g_0(x-t)}}dt}{|\d|}dx\geq \dint_{0}^{\infty}\dfrac{g_0'(x)^2}{g_0(x)}dx=\dfrac{\I}{2}.\]
  Thus \eqref{intheorem: the bound on the ratio} leads to
  \[    \dfrac{2T_c^2}{4H(\hgm,g_0)^2\d^2}=o_p(1).\]
  from which it is obvious that
  \[  \dfrac{\sqrt{2}|T_c|}{|\d|^2+H(\hgm,g_0)^2}\leq  \dfrac{\sqrt{2}|T_c|}{2H(\hgm,g_0)|\d|}=o_p(1),\]which proves \eqref{claim: cross term small for rough rate calculations} and thus completes the proof of part B of Theorem \ref{MLE: Rate results}.                                \hfill $\Box$
 
\subsection{Auxilliary lemmas for Theorem~\ref{MLE: Rate results}}
\label{subsec: aux: mle: rate}

 %%%%%%%%%%%%%%%%%%%%%%%%%%%%%%%%%%%%%%%%%%%%%%%%%%%%%%%%%%%%
 \begin{proof}[Proof of Lemma~\ref{lemma: Pal's paper Lemma}]
 The proof exactly follows the proof of Theorem $3.2$ of \cite{exist}.
 Since $\hgf$ is piecewise linear, $\hgf$ attains its maxima at some order statistic, say $X_{(m)}$. If $m>n/2,$ set $m_q=[n/4]$ where $[x]$ is the greatest integer less than or equal to $x.$ For $m\leq n/2,$ we let $m_q=[3n/4]+1.$ Set $K_n=m_q$ or $n-m_q,$ accordingly as $m>n/2$ or $\leq n/2.$ It is easy to see that $n/K_n\to 4$ as $n\to\infty$. Also, 
 \begin{equation}\label{wtf1}
 \hgf(X_{(m)})\leq\dfrac{1}{|X_{(m)}-X_{(m_q)}|}\lb1+\log \dfrac{\hgf(X_{(m)})}{\hgf(X_{(m_q)})}\rb
 \end{equation}
 by Lemma $3$ of \cite{exist} (see our Lemma \ref{Lemma: Pal: 3}). Now since
 \[\sum_{i=1}^n\psi_0(X_i-\th_0)\leq\sum_{i=1}^n\hpm(X_i-\hthm)\leq K_n\log(\hgf(X_{(m_q)}))+(n-K_n)\log (\hgf(X_{(m)})),\]
 \begin{equation}\label{wtf2}
 K_n\log \dfrac{\hgf(X_{(m)})}{\hgf(X_{(m_q)})}\leq n\lb\log\hgf(X_{(m)})-l_n(\th_0,\ps_0)/n \rb.
 \end{equation}
Combining \eqref{wtf1} and \eqref{wtf2} we obtain that
\begin{align*}
\hgf(X_{(m)})\leq&\ \dfrac{1}{|X_{(m)}-X_{(m_q)}|}\lb 1+\dfrac{n}{K_n}\lb\log\hgf(X_{(m)})-l_n(\th_0,\ps_0)/n \rb\rb
\\
=&\ \dfrac{n/K_n}{|X_{(m)}-X_{(m_q)}|}\log\hgf(X_{(m)})+\dfrac{1}{|X_{(m)}-X_{(m_q)}|}\lb1-\dfrac{l_n(\th_0,\ps_0)}{K_n}\rb.
\end{align*}
Therefore by Lemma $4$ of \cite{exist} (see our Lemma \ref{Lemma: Pal: 4}),
\[\hgf(X_{(m)})\leq \dfrac{2n/K_n}{|X_{(m)}-X_{(m_q)}|}\log\lb \dfrac{2n/K_n}{|X_{(m)}-X_{(m_q)}|}\rb+\dfrac{2}{|X_{(m)}-X_{(m_q)}|}\lb1-\dfrac{l_n(\th_0,\ps_0)}{K_n}\rb\]which is finite by our choices of $m$, $m_q$ and $K_n.$
 \end{proof} 
 
 The following lemmas appear in \cite{exist} as Lemma $3$ and $4$ respectively.
  \begin{lemma}\label{Lemma: Pal: 3}
  Suppose $f$ is a log-concave density. If $0<f(x)\leq f(y)$ for $x,y\in\RR,$ then
  \[f(y)\leq \dfrac{1+\log(f(y)/f(x))}{|y-x|}.\]
  \end{lemma} 
   \begin{lemma}\label{Lemma: Pal: 4}
  If   $x,c_1,c_2>0$ and 
  $x\leq c_1\log x+c_2,$
    then
  $x\leq 2c_1\log (2c_1)+2c_2$.
  \end{lemma} 

\section{Technical facts}
\label{sec: technical facts}

Below we  list some facts which have been used repeatedly in our proofs. We begin with a well-known fact on total variation distance.

\begin{fact}\label{fact: dTV and hellinger}
Suppose $F$ and $G$ are two distribution functions with densities $f$ and $g$, respectively. Then 
$d_{TV}(F,G)\leq \sqrt{2}H(f,g).$
\end{fact}

\begin{fact}[Theorem 5.7 (ii) of \cite{shorack2000}]\label{fact: convergence in probability to convergence almost surely}
Suppose  $\{X_n\}_{n\geq 1}$ is a random sequence. If $X_n$ satisfies $X_n\to_p X$ for some random variable $X$, then there exists a  subsequence $n_k$ such that $X_{n_k}\as X$. 
\end{fact}

\begin{fact}[Theorem 5.7 (vii) of \cite{shorack2000}]
\label{fact: Shorack}
Suppose $X_n$ is a sequence of random variables.
Then for some random variable $X$,  $X_n\to_p X$ if and only if every subsequence $\{n_k\}_{k\geq 1}$ contains
a further subsequence $\{n_{r}\}_{r\geq 1}$ for which $X_{n_r}\as X$.
\end{fact}

\begin{fact}[Proposition A.18 of \cite{bobkovbig}]
 \label{fact: bobkov big}
 Suppose the density $f$ is supported on an open  interval (possibly unbounded). Then $F^{-1}$ is strictly increasing, and  $F^{-1}(q_2)-F^{-1}(q_1)=\int_{q_1}^{q_2}dt/f(F^{-1}(t))$ for all $0<q_1<q_2<1$.
 \end{fact}
 
% \begin{fact}[Theorem 1 (iv) of \cite{dumbgen2017}]
% \label{fact: theorem 1.iv of bi-log concave paper}
%   Suppose $f$ is a log-concave density with distribution function $F$. Then 
%   $|f'(x)|\leq f^2(x)/\min(F(x),1-F(x))$.
%   \end{fact}
  
%   The following fact, which gives an upper bound of $g_0'(x)$, is an immediate consequence of Fact~\ref{fact: theorem 1.iv of bi-log concave paper}.
%   \begin{fact}
%   \label{fact: ub of g'}
%   Suppose $f$ is symmetric about $\theta$. Let $\e>0$ be so that $[\theta-\e,\theta+\e]\subset\supp(f)$. Then
%   \[\sup_{x\in[\theta-\e,\theta+\e]}|f'(x)|\leq  f^2(\theta)/F(\theta-\e).\]
%   \end{fact}
  
  \begin{fact}
 \label{fact: empirical process of m-p's}
 Suppose $\mathcal F$ is a class of measurable functions $h$ such that $\int h^2dP_0<\epsilon^2$ where $\|h\|_\infty\leq M$ for some constant $M>0$. Then 
 \[E\|\mathbb{G}_n\|_{\mathcal F}\lesssim J_{[\  ]}(\epsilon,\mathcal F, L_2(P_0))\lb 1+\frac{M J_{[\  ]}(\epsilon,\mathcal F, L_2(P_0))}{\epsilon^2\sqn}\rb\]
 where
 \[J_{[\  ]}(\epsilon,\mathcal F, L_2(P_0))=\dint_0^\epsilon\sqrt{ 1+\log N_{[\  ]}(\epsilon',\mathcal F,L_2(P_0))}d\epsilon'.\]
 \end{fact}
 
 \begin{proof}
 Follows from Lemma 3.4.2, pp. 324 of \cite{vdv}.
 \end{proof}
 
 The next fact is Pratt's lemma \citep[][Theorem 1]{Pratt1960}.
 We state it here for convenience.
 \begin{fact}\label{fact: pratt's lemma}
 Suppose $(\Omega,\mathcal{F},\mu)$ is a measure space and $a_n,b_n,c_n$ are sequences of functions on $\Omega$ converging almost everywhere to functions $a,b,c$ respectively. Also, all functions are integrable and
 $\int a_n d\mu\to \int ad\mu$ and $\int c_n d\mu\to \int cd\mu.$ Moreover, $a_n\leq b_n\leq c_n.$ Then 
 $\int b_n d\mu\to\int b d\mu$.
\end{fact}

 \begin{fact}\label{fact: consistency of the quantiles}
 Suppose $(F_n)_{n\geq 1}$ and $F$ are  distribution functions satisfying $\|F_n-F\|_{\infty}\to 0$. Further suppose $F$ has density $f$ and  $t\in\iint(\supp(f))$. Then \\ $|F_n^{-1}(t)-F^{-1}(t)|\to 0$.
 \end{fact}
 
 \begin{proof}
Since $F^{-1}$ is continuous at $t$, this is essentially Lemma A.5 of \cite{bobkovbig}.
 \end{proof}
 
 The following is a property of integrable functions.
 \begin{fact}[Exercise 16.18, pp. 223 of \cite{billingsley}]
 \label{fact: condition for integrability}
 Suppose $P$ is a finite measure on $\RR$ and $\int_{\RR} |h| dP<\infty$ for some measurable function $h$. Then for each $\epsilon>0$, there exists $\sigma>0$ so that any $P$-measurable set $\mathcal B$ with $P(\mathcal B)<\sigma$ satisfies 
 $\int_{\mathcal B}|h|dP<\epsilon$.
 \end{fact}
 
 The following is a sufficient (and necessary) condition for uniform integrability.
 \begin{fact}[Exercise 16.19, pp. 223 of  \cite{billingsley}]
 \label{fact: condition for UI}
 Suppose $P$ is a finite measure on $\RR$ and $(h_n)_{n\geq 1}$ is a sequence of $P$-measurable functions. Then $(h_n)_{n\geq 1}$ is uniformly integrable if and only if (i) $\sup_{n\geq 1}\int |h_n|dP<\infty$ (ii) given any  $\epsilon>0$, there exists $\sigma>0$ so that any $P$-measurable set $\mathcal B$ with $P(\mathcal B)<\sigma$ satisfies 
 $\sup_{n\geq 1}\int_{\mathcal B}|h_n|dP<\epsilon$.
 \end{fact}
 
 The following fact is a Glivenko-Cantelli type result for a class of functions $\mathcal{F}_n$ changing with $n$.
 \begin{fact}\label{fact: GC}
 Suppose 
 $\mathcal{F}_n$ 
 is a class of functions such that $\sup_{f\in \mathcal{F}_n}\|f\|_{\infty}\leq M_n$. Further suppose for any fixed $\e>0$,
 $M_n^2\sup_{Q}\log N(\e,\mathcal{F}_n,L_2(Q))=o(n)$  where the supremum is over all probability measures on $\RR$. Then
 $E\|\mathbb P_n-P\|_{\mathcal F_n}\to 0$ as $n\to\infty$.
 \end{fact}
 
 \begin{proof}[Proof of Fact~\ref{fact: GC}]
 The proof is similar to the proof of Theorem 2.4.3 of \cite{wc}. Therefore we only highlight the differences. Suppose $X_1,\ldots,X_n\iid P$. Consider also $n$ independent Rademacher random variables $\e_1,\ldots,\e_n$.  Using the symmetrization inequality   \citep[cf. Lemma 2.3.1 of ][  ]{wc} and Fubini's theorem, one can show that
 \[E\|\mathbb P_n-P\|_{\mathcal F_n}\leq 2 E_X\underbrace{E_{\e}\norm{\frac{1}{n}\sum_{i=1}^n\e_i f(X_i)}_{\mathcal F_n}}_{\mathbb Y_n(X)\equiv\mathbb Y_n(X_1,\ldots,X_n)},\]
 where $E_X$ and $E_{\e}$ denote the expectations with respect to $P$ and the law of $\e_1$, respectively. Fixing $\delta>0$, and using the argument in the proof of Theorem 2.4.3 of \cite{wc}, we can show that
 \begin{equation}\label{infact: GC}
     \mathbb Y_n(X)\leq(1+\log N(\delta, \mathcal F_n,L_2(\mathbb F_n))^{1/2}M_n\sqrt{6/n}+\delta
 \end{equation}
 where $\mathbb F_n$ is the empirical distribution function of $X_1,\ldots,X_n$.
 Taking $\delta=1/2$, for sufficiently large $n$, we have
 $\mathbb Y_n(X)\leq 1$
 for any realizations of $X_1,\ldots,X_n$. Therefore $\mathbb Y_n(X)$ is a bounded sequence. For any $\delta>0$ , \eqref{infact: GC} also implies that $\lim_{n\to\infty}\mathbb Y_n(X)\leq \delta$. Since $\delta$ is arbitrary, this implies $\mathbb{Y}_n(X)\to 0$ as $n\to\infty$ for any realization of $X\equiv X_1,\ldots,X_n$.  Therefore, using dominated convergence theorem we conclude that $E_X[\mathbb Y_n(X)]\to_n 0$.
 \end{proof}

 \begin{fact}\label{fact: helli fknot}
 Suppose $f_0$ is a log-concave density with $\I<\infty$. Then $H(f_0(\mathord{\cdot}+y),f_0)=O(|y|)$.
 \end{fact}
 \begin{proof}[Proof of Fact~\ref{fact: helli fknot}]
  Note that
\[H(f_0(\mathord{\cdot}+y),f_0)^2= \edint (\sqrt{f_0(x+y)}-\sqrt{f_0(x)})^2dx\leq \edint \lb\dint_{x-|y|}^{x+|y|}\frac{|f_0'(z)|}{2 \sqrt{f_0(z)}}dz \rb^2dx,\]
which, by the Cauchy-Schwarz inequality, is bounded above by
     \[
    \frac{|y|}{4} \edint \dint_{x-|y|}^{x+|y|}\frac{f_0'(z)^2}{f_0(z)}dz dx
    \stackrel{(a)}{=} \frac{|y|^2}{2}\edint \phi'_0(z)^2f_0(z)dz= |y|^2\I/2,\]
    where (a) follows by Fubini's Theorem.
     Since $\I<\infty$, the above is of order $O(|y|^2)$.
 \end{proof}
 
 \begin{fact}\label{fact: concave-convex}
 Suppose $0\leq u,v\leq 1/2$. Then it holds that
 \[u(1-u)+v(1-v)\geq \min\slb \frac{u+v}{2}, 1-\frac{u+v}{2}\srb.\]
 \end{fact}
 
 \begin{proof}[Proof of Fact \ref{fact: concave-convex}]
 Suppose  $0 \le u, v \le 1/2$. Since $u(1-u) \ge \min \{ u, 1-u \} /2$, it follows that 
\begin{eqnarray*}
u(1-u) + v(1-v) 
& \ge & \frac{1}{2} \min \{ u, 1-u \} + \frac{1}{2} \min \{ v, 1-v \}  \\
& \stackrel{(a)}{\ge} & \min \lb \frac{u}{2}, \frac{1}{4} \rb  +  \min \lb \frac{v}{2}, \frac{1}{4} \rb \\
& = & \frac{u}{2} + \frac{v}{2} = \frac{u+v}{2} 
\end{eqnarray*}
where (a) follows because $1-u, 1-v\geq 1/2$ for $0\leq u,v\leq 1/2$. 
Similarly, for $1/2 \le u,v \le 1$, 
\begin{eqnarray*}
u(1-u) + v(1-v) 
& \ge & \frac{1}{2} \min \{ u, 1-u \} + \frac{1}{2} \min \{ v, 1-v \} \\
& \stackrel{(a)}{\ge} & \min \lb \frac{1}{4}, \frac{1-u }{2} \rb  +  \min \lb \frac{1}{2}, \frac{1-v}{2} \rb \\
& = & \frac{1-u}{2} + \frac{1-v}{2} = 1- \frac{u+v}{2} 
\end{eqnarray*}
where (a) follows because $u,v\geq 1/2$. Hence, the proof follows.
 \end{proof}
 
 \section{ Tuning parameters for \citeauthor{stone} and \citeauthor{beran}'s estimators}
\label{app: stone}
\citeauthor{stone}'s estimator has two tuning parameters $d_n$ and $t_n$. To find the optimal $(d_n,t_n)$ pair, we implement a grid search on a two dimensional grid. Each point on the grid is of the form $(d,t)$ where $d\in\{10,20,30,\ldots, 80\}$, and $t_n\in\{0.10, 0.20,\ldots, 0.60\}$. For each distribution and each sample size, we estimate the efficiency of each  pair using one hundred Monte Carlo samples. The optimal pair is the one which maximizes the estimated efficiency.
  Since \citeauthor{beran}'s estimator also uses two tuning parameters $b_{c,n}$ and $\rho_n$, we repeat the same  procedure for finding the optimal tuning parameters.  The only difference is that in this case,  the scaling parameter is chosen from the grid $\{0.10,0.20,\ldots,1.50\}$, and the number of basis functions is allowed to vary within the set $\{10, 20,\ldots, 50\}$.
 Table~\ref{Table: app: stone} and \ref{Table: app: beran}  tabulate the optimal tuning parameters that we obtained following the above-mentioned procedure.
 
\begin{table*}[ht]
\caption{The optimal $(d_n,t_n)$ pair for \citeauthor{stone}'s estimattor
  }\label{Table: app: stone}
     \begin{tabular}{@{}llllll@{}}
\toprule
 n  &  Gaussian & Laplace & Symmetric beta  & Symmetric beta & Logistic\\ 
  & & & $(r=2.1)$ & $(r=4.5)$ & \\
 \midrule
40 & (10, 0.80) & (20, 0.60) & (20, 0.60) & (40, 0.80) & (10, 0.80)\\
100 & (50, 0.80) & (20, 0.50) & (40, 0.50) & (30, 0.60) & (10, 0.80)\\
200 & (50, 0.80) & (20, 0.50) & (40, 0.50) & (50, 0.60) & (10, 0.80)\\
500 & (60, 0.80) & (10, 0.50) & (20, 0.30) & (30, 0.40) & (30, 0.50)\\
   \bottomrule
 \end{tabular} 
   \end{table*}
   \vspace*{-5mm}
   \begin{table*}[ht]
\caption{The optimal $(b_{c,n},\rho_n)$ pair for \citeauthor{beran}'s estimattor
  }\label{Table: app: beran}
     \begin{tabular}{@{}llllll@{}}
\toprule
 n  &  Gaussian & Laplace & Symmetric beta  & Symmetric beta & Logistic\\ 
  & & & $(r=2.1)$ & $(r=4.5)$ & \\
 \midrule
40 & (10, 1.00) & (40, 0.40) & (10, 0.80) & (40, 1.40) & (10, 1.40)\\
100 & (10, 1.00) & (40, 0.20) & (10, 0.40) & (40, 1.20) & (20, 1.40)\\
200 & (10, 1.00) & (40, 0.20) & (40, 0.60) & (40, 1.00) & (25, 1.00)\\
500 & (10, 0.60) & (40, 0.20) & (40, 0.60) & (35, 0.80) & (30, 1.00)\\
   \bottomrule
 \end{tabular} 
   \end{table*}

As  mentioned previously, we consider another set of tuning parameters for these nonparametric estimators. These tuning prameters, i.e. the non-optimal tuning parameters, are provided in Table~\ref{Table: app: stone worst} and \ref{Table: app: beran worst}.

%\vspace*{-4mm}
\begin{table*}[ht]
\caption{The non-optimal $(d_n,t_n)$ pair for \citeauthor{stone}'s estimattor
  }\label{Table: app: stone worst}
     \begin{tabular}{@{}llllll@{}}
\toprule
 n  &  Gaussian & Laplace & Symmetric beta  & Symmetric beta & Logistic\\ 
  & & & $(r=2.1)$ & $(r=4.5)$ & \\
 \midrule
40 & (30, 0.50) & (50, 0.50) & (40, 0.50) & (50, 0.50) & (50, 0.50)\\
100 & (30, 0.50) & (50, 0.50) & (50, 0.50) & (50, 0.50) & (50, 0.50)\\
200 & (30, 0.50) & (50, 0.50) & (50, 0.50) & (50, 0.50) & (50, 0.50)\\
500 & (30, 0.50) & (50, 0.50) & (40, 0.50) & (50, 0.50) & (50, 0.50)\\
   \bottomrule
 \end{tabular} 
   \end{table*}
   \vspace*{-5mm}
 \begin{table*}[ht]
\caption{The non-optimal $(b_{c,n},\rho_n)$ pair for \citeauthor{beran}'s estimattor
  }\label{Table: app: beran worst}
     \begin{tabular}{@{}llllll@{}}
\toprule
 n  &  Gaussian & Laplace & Symmetric beta  & Symmetric beta & Logistic\\ 
  & & & $(r=2.1)$ & $(r=4.5)$ & \\
 \midrule
40 & (40, 0.20) & (10, 0.40) & (40, 0.20) & (30, 0.20) & (40, 0.20)\\
100 & (40, 0.20) & (10, 1.20) & (40, 0.20) & (35, 0.20) & (40, 0.20)\\
200 & (40, 0.20) & (10, 1.20) & (40, 0.20) & (40, 0.20) & (40, 0.20)\\
500 & (40, 0.20) & (10, 1.20) & (40, 0.20) & (40, 0.20) & (40, 0.20)\\
   \bottomrule
 \end{tabular} 
   \end{table*}

   \newpage
\section{Corrections from the previous arxived version}
\begin{enumerate}
\item Previously, the proof of Fact \ref{fact: f grtr than F} incorrectly stated that the function mapping $(x,y)\mapsto xy$ is convex on the set $\{(x,y): x\geq 0,\ y\geq 0\}$. However, this function is not convex. We have removed the part of the proof that relied on this assertion. Now we use the following algebraic inequality:
\[u(1-u)+v(1-v)\geq \min\left(\frac{u+v}{2},1-\frac{u+v}{2}\right)\quad\text{for}\quad u,v\in[0,1/2]\]
 to prove Fact 4 instead. The above algebraic fact  
 is now given in Fact \ref{fact: concave-convex}. 
\item The term $(\log n)^2$ in the statement of Lemma B.12 was incorrect; it is now corrected to $(\log n)^4$. This correction resulted in changes to the power of the $\log(n)$ terms in the statements of several auxiliary lemmas for proving Theorems \ref{theorem: main: one-step: full} and \ref{theorem: main: one-step: hnss}, namely Lemma B.12, Lemma B.14, Lemma B.16, Lemma D.4, and Lemma D.7. Some minor corrections were also made in the proof of the first step of Theorems \ref{theorem: main: one-step: full} and \ref{theorem: main: one-step: hnss} due to these changes.  In the proof of Theorem \ref{theorem: main: one-step: full}, the $h_n$ defined in \eqref{inthm: 1: def: hn in } now has an $L_2(P_0)$ norm $\|h_n\|_{P_0,2}$ of the order $O_p(n^{-2p/5}(\log n)^{3})$, which previously was of the order $O_p(n^{-2p/5}(\log n)^{3/2})$. Similarly, the corrected order of the $L_2(P_0)$ norm of $h_n$ defined in \eqref{inlemma: def: main: hnss} in the proof of Theorem \ref{theorem: main: one-step: hnss} is $O_p(n^{-2p/5}(\log n)^{3})$. 
However, these minor alterations did not affect the statements of Theorems \ref{theorem: main: one-step: full} and \ref{theorem: main: one-step: hnss}.
\item In the previous version, we let $p\in(0,1)$ in its definition in Condition \ref{cond: hellinger rate}. However, since $O_p(n^{-p})$ is the rate of Hellinger decay for nonparametric estimation of $g_0$ in $\mathcal P_0$, $p$ is expected to be less than  $1/2$.  When $\theta_0$ is known, the conjectured minimax rate of Hellinger error decay for nonparametrically estimating a symmetric log-concave density is $O_p(n^{-2/5})$ \citep{dosssymmetric}. Therefore, in the current version, we set $p\in(0,1/2]$ when we define it in Condition \ref{cond: hellinger rate}.
\item  There was an algebraic mistake in Step 1 of the proof of Theorem \ref{theorem: main: one-step: full} while bounding $E\left[\|\mathbb G_n\|_{\mathcal H_n(C)}\right]$ (page 30). This has now been corrected. To elaborate, previously, we incorrectly deduced $K_n^{-1}M_n^2n^{-1/2}$ to be of the order $O((\log n)^{-3/2}n^{-1/2})$. However, after all corrections, the rate turns out to be 
\[K_n^{-1}M_n^2n^{-1/2}=O((\log n)^{-3}n^{(8p-5)/10}),\]
which goes to zero because $p\leq 1/2$ as per its current definition in Condition \ref{cond: hellinger rate}.  
\item Several typographical mistakes were corrected, including the definition of the Hellinger distance on page 4. 
\item Some references to books or other papers were made more explicit by adding page numbers to them.

\end{enumerate}
\end{document}